\documentclass{article}
\usepackage{amssymb, amsmath, amscd, amsthm, amsfonts, amscd, mathrsfs}
\usepackage{graphics}

\newcommand \del {\partial}

\newcommand \lal {\lambda^{(l)}}

\newcommand \hhl {h^{(l)}}

\newcommand \modk {{\cal M}_{\kappa}}

\newcommand \Oomega {{\overline \Omega}}

\newcommand \oomega {{\overline \omega}}
\newcommand \HH {{\cal H}}

\newcommand \la {\lambda}

\newcommand \R {{\cal R}}
\newcommand \A {{\mathbb A}}

 \newcommand \Prob {\mathbb P}
\newcommand \prob {\mathbb P}

\newcommand \ee {\mathbb E}

\newcommand \U {{\cal U}}

\newcommand \TT {\mathbf T}

\newcommand \Y {{\mathscr  Y}}

\newcommand \BB {\mathfrak B}

\newcommand \B {{\mathfrak B}}
\newcommand \F {{\cal F}}

\newcommand \C {{\mathfrak C}}

\newcommand \oo {{\mathfrak o}}

\newcommand \mm {{\mathfrak m}}
\newcommand \MM {{\mathfrak M}}

\newcommand \VV {{\mathfrak V}}

\newtheorem{theorem}{Theorem}

\newtheorem {lemma} {Lemma}[section]

\newtheorem{proposition}[lemma]{Proposition}

\newtheorem{corollary}[lemma]{Corollary}

\newtheorem{assumption} [lemma]{Assumption}

\title{Limit Theorems for Translation Flows.}
\author{Alexander I. Bufetov \footnote{Rice University, Steklov Institute of Mathematics, the Institute for Information Transmission Problems,
NRU-HSE, the Independent University of Moscow.}}
\date{}
\begin{document}
\maketitle

\begin{abstract}

The aim of this paper is to obtain an asymptotic expansion for ergodic integrals of
translation flows on flat surfaces of higher genus (Theorem \ref{multiplicmoduli})
and to give a limit theorem for these flows (Theorem \ref{limthmmoduli}).

\end{abstract}

\tableofcontents
\section{Introduction.}

\subsection{Outline of the main results.}

A compact Riemann surface endowed with an abelian differential admits two natural flows, called, respectively,
{\it horizontal} and {\it vertical}. One of the main objects of this paper is the space ${\mathfrak B}^+$
of H{\"o}lder cocycles over the vertical flow, invariant under the holonomy by the horizontal flow.
Equivalently, cocycles in ${\mathfrak B}^+$ can be viewed, in the spirit of F. Bonahon \cite{bonahon1},  \cite{bonahon2},
as finitely-additive transverse invariant measures for the horizontal foliation of our abelian differential.
Cocycles in ${\mathfrak B}^+$ are closely connected to the invariant distributions for translation flows in the sense
of G.Forni \cite{F2}.

The space ${\mathfrak B}^+$ is finite-dimensional, and for a generic abelian
differential the
dimension of ${\mathfrak B}^+$ is equal to the genus of the underlying surface.
Theorem \ref{multiplicmoduli}, which extends earlier work of A.Zorich \cite{Z} and G. Forni \cite{F2},
states that the  time integral of  a Lipschitz function under the vertical flow
can be uniformly approximated by a suitably chosen cocycle from ${\mathfrak B}^+$ up to an error that grows
more slowly than any power of time.
The renormalizing action of the Teichm{\"u}ller flow on the space of H{\"o}lder cocycles now allows one
to obtain limit theorems for translation flows on flat surfaces (Theorem \ref{limthmmoduli}).

The statement of Theorem \ref{limthmmoduli} can be informally summarized as follows.
Taking the leading term in the asymptotic expansion of Theorem \ref{multiplicmoduli},
to a generic abelian differential one assigns a compactly supported probability measure on
the space of continuous functions on the unit interval.  The normalized distribution
of the time integral of a Lipschitz function converges, with respect to weak topology, to the trajectory of the
corresponding ``asymptotic distribution'' under the action of the Teichm{\"u}ller flow.
Convergence is exponential with respect to both the L{\'e}vy-Prohorov and the Kantorovich-Rubinstein metric.

The cocycles in ${\mathfrak B}^+$ are constructed explicitly using a symbolic representation for translation flows as suspension flows over
Vershik's automorphisms. Vershik's Theorem \cite{Vershik1} states that every ergodic automorphism of a Lebesgue probability space can be represented as a Vershik's automorphism of a Markov compactum. For interval exchange transformations an explicit representation can be obtained using Rohlin towers
given by Rauzy-Veech induction. Passing to Veech's zippered rectangles and their
bi-infinite Rauzy-Veech expansions, one represents a minimal translation flow as a flow along the leaves of the asymptotic foliation of a bi-infinite Markov compactum. In this representation, cocycles in $\B^+$ become finitely-invariant measures on
the asymptotic foliations of a Markov compactum. For previous attemps in this direction, see \cite{ito}, \cite{dumontkamae}, \cite{kamae}.

Thus, after passage to a finite cover (namely, the Veech space of zippered rectangles),
the moduli space of abelian differentials is represented as a space of Markov compacta.
The Teichm{\"u}ller flow and the Kontsevich-Zorich cocycle admit a
simple description in terms of this symbolic representation, and the cocycles
in $\B^+$ are constructed explicitly.
Theorems  \ref{multiplicmoduli}, \ref{limthmmoduli} are then derived from their
symbolic counterparts, Corollary \ref{symbmultiplic} and Theorem \ref{limthmmarkcomp}.

\subsection{H{\"o}lder cocycles over translation flows.}
Let $\rho\geq 2$ be an integer, let $M$ be a compact orientable surface of genus $\rho$,
and let $\omega$ be a holomorphic one-form on $M$.
Denote by
$
{\bf  m}=i(\omega\wedge {\overline \omega})/2
$
the area form induced by $\omega$ and assume that ${\bf m}(M)=1$.

Let $h_t^+$ be the {\it vertical} flow on $M$ (i.e., the flow corresponding to $\Re(\omega)$);
let $h_t^-$ be the {\it horizontal} flow on $M$ (i.e., the flow corresponding to $\Im(\omega)$).
The flows $h_t^+$, $h_t^-$ preserve the area ${\bf m}$.

Take $x\in M$, $t_1, t_2\in {\mathbb R}_+$ and assume that the closure of the set
\begin{equation}
\label{admrectpsan}
\{h^+_{\tau_1} h^{-}_{\tau_2}x, 0\leq \tau_1< t_1, 0\leq \tau_2< t_2\}
\end{equation}
does not contain zeros of the form $\omega$.  The set (\ref{admrectpsan})
is then called {\it an admissible rectangle} and denoted $\Pi(x, t_1, t_2)$.
Let ${\overline {\mathfrak C}}$ be the semi-ring of admissible rectangles.

Consider the linear space ${\mathfrak B}^+$ of H{\"o}lder cocyles $\Phi^+(x,t)$ over the vertical
flow $h_t^+$ which are invariant under horizontal holonomy. More precisely, a function
$\Phi^+(x,t): M\times {\mathbb R}\to {\mathbb R}$ belongs to the space ${\mathfrak B}^+$
if it satisfies:
\begin{assumption}
\label{bplusx}
\begin{enumerate}
\item  $\Phi^+(x,t+s)=\Phi^+(x,t)+\Phi^+(h_t^+x, s)$;
\item  There exists $t_0>0$, $\theta>0$ such that
$|\Phi^+(x,t)|\leq t^{\theta}$ for all $x\in M$ and all $t\in {\mathbb R}$ satisfying $|t|<t_0$;
\item  If $\Pi(x, t_1, t_2)$ is an admissible rectangle, then
$\Phi^+(x, t_1)=\Phi^+(h_{t_2}^-x, t_1)$.
\end{enumerate}
\end{assumption}
For example, a cocycle $\Phi_1^+$  defined by $\Phi_1^+(x,t)=t$  belongs to ${\BB}^+$.

In the same way  define the space ${\mathfrak B}^-$ of H{\"o}lder cocyles $\Phi^-(x,t)$ over
the horizontal flow $h_t^-$  which are invariant under vertical holonomy, and set $\Phi_1^-(x,t)=t$.

Given $\Phi^+\in {\mathfrak B}^+$, $\Phi^-\in {\mathfrak B}^-$, a finitely additive
measure $\Phi^+\times \Phi^-$ on the semi-ring ${\overline {\mathfrak C}}$ of admissible rectangles
is introduced by the formula
\begin{equation}
\Phi^+\times \Phi^-(\Pi(x, t_1, t_2))=\Phi^+(x,t_1)\cdot \Phi^-(x, t_2).
\end{equation}

In particular, for $\Phi^-\in {\mathfrak B }^-$, set $m_{\Phi^-}=\Phi_1^+\times\Phi^-$:
\begin{equation}
\label{mphipsan}
m_{\Phi^-}(\Pi(x, t_1, t_2))=t_1\Phi^-(x, t_2).
\end{equation}
For any $\Phi^-\in {\mathfrak B}^-$ the measure
$m_{\Phi^-}$ satisfies $(h_t^+)_*m_{\Phi^-}=m_{\Phi^-}$ and is an invariant distribution in the
sense of G.~Forni \cite{F1}, \cite{F2}. For instance, $m_{\Phi_1^-}={\bf m}$.

An ${\mathbb R}$-linear pairing between ${\mathfrak B}^+$ and ${\mathfrak B^-}$ is given, for
$\Phi^+\in {\mathfrak B}^+$, $\Phi^-\in {\mathfrak B}^-$,
by the formula
\begin{equation}
\label{psanpairing}
\langle \Phi^+, \Phi^-\rangle=\Phi^+\times \Phi^-(M).
\end{equation}

\subsection{Characterization of cocycles.}

Let $\B^+_c({\bf X})$ be the space  of continuous holonomy-invariant cocycles:
more precisely, a function
$\Phi^+(x,t): M\times {\mathbb R}\to {\mathbb R}$ belongs to the space ${\mathfrak B}_c^+({\bf X})$
if it satisfies conditions 1 and 3 of Assumption \ref{bplusx}, while condition 2
is replaced by the following weaker version:

For any $\varepsilon>0$ there exists $\delta>0$  such that
$|\Phi^+(x,t)|\leq \varepsilon$ for all $x\in M$ and all $t\in {\mathbb R}$ satisfying $|t|<\delta$.

Given an abelian differential ${\bf X}=(M, \omega)$, we now construct, following Katok \cite{katok},
an explicit  mapping of $\BB_c^+(M, \omega)$ to $H^1(M, {\mathbb R})$.

A continuous closed curve $\gamma$ on $M$ is called {\it rectangular} if
$$
\gamma=\gamma_1^+\sqcup\dots \sqcup \gamma_{k_1}^+\bigsqcup \gamma_1^-\sqcup\dots \sqcup \gamma_{k_2}^-,
$$

where $\gamma_i^+$ are arcs of the flow $h_t^+$, $\gamma_i^-$ are arcs of the flow $h_t^-$.

For $\Phi^+\in\BB_c^+$ define
$$
\Phi^+(\gamma)=\sum_{i=1}^{k_1} \Phi^+(\gamma_i^+);
$$
similarly, for $\Phi^-\in\BB_c^-$ write
$$
\Phi^-(\gamma)=\sum_{i=1}^{k_2} \Phi^-(\gamma_i^-).
$$

Thus, a cocycle $\Phi^+\in\BB_c$ assigns a number $\Phi^+(\gamma)$ to every closed rectangular curve $\gamma$.
It is shown in Proposition \ref{acthomology} below that if $\gamma$ is homologous to $\gamma^{\prime}$, then
$\Phi^+(\gamma)=\Phi^+(\gamma^{\prime})$. For an abelian differential ${\bf X}=(M, \omega)$,
we thus obtain  maps
\begin{equation}
\label{maptocohomology}
{\check {\cal I}}_{\bf X}^+: \B_c^+({\bf X})\to H^1(M, {\mathbb R}), \ {\check {\cal I}}_{\bf X}^-: \B_c^-({\bf X})\to H^1(M, {\mathbb R}).
\end{equation}

For a generic abelian differential, the image of $\B^+$ under the map ${\check {\cal I}}_{\bf X}^+$ is the strictly unstable space of
the Kontsevich-Zorich cocycle over the Teichm{\"u}ller flow.

More precisely, let $\kappa=(\kappa_1, \dots, \kappa_{\sigma})$ be
a nonnegative integer vector such that  $\kappa_1+\dots+\kappa_{\sigma}=2\rho-2$.
Denote by $\modk$ the moduli space of pairs $(M, \omega)$, where $M$ is
a Riemann surface of genus $\rho$ and $\omega$ is a holomorphic differential
of area $1$ with singularities of orders $k_1, \dots, k_{\sigma}$.
The space $\modk$ is often called the {\it stratum} in the moduli space of abelian differentials.

The Teichm{\"u}ller flow ${\bf g}_s$ on $\modk$ sends the
modulus of a pair $(M, \omega)$ to the modulus of the pair
$(M, \omega^{\prime})$, where $\omega^{\prime}=e^s\Re(\omega)+ie^{-s}\Im(\omega)$;
the new complex structure on $M$ is uniquely determined by the requirement that the form $\omega^{\prime}$
 be holomorphic.
As shown by Veech, the space $\modk$ need not be connected;
let $\HH$ be a connected component of $\modk$.

Let ${\mathbb H}^1(\HH)$ be the fibre bundle over $\HH$ whose fibre at a point $(M, \omega)$ is the cohomology group
$H^1(M, {\mathbb R})$. The bundle ${\mathbb H}^1(\HH)$ carries the {\it Gauss-Manin connection} which declares continuous integer-valued sections of our bundle to be flat and is uniquely defined by that requirement. Parallel transport with respect to the Gauss-Manin connection
along the orbits of the Teichm{\"u}ller flow yields a cocycle over the Teichm{\"u}ller flow, called the {\it Kontsevich-Zorich cocycle}
and denoted  ${\mathbf A}={\mathbf A}_{KZ}$.

Let $\Prob$ be a ${\bf g}_s$-invariant ergodic probability
measure on $\HH$. For ${\bf X}\in\HH$, ${\bf X}=(M, \omega)$, let ${\mathfrak B}_{\bf X}^+$, ${\mathfrak B}_{\bf X}^-$ be the
corresponding spaces of H{\"o}lder cocycles.

Denote by $E_{\bf X}^u\subset H^1(M, {\mathbb R})$
the space spanned by vectors corresponding to the positive Lyapunov exponents of the Kontsevich-Zorich
cocycle, by $E_{\bf X}^s\subset H^1(M, {\mathbb R})$ the space spanned by vectors corresponding to
the negative exponents of the Kontsevich-Zorich cocycle.

\begin{proposition}
For $\Prob$-almost all ${\bf X}\in\HH$ the map
${\check {\cal I}}_{\bf X}^+$ takes $\BB_{\bf X}^+$ isomorphically onto $E_{\bf X}^u$, the map
${\check {\cal I}}_{\bf X}^-$ takes $\BB_{\bf X}^-$ isomorphically onto $E_{\bf X}^s$.

The pairing $\langle, \rangle$ is nondegenerate and is taken by the isomorphisms ${\cal I}_{\bf X}^+$, ${\cal I}_{\bf X}^-$ to the cup-product in the cohomology $H^1(M, {\mathbb R})$.
\end{proposition}

{\bf Remark.} In particular, if $\Prob$ is the Masur-Veech ``smooth" measure \cite{masur, veech}, then
$\dim {\BB}_{\bf X}^+=\dim {\BB}_{\bf X}^-=\rho$.

{\bf Remark.} The isomorphisms ${\check {\cal I}}_{\bf X}^+$, ${\check {\cal I}}_{\bf X}^-$
are analogues of G. Forni's isomorphism \cite{F2} between his space of invariant distributions and the
unstable space of the Kontsevich-Zorich cocycle.

Now recall that to every cocycle $\Phi^-\in \BB_{\bf X}^-$ we have assigned a
finitely-additive H{\"o}lder measure $m_{\Phi^-}$ invariant under the flow $h_t^+$.
Considering hese measures  as distributions in the sense of Sobolev and Schwartz,
we arrive at the following proposition.

\begin{proposition}
\label{forni-classif} Let $\mathbb{P}$ be an ergodic
$g_s-invariant$ probability measure on $\mathcal{H}.$ The for
$\mathbb{P}-almost$ every abelian differential $(M,\omega)$ the
space $\{m_{\Phi^-},\Phi^-\in\mathfrak{B}^-(M,\omega)\}$ coincides
with the space of $h^{+}_{t}-invariant$ distributions belonging to
the Sobolev space $H^{-1}.$
\end{proposition}

\begin{proof} By definition for any
$\Phi^+\in\mathfrak{B}^+$ the distribution $m_{\Phi^+}$ is
$h^-_t$-invariant and belongs to the Sobolev space $H^{-1}$.
G.Forni has shown that for any $g_s$-invariant ergodic measure
$\mathbb{P}$ and $\mathbb{P}$-almost every abelian differential
$(M,\omega),$ the dimension of the space of $h_t^-$-invariant
distributions belonging to the Sobolev space $H^{-1}$ ${\it does\
not\ exceed}$ the dimension of the strictly expanding Oseledets
subspace of the Kontsevich-Zorich cocycle (under mild additional
assumption on the measure $\mathbb{P}$ G.Forni proved that these
dimensions are in fact equal, see Theorem 8.3  and Corollary
$8.3^{\prime}$ in \cite{F2}; note, however, that the proof of
the {\it upper} bound in Forni's Theorem only uses ergodicity of the measure).
Since the dimension of the space
$\{m_{\Phi^-},\Phi^-\in\mathfrak{B}^-\}$ equals that of the strictly expanding space for the
Kontsevich-Zorich cocycle for $\mathbb{P}$-almost all
$(M,\omega),$ the proposition is proved completely.
\end{proof}

Consider the inverse isomorphisms
$$
{\cal I}_{\bf X}^+=\left({\check {\cal I}}_{\bf X}^+\right)^{-1}; \ {\cal I}_{\bf X}^-=\left({\check {\cal I}}_{\bf X}^-\right)^{-1}.
$$
Let $1=\theta_1>\theta_2>\dots>\theta_{l}>0$
be the distinct positive Lyapunov exponents of the Kontsevich-Zorich cocycle ${\bf A}_{KZ}$, and let
$$
E_{\bf X}^u=\bigoplus\limits_{i=1}^l E_{{\bf X}, \theta_i}^u
$$
be the corresponding Oseledets decomposition at ${\bf X}$.
\begin{proposition}
\label{hoeldergrowth}
Let
$v\in E_{{\bf X}, \theta_i}^u$, $v\neq 0$, and denote $\Phi^+={\cal I}_{\bf X}^+(v)$.
Then for any $\varepsilon>0$ the cocycle $\Phi^+$ satisfies the
H{\"o}lder condition with exponent $\theta_i-\varepsilon$ and for any $x\in M({\bf X})$ such that
$h_t^+x$ is defined for all $t\in {\mathbb R}$
we have
$$
\limsup\limits_{T\to\infty} \frac{\log|\Phi^+(x,T)|}{\log T}=\theta_i; \  \limsup\limits_{T\to 0} \frac{\log|\Phi^+(x,T)|}{\log T}=\theta_i.
$$
\end{proposition}

\begin{proposition}
\label{cochyperb}
If the Kontsevich-Zorich cocycle does not have zero Lyapunov exponent with respect
to $\Prob$, then $\B^+_c({\bf X})=\B^+({\bf X})$.
\end{proposition}
{\bf Remark.} The condition of the absence of zero Lyapunov exponents can be weakened:
it suffices to require that the Kontsevich-Zorich cocycle act isometrically on the
neutral Oseledets subspace corresponding to the Lyapunov exponent zero.
Isometric action means here that there exists an inner product which
depends measurably on the point in the stratum and which is invariant under the
Kontsevich-Zorich cocycle.

{\bf Question.} Does  there exist a ${\bf g}_s$-invariant ergodic probability measure $\Prob^{\prime}$
 on $\HH$ such that the inclusion $\BB^+\subset \BB^+_c$ is proper almost surely with respect to $\Prob^{\prime}$?

{\bf Remark.} G.Forni has made the following
remark. To a cocycle $\Phi^+\in\mathfrak{B}^+$ assign a $1$-current
$\beta_{\Phi^+},$ defined, for a smooth 1-form $\eta$ on the
surface $M,$ by the formula
$$\beta_{\Phi^+}(\eta)=\int_M\Phi^+\wedge\eta,$$ where the
integral in the right hand side is defined as the limit of Riemann
sums. The resulting current $\beta_{\Phi^+}$ is a ${\it basic\
current}$ for the horizontal foliation.

The mapping of H\"{o}lder cocycles into the cohomology
$H^1(M,\mathbb{R})$ of the surface corresponds to G. Forni's map
that to each basic current assign it's cohomology class (the
latter is well-defined by the de Rham Theorem). In particular, it
follows that for any ergodic $g_s$-invariant probability measure
$\mathbb{P}$ on $\mathcal{H}$ and  $\mathbb{P}$-almost every
abelian differential $(M,\omega)$ every basic current from the
Sobolev space $H^{-1}$ is induced by a H\"{o}lder cocycle
$\Phi^+\in\mathfrak{B}^+(M,\omega).$

\subsection{Approximation of weakly Lipschitz functions.}
\subsubsection{The space of weakly Lipschitz functions.}

The space of Lipschitz functions is not invariant under $h_t^+$, and a larger
function space $Lip_w^+(M, \omega)$ of weakly Lipschitz functions is introduced as follows.
A bounded measurable function $f$ belongs to $Lip_w^+(M, \omega)$ if there exists a
constant $C$, depending only on $f$, such that for any admissible rectangle $\Pi(x, t_1, t_2)$ we
have
\begin{equation}
\label{weaklippsan}
\left| \int_0^{t_1} f(h_t^+x)dt -\int_0^{t_1}  f(h_t^+(h^-_{t_2}x)dt \right|\leq C.
\end{equation}
Let $C_f$ be the infimum of all $C$ satisfying (\ref{weaklippsan}). We norm $Lip_w^+(M, \omega)$ by setting
$$
||f||_{Lip_w^+}=\sup_M f+C_f.
$$

By definition, the space  $Lip_w^+(M, \omega)$ contains all Lipschitz functions on $M$ and is
invariant under $h_t^+$. If $\Pi$ is an admissible rectangle, then its characteristic function
$\chi_{\Pi}$  is weakly Lipschitz (I am grateful to C. Ulcigrai for this remark).

We denote by $Lip_{w,0}^+(M, \omega)$ the
subspace of  $Lip_w^+(M, \omega)$ of
functions whose integral with respect to ${\bf m}$ is $0$.

For any $f\in Lip_{w}^+(M, \omega)$ and any $\Phi^-\in \BB^-$ the
integral $\int_M f dm_{\Phi^-}$ can be defined as the limit of Riemann sums.

\subsubsection{The cocycle corresponding to a weakly Lipschitz function.}
If the pairing $\langle, \rangle$ induces an isomorphism between $\BB^+$ and the
dual $(\BB^-)^*$, then one can assign to a function $f\in Lip_{w}^+(M, \omega)$ the functional $\Phi_f^+$ by
the formula
\begin{equation}
\label{defphifteich}
\langle \Phi_f^+, \Phi^-\rangle =\int\limits_M f dm_{\Phi^-}, \Phi^-\in \BB^-.
\end{equation}

By definition,  $\Phi^+_{f\circ h_t^+}=\Phi_f^+$.
We are now ready to formulate our first main result, the Approximation Theorem \ref{multiplicmoduli}.
\begin{theorem}
\label{multiplicmoduli}
Let $\Prob$ be an ergodic probability ${\bf g}_s$-invariant measure on $\HH$.
For any $\varepsilon>0$ there exists
a constant $C_{\varepsilon}$ depending only on $\Prob$ such that for
$\Prob$-almost every ${\bf X}\in\HH$, any $f\in Lip_w^+({\bf X})$, any $x\in M$ and any $T>0$ we have
$$
\left| \int_0^T
f\circ h_t^+(x) dt -\Phi^+_f(x,T)\right|\leq
C_{\varepsilon}||f||_{Lip_w^+}(1+T^{\varepsilon}).
$$
\end{theorem}

\subsubsection{Invariant measures with simple Lyapunov spectrum.}

Consider the case in which the
Lyapunov spectrum of
the Kontsevich-Zorich cocycle is simple in restriction
to the space $E^u$ (as, by the Avila-Viana theorem \cite{AV},
is the case with the Masur-Veech smooth measure).
Let  $l_0={\rm dim} E^u$ and let
\begin{equation}
1=\theta_1>\theta_2>\dots>\theta_{l_0}
\end{equation}
be the corresponding simple expanding Lyapunov exponents.

Let $\Phi_1^+$ be given by the formula $\Phi_1^+(x,t)=t$
and introduce a basis
\begin{equation}
\Phi_1^+, \Phi_2^+, \dots, \Phi_{l_0}^+
\end{equation}
in ${\mathfrak B}^+_{\bf X}$ in such a
way that ${\check {\cal I}}_{\bf X}^+(\Phi_i^+)$ lies in the Lyapunov subspace with exponent $\theta_i$.
By Proposition \ref{hoeldergrowth}, for any $\varepsilon>0$ the cocycle $\Phi_i^+$ satisfies the H{\"o}lder condition with
exponent $\theta_i-\varepsilon$, and for any $x\in M({\bf X})$ we have
$$
\limsup\limits_{T\to\infty} \frac{\log|\Phi_i^+(x,T)|}{\log T}=\theta_i; \ \limsup\limits_{T\to 0} \frac{\log|\Phi_i^+(x,T)|}{\log T}=\theta_i.
$$

Let $\Phi_1^-, \dots, \Phi_{l_0}^-$ be the dual basis in ${\mathfrak  B}^-_{\bf X}$.
Clearly, $\Phi_1^-(x,t)=t$.

By definition, we have
\begin{equation}
\Phi_f^+=\sum_{i=1}^{l_0} m_{\Phi_i^-}(f)\Phi_i^+.
\end{equation}

Noting that by definition we have
$$m_{\Phi_1^-}={\bf m},$$

we derive from
Theorem \ref{multiplicmoduli} the following corollary.

\begin{corollary}

Let $\Prob$ be an invariant ergodic probability measure for the Teichm{\"u}ller
flow such that with respect to $\Prob$ the
Lyapunov spectrum of
the Kontsevich-Zorich cocycle is simple in restriction to its strictly
expanding subspace.

Then for any $\varepsilon>0$ there exists
a constant $C_{\varepsilon}$ depending
only on $\Prob$ such that for
$\Prob$-almost every ${\bf X}\in\HH$, any
$f\in Lip_w^+({\bf X})$, any $x\in {\bf X}$
and any $T>0$ we have
$$
\left| \int_0^T f\circ h_t^+(x) dt - T(\int_M fd{\bf m})-
\sum\limits_{i=2}^{l_0} m_{\Phi_i^-}(f)\Phi_i^+(x,T)
\right|\leq C_{\varepsilon}||f||_{Lip_w^+}(1+T^{\varepsilon}).
$$
\end{corollary}
For horocycle flows a related asymptotic expansion has been obtained by Flaminio and Forni \cite{flafo}.

{\bf Remark.}
If $\Prob$ is the Masur-Veech smooth measure on $\HH$, then
it follows from the work of G.Forni \cite{F1}, \cite{F2}, \cite{forniarxiv}
and S. Marmi, P. Moussa, J.-C. Yoccoz \cite{MMY}
that the left-hand side is bounded for any $f\in C^{1+\varepsilon}(M)$
(in fact, for any $f$ in the Sobolev space $H^{1+\varepsilon}$).
In particular, if $f\in C^{1+\varepsilon}(M)$ and $\Phi_f^+=0$,
then $f$ is a coboundary.

\subsection{Holonomy invariant transverse finitely-additive measures for oriented measured foliations.}

Holonomy-invariant cocycles assigned to an abelian differential can be interpreted
as transverse invariant measures for its foliations in the spirit of Bonahon
\cite{bonahon1}, \cite{bonahon2}.

Let $M$ be a compact oriented surface of genus at least two, and
let $\F$ be a minimal oriented measured foliation on $M$.
Denote by ${\bf m}_{\F}$ the transverse invariant measure of $\F$.
If $\gamma=\gamma(t), t\in [0,T]$ is a  smooth curve on $M$, and $s_1, s_2$ satisfy
$0\leq s_1< s_2\leq T$,  then we denote by ${\rm res}_{[s_1, s_2]}\gamma$ the curve
$\gamma(t), t\in [s_1, s_2]$.

Let  $\B_c(\F)$ be the space of uniformly continuous finitely-additive transverse invariant
measures for $\F$. In other words,  a map $\Phi$ which to every
smooth arc $\gamma$ transverse to $\F$ assigns a real number $\Phi(\gamma)$
belongs to the space ${\B}_c(\F)$
if it satisfies the following:
\begin{assumption}
\label{finaddmeasfoliation}
\begin{enumerate}
\item ( finite additivity) For $\gamma=\gamma(t), t\in [0,T]$ and any $s\in (0,T)$, we have
$$
\Phi(\gamma)=\Phi({\rm res}_{[0, s]}\gamma)+\Phi({\rm res}_{[s, T]}\gamma);
$$
\item ( uniform continuity) for any $\varepsilon>0$ there exists $\delta>0$ such that
for any transverse arc $\gamma$ satisfying ${\bf m}_{\F}(\gamma)<\delta$ we have $|\Phi(\gamma)|<\varepsilon$;
\item ( holonomy invariance)  the value $\Phi(\gamma)$ does not change if $\gamma$ is deformed in
such a way that it stays transverse to $\F$ while the endpoints of $\gamma$ stay on their respective leaves.
\end{enumerate}
\end{assumption}
A measure $\Phi\in \B_c(\F)$ is called H{\"o}lder with exponent $\theta$
if there exists $\varepsilon_0>0$ such that for any transverse arc $\gamma$
satisfying ${\bf m}_{\F}(\gamma)<\varepsilon_0$ we have
$$
|\Phi(\gamma)|\leq \left({\bf m}_{\F}(\gamma)\right)^{\theta}.
$$

Let $\B(\F)\subset \B_c(\F)$ be the subspace of H{\"o}lder transverse measures.

As before, we have a natural map
$$
{\cal I}_{\F}: \B_c(\F)\to H^1(M, {\mathbb R})
$$
defined as follows.  For a smooth closed curve $\gamma$ on
$M$, and a measure $\Phi\in \B_c(\F)$  the integral
$\int_{\gamma} d\Phi$ is well-defined as the limit of Riemann sums; by holonomy-invariance and continuity of $\Phi$, this operation
descends to homology and assigns to $\Phi$ an element of $H^1(M, {\mathbb R})$.

Now take an abelian differential ${\bf X}=(M, \omega)$ and let $\F^-_{\bf X}$ be its horizontal foliation.
We have a ``tautological'' isomorphism between $\BB_c(\F^-_{\bf X})$ and $\BB_c^+({\bf X})$: every
transverse measure for the horizontal foliation induces a cocycle for the vertical foliation and
vice versa; to a H{\"o}lder measure corresponds a H{\"o}lder cocycle.
For brevity, write ${\cal I}_{\bf X}={\cal I}_{\F^-_{\bf X}}$.
Denote by $E^u_{\bf X}\subset H^1(M, {\mathbb R})$ the unstable subspace of the
Kontsevich-Zorich cocycle of the abelian differential ${\bf X}=(M, \omega)$.

Theorem \ref{multiplicmoduli} and Proposition \ref{cochyperb} yield the following
\begin{corollary}
\label{finaddmeasfol}
Let $\Prob$ be an ergodic probability measure for the Teichm{\"u}ller flow ${\bf g}_t$ on $\HH$.
Then for almost every abelian differential ${\bf X}\in\HH$ the map ${\cal I}_{\bf X}$ takes $\BB(\F^-_{\bf X})$ isomorphically
onto $E^u_{\bf X}$.

If all the Lyapunov exponents of the Kontsevich-Zorich cocycle are nonzero with respect to $\Prob$,
then for almost all ${\bf X}\in\HH$ we have $\B_c(\F_{\bf X})=\B(\F_{\bf X})$.
\end{corollary}

In other words, in the absence of zero Lyapunov exponents
all continuous transverse finitely-additive invariant measures are in fact H{\"o}lder.

{\bf Remark.} As before, the condition of the absence of zero Lyapunov exponents can be weakened:
it suffices to require that the Kontsevich-Zorich cocycle act isometrically on the Oseledets
subspace corresponding to the Lyapunov exponent zero.

By definition, the space $\BB(\F^-_{\bf X})$ only depends on the horizontal foliation
of our abelian differential; so does $E^u_{\bf X}$.

\subsection{Finitely-additive invariant measures for interval exchange transformations.}
\subsubsection{The space of invariant continuous finitely-additive measures.}

Let $m\in {\mathbb N}$. Let $\Delta_{m-1}$ be the standard unit simplex
$$
\Delta_{m-1}=\{\la\in {\mathbb R}^m_+,\ \la=(\la_1, \dots,\la_m),\la_i>0, \sum\limits_{i=1}^m \la_i=1\}.
$$

Let $\pi$ be a permutation of $\{1, \dots, m\}$ satisfying the {\it irreducibility} condition:
we have $\pi\{1, \dots, k\}=\{1, \dots, k\}$ if and only if $k=m$.

On the half-open interval $I=[0, 1)$ consider
the points $\beta_1=0$, $\beta_i=\sum_{j<i}\la_j$, $\beta_1^{\pi}=0$,
$\beta_i^{\pi}=\sum_{j<i}\la_{\pi^{-1}j}$ and
denote $I_i=[\beta_i, \beta_{i+1})$, $I_i^{\pi}=[\beta_i^{\pi}, \beta_{i+1}^{\pi})$.
The length of $I_i$ is $\la_i$, while the length of $I_i^{\pi}$  is $\la_{\pi^{-1}i}$.
Set
$$
{\bf T}_{(\la,\pi)}(x)=x+\beta_{\pi i}^{\pi}-\beta_i {\rm \ for\ } x\in I_i.
$$
The map ${\bf T}_{(\la,\pi)}$ is called {\it an interval exchange transformation} corresponding to $(\la, \pi)$.
By definition, the map ${\bf T}_{(\la,\pi)}$ is invertible and preserves the Lebesgue measure on $I$.
By the theorem of Masur \cite{masur} and Veech \cite{veech}, for any irreducible permutation $\pi$ and
for Lebesgue-almost all $\la\in\Delta_{m-1}$, the corresponding interval exchange transformation ${\bf T}_{(\la,\pi)}$
is uniquely ergodic: the Lebesgue measure is the only invariant probability measure for ${\bf T}_{(\la,\pi)}$.

Consider the space of complex-valued continuous finitely-additive invariant measures for ${\bf T}_{(\la,\pi)}$.

More precisely, let $\BB_c({\bf T}_{(\la,\pi)})$ be the space of all continuous functions $\Phi: [0,1]\to {\mathbb R}$
satisfying
\begin{enumerate}
\item  $\Phi(0)=0$;
\item if $0\leq t_1\leq t_2<1$ and ${\bf T}_{(\la,\pi)}$ is continuous on $[t_1, t_2]$, then
$\Phi(t_1)-\Phi(t_2)=\Phi({\bf T}_{(\la,\pi)}(t_1))-\Phi({\bf T}_{(\la,\pi)}(t_2))$.
\end{enumerate}

Each function $\Phi$ induces a finitely-additive measure on $[0,1]$ defined on the semi-ring
of subintervals (for instance, the function $\Phi_1(t)=t$ yields the Lebesgue measure on $[0,1]$).

Let $\BB({\bf T}_{(\la,\pi)})$ be the subspace of H{\"o}lder functions $\Phi\in \BB_c({\bf T}_{(\la,\pi)})$.

The classification of H{\"o}lder cocycles over translation flows and the asymptotic formula of
Theorem \ref{multiplicmoduli} now
yield the classification of the space $\BB({\bf T}_{(\la,\pi)})$ and
an asymptotic expansion for time averages of almost all interval exchange maps.

\subsubsection{The approximation of ergodic sums}

Let ${\bf X} = (M,\omega)$ be an abelian differential, and let $I\subset M$ be a closed interval lying on a leaf of a horizontal foliation. The vertical flow $h^+_t$ induces an interval exchange map $\mathbf{T}_I$ on $I$, namely, the Poincar\'{e} first return map of the flow. By definition, again there is a natural tautological identification of the spaces $\mathfrak{B}_c(\mathbf{T}_I)$ and $\mathfrak{B}_c^-(\mathbf{X})$, as well as of the spaces $\mathfrak{B}(\mathbf{T}_I)$ and $\mathfrak{B}^-(\mathbf{X})$.

For $x\in M$, let $\tau_I(x)=\min\left\{t\geqslant0:h^+_{-t}x\in I\right\}$. Note that the function $\tau_I$ is uniformly bounded on $M$. Now take a Lipschitz function $f$ on $I$, and introduce a function $\widetilde{f}$ on $M$ by the formula
$$\widetilde{f}(x)=f(h^+_{-\tau_I(x)} x)$$
(setting $\widetilde{f}(x)=0$ for points at which $\tau_I$ is not defined).

By definition, the function $\widetilde{f}$ is weakly Lipschitz, and Theorem \ref{multiplicmoduli} is applicable to $\widetilde{f}$.

The ergodic integrals of $\widetilde{f}$ under $h^+_t$ are of course closely related to ergodic sums of $f$ under $\mathbf{T}_I$, and for any $N\in\mathbb{N}$, $x\in I$, there exists a time $t(x,N)\in\mathbb{R}$ such that
$$\int\limits^{t(x,N)}_0\widetilde{f}\circ h^+_s(x)\,ds=\sum\limits^{N-1}_{k=0}f\circ \mathbf{T}_I^k(x)\:.$$

By the Birkhoff--Khintchine Ergodic Theorem  we have
$$\lim\limits_{N\rightarrow\infty}\frac{t(x,N)}{N}=\frac{1}{Leb(I)}\:,$$
where $Leb(I)$ stands for the length of $I$.

Furthermore, Theorem \ref{multiplicmoduli} yields the existence of constants $C(I)>0$, $\theta\in (0,1)$, such that for all $x\in I$, $N\in\mathbb{N}$, we have
\begin{equation}\label{txn-est}
\left|t(x,N)-\frac{N}{Leb(I)}\right|\leqslant C(I)\cdot N^{\theta}.
\end{equation}

Indeed, the interval $I$ induces a decomposition of our surface into weakly admissible rectangles $\Pi_1,\ldots,\Pi_m$; denote by $h_i$ the height of the rectangle $\Pi_i$, and introduce a weakly Lipschitz function that assumes the constant value $\frac{1}{h_i}$ on each rectangle $\Pi_i$.
Applying Theorem \ref{multiplicmoduli}
 to this function we arrive at desired estimate.

In view of the estimate (\ref{txn-est}), Theorem \ref{multiplicmoduli} applied to the function $\widetilde{f}$ now yields the following Corollary.

\begin{corollary} \label{iet-approx}
Let $\mathbb{P}$ be a
${\bf{g}}_s-invariant$ ergodic probability measure on
${\mathcal{H}}.$ For any $\varepsilon>0$ there exists
$C_{\varepsilon}>0$ depending only on $\mathbb{P}$ such that the
following holds.

For almost every abelian differential
${\bf{X}}\in{\mathcal{H}},{\bf{X}}=(M,\omega),$ any horizontal
closed interval $I\subset M$, any Lipschitz function
$f:I\to\mathbb{R}$, any $x\in I$ and all $n\in\mathbb{N}$ we have
$$\left|\sum_{k=0}^{N-1}f\circ{\bf{T}}_I^k(x)-\Phi_{\widetilde{f}}^+(x,N)\right|
\leqslant C_{\varepsilon}||f||_{Lip}N^{\varepsilon}.$$
\end{corollary}

{\bf Proof.} Applying Theorem \ref{multiplicmoduli} to $\widetilde{f}$, using the estimate
$(\ref{txn-est})$ and noting that the weakly Lipschitz norm of
$\widetilde{f}$ is majorated by the Lipschitz norm of $f,$ we
obtain the desired Corollary.

Let $\theta_1>\theta_2>\ldots>\theta_{l_0}>0$ be the distinct positive
Lyapunov exponents of the measure $\mathbb{P},$ and let
$d_1=1,d_2,\ldots,d_{l_0}$ be the dimensions of the corresponding
subspaces. The tautological identification of
$\mathfrak{B}({\bf{T}}_{I})$ and $\mathfrak{B}^-({\bf{X}})$
together with the results of the previous Corollary now implies
Zorich-type estimates for the growth of ergodic sums of
${\bf{T}}_I.$ More precisely, we have the following

\begin{corollary} \label{iet-logasympt}
In the assumptions of the
preceding Corollary, the space $\mathfrak{B}({\bf{T}}_I)$ admits a
flag of subspaces
$$0=\mathfrak{B}_0\subset\mathfrak{B}_1=\mathbb{R}Leb_I\subset\mathfrak{B}_2
\subset\ldots\subset\mathfrak{B}_{l_0}=\mathfrak{B}({\bf{T}}_I)$$
such that any finitely-additive measure $\Phi\in\mathfrak{B}_i$ in
H\"{o}lder with exponent $\frac{\theta_i}{\theta_1}-\varepsilon$
for any $\varepsilon>0$ and that for any Lipschitz function
$f:I\to\mathbb{R}$ and for any $x\in I$  we have
$$\lim_{N\to\infty}\sup\frac{\log\left|\sum^{N-1}_{k=0}f{\bf{T}}_I^k(x)\right|}{\log
N}=\frac{\theta_{i(f)}}{\theta_1},$$ where $i(f)=1+\max\{j:\int_I fd\Phi=0$ for all
$\Phi\in\mathfrak{B}_j\}$ and by convention we set
$\theta_{l_0+1}=0.$

If with respect to the measure $\Prob$ the  Kontsevich-Zorich cocycle acts isometrically on its neutral
subspaces, then we also have
$\mathfrak{B}_c({\bf{T}}_I)=\mathfrak{B}({\bf{T}}_I).$
\end{corollary}

{\bf{Remark}}.
Corollaries \ref{iet-approx}, \ref{iet-logasympt} thus yield the asymptotic expansion
in terms of H\"{o}lder cocycles as well as Zorich-type logarithmic
estimates for almost all interval exchange transformations with
respect to any conservative ergodic measure $\mu$ on the space of
interval exchange transformations, invariant under the Rauzy-Veech
induction map and such that the Kontsevich-Zorich cocycle is
log-integrable with respect to $\mu.$

In particular, for the Lebesgue measure, if we
let $\R$ be the Rauzy class of the permutation $\pi$, then, using the simplicity of
the Lyapunov spectrum given by the Avila-Viana theorem \cite{AV}, we obtain
\begin{corollary}
For any irreducible permutation $\pi$ and for Lebesgue-almost all $\la$
all continuous finitely-additive measures are H{\"o}lder: we have
 $$\BB({\bf T}_{(\la,\pi)})=\BB_c({\bf T}_{(\la,\pi)}).$$
For any irreducible permutation $\pi$ there exists
a natural number $\rho=\rho(\R)$
depending only on the Rauzy class of $\pi$ and
such that
\begin{enumerate}
\item for Lebesgue-almost all $\la$ we have $\dim \BB(\la,\pi)=\rho$;
\item  all the spaces ${\mathfrak B}_i$ are one-dimensional and $l_0=\rho$.
\end{enumerate}
\end{corollary}

The second statement of Corollary \ref{iet-logasympt} recovers, in the case of the Lebesgue measure on the space of interval exchange transformations, the  Zorich logarithmic asymptotics for ergodic sums \cite{Z}, \cite{zorichwind}.

{\bf Remark.} Objects related to finitely-additive measures for interval exchange
transformations have been studied by X. Bressaud, P. Hubert and A. Maass in \cite{bhm} and by
S. Marmi, P. Moussa and J.-C. Yoccoz in \cite{MMY2}. In particular, the ``limit shapes''
of \cite{MMY2} can be viewed as graphs of  the cocycles $\Phi^+(x,t)$ considered as  functions in $t$.

\subsection{Limit theorems for translation flows.}
\subsubsection{Time integrals as random variables.}
As before, $(M, \omega)$ is an abelian differential, and
$h_t^+$, $h_t^-$ are, respectively, its vertical and horizontal flows. Take
$\tau\in [0,1]$, $s\in {\mathbb R}$, a real-valued
$f\in Lip_{w,0}^+(M, \omega)$ and introduce the function
\begin{equation}
\label{sfstaux}
{\mathfrak S}[f,s;\tau, x]=\int_0^{\tau\exp(s)} f\circ h^{+}_t(x)dt.
\end{equation}

For fixed $f$, $s$ and $x$ the quantity ${\mathfrak S}[f,s;\tau, x]$
is a continuous function of $\tau\in [0,1]$; therefore, as $x$ varies
in the probability space $(M, {\bf m})$, we obtain a random element
of $C[0,1]$. In other words, we have a
 random variable
\begin{equation}
{\mathfrak S}[f,s]: (M, {\bf m})\to C[0,1]
\end{equation}
defined by the formula (\ref{sfstaux}).

For any fixed $\tau\in [0,1]$  the formula (\ref{sfstaux}) yields
a real-valued random variable
\begin{equation}
{\mathfrak S}[f,s; \tau]: (M, {\bf m})\to {\mathbb R},
\end{equation}
whose expectation, by definition, is zero.

Our first aim is to estimate the growth of its variance as $s\to\infty$.
Without losing generality, one may take $\tau=1$.

\subsubsection{The growth rate of the variance in the case of a simple second Lyapunov exponent.}

Let $\Prob$ be an invariant ergodic probability measure for the
Teichm{\"u}ller flow such that with respect to $\Prob$
the second Lyapunov exponent $\theta_2$ of
the Kontsevich-Zorich cocycle is positive and simple (recall that, as Veech and Forni
showed, the  first one, $\theta_1=1$, is always simple \cite{V2, F2} and that,
by the Avila-Viana
theorem \cite{AV}, the second one is simple for the Masur-Veech smooth measure).

For an abelian differential ${\bf X}=(M, \omega)$, denote by $E_{2,{\bf X}}^+$ the one-dimensional
subspace in $H^1(M, {\mathbb R})$ corresponding to the second Lyapunov exponent
$\theta_2$, and let ${\mathfrak B}_{2,{\bf X}}^+={\cal I}_{\bf X}^+(E_{2,{\bf X}}^+)$.
Similarly, denote by $E_{2,{\bf X}}^-$ the one-dimensional
subspace in $H^1(M, {\mathbb R})$ corresponding to the Lyapunov exponent
$-\theta_2$, and let ${\mathfrak B}_{2,{\bf X}}^-={\cal I}_{\bf X}^-(E_{2,{\bf X}}^-)$.

Recall that the space $H^1(M, {\mathbb R})$ is endowed with the Hodge
norm $|\cdot |_H$; the isomorphisms ${\cal I}_{\bf X}^{\pm}$ take the Hodge norm to a norm on ${\BB}_{\bf X}^{\pm}$;
slightly abusing notation, we denote the latter norm by the same symbol.

Introduce a multiplicative cocycle $H_2(s,{\bf X})$ over the Teichm{\"u}ller flow ${\bf g}_s$ by
taking $v\in E_{2,{\bf X}}^+$, $v\neq 0$, and setting
\begin{equation}
\label{h2sx}
H_2(s, {\bf X})=\frac{|{\bf A}(s,{\bf X})v|_H}{|v|_H}.
\end{equation}

The Hodge norm is chosen only for concreteness in (\ref{h2sx}); any other norm can be
used instead.

By definition, we have
\begin{equation}
\lim\limits_{s\to\infty} \frac{\log H_2(s,{\bf X})}{s}=\theta_2.
\end{equation}

Now take $\Phi^+_2\in {\mathfrak B}^+_{2,{\bf X}}$ $\Phi^-_2\in {\mathfrak B}^-_{2,{\bf X}}$ in such a way that
$\langle\Phi_2^+, \Phi_2^-\rangle=1$.

\begin{proposition}
\label{varestfmoduli}
There exists $\alpha>0$ depending only on $\Prob$ and positive measurable  functions
$$
C: \HH\times \HH\to {\mathbb R}_+, \
V: \HH\to {\mathbb R}_+,\
s_0:\HH\to{\mathbb R}_+
$$
such that the following is true for $\Prob$-almost all ${\bf X}\in\HH$.

If $f\in Lip_{w,0}^+({\bf X})$ satisfies $m_{\Phi_2^-}(f)\neq 0$, then for all $s\geq s_0({\bf X})$ we have                                                                                                                                 \begin{equation}\label{varestest}                                                                                                                     \left|\frac{Var_{{\bf m}} {\mathfrak S}(f,x,e^s)}{V({\bf g}_s{\bf X}) (m_{\Phi_2^-}(f)|\Phi_2^+|H_2(s, {\bf X}))^2}-1\right|\leq
C({\bf X}, {\bf g}_s{\bf X})\exp(-\alpha s).
\end{equation}
\end{proposition}
{\bf Remark.} Observe that the quantity $(m_{\Phi_2^-}(f)|\Phi^+|)^2$ does not depend on the specific choice of $\Phi_2^+\in\BB^+_{2, {\bf X}}$,
$\Phi_2^-\in\BB^-_{2, {\bf X}}$ such that $\langle \Phi_2^+, \Phi_2^-\rangle=1$.

{\bf Remark.} Note that by theorems of Egorov and Luzin, the estimate (\ref{varestest})
holds {\it uniformly} on compact subsets of $\HH$ of probability arbitrarily close to $1$.

Proposition \ref{varestfmoduli} is based on
\begin{proposition}
There exists  a  positive measurable function $V: \HH\to {\mathbb R}_+$
such that for $\Prob$-almost all ${\bf X}\in\HH$, we have
\begin{equation}
{Var_{{\bf m}({\bf X})} {\Phi_2^+}(x,e^s)}=V({\bf g}_s{\bf X})|\Phi_2^+|^2(H_2(s, {\bf X}))^2.
\end{equation}
\end{proposition}

In particular ${Var_{{\bf m}} {\Phi_2^+}(x,e^s)}\neq 0$ for any $s\in {\mathbb R}$.
The function $V({\bf X})$ is given by
$$
V({\bf X})=\frac{{Var_{{\bf m}({\bf X})} {\Phi_2^+}(x,1)}}{|\Phi^+_2|^2}.
$$
Observe that the right-hand side does not depend on  a particular choice of $\Phi_2^+\in {\BB}_{2, {\bf X}}^+$, $\Phi_2^+\neq 0$.

\subsubsection{The limit theorem in the case of a simple second Lyapunov exponent.}

Go back to the $C[0,1]$-valued random variable ${\mathfrak S}[f,s]$ and denote by
$\mm[f,s]$ the distribution of the normalized random variable
\begin{equation}
\frac{{\mathfrak S}[f,s]}{{\sqrt{Var_{{\bf m}} {\mathfrak S}[f,s;1]}}}.
\end{equation}
The measure $\mm[f,s]$ is thus a probability distribution on the space $C[0,1]$ of continuous functions on the unit interval.

For $\tau\in {\mathbb R}$, $\tau\neq 0$, we also let $\mm[f,s; \tau]$ be the distribution of the ${\mathbb R}$-valued random variable
\begin{equation}
\frac{{\mathfrak S}[f,s; \tau]}{{\sqrt{Var_{{\bf m}} {\mathfrak S}[f,s; \tau]}}}.
\end{equation}

If $f$ has zero average, then, by definition,  $\mm[f,s; \tau]$ is a measure on ${\mathbb R}$ of expectation $0$ and variance $1$.

By definition, $\mm[f,s]$ is a Borel probability measure on $C[0,1]$; furthermore, if
$\xi=\xi(\tau)\in C[0,1]$, then the following natural normalization requirements hold for $\xi$ with respect to  $\mm[f,s]$:
\begin{enumerate}
\item $\xi(0)=0$ almost surely with respect to $m[f,s]$;
\item $\ee_{\mm[f,s]}\xi(\tau)=0$ for all $\tau\in [0,1]$;
\item $Var_{\mm[f,s]}\xi(1)=1$.
\end{enumerate}

We are interested in the weak accumulation points of $\mm[f,s]$ as $s\to\infty$.

Consider the space $\HH^{\prime}$ given by
the formula
$$
\HH^{\prime}=\{{\bf X}^{\prime}=(M, \omega, v), v\in E_2^+(M, \omega), |v|_H=1\}.
$$
By definition, the space $\HH^{\prime}$ is a $\prob$-almost surely defined
two-to-one cover of the space $\HH$. The skew-product flow of the Kontsevich-Zorich cocycle
over the Teichm{\"u}ller
flow yields a flow ${\bf g}_s^{\prime}$ on $\HH^{\prime}$ given by the formula
$$
{\bf g}_s^{\prime}({\bf X}, v)=({\bf g}_s{\bf X}, \frac{{\bf A}(s,{\bf X})v}{|{\bf A}(s,{\bf X})v|_H}).
$$
Given ${\bf X}^{\prime}\in \HH^{\prime}$, set
$$
\Phi_{2,{\bf X}^{\prime}}^+={\cal I}^+(v).
$$

Take ${\tilde v}\in E_2^-(M, \omega)$ such that $\langle v, {\tilde v}\rangle=1$ and set
$$
\Phi_{2,{\bf X}^{\prime}}^-={\cal I}^-(v), \ m_{2,{\bf X}^{\prime}}^-=m^-_{\Phi_{2,{\bf X}^{\prime}}}.
$$

Let $\MM$ be the space of all probability distributions on $C[0,1]$ and
introduce a $\Prob$-almost surely defined map
$
{\cal D}_2^+: \HH^{\prime}\to \MM
$
by setting  ${\cal D}_2^+({\bf X}^{\prime})$ to be the distribution
of the $C[0,1]$-valued normalized random variable
$$
\frac{\Phi^+_{2,{\bf X}^{\prime}}(x, \tau)}{\sqrt{Var_{{\bf m}}\Phi^+_{2,{\bf X}^{\prime}}(x, 1)}}, \  \tau\in[0,1].
$$

By definition, ${\cal D}_2^+({\bf X}^{\prime})$ is a Borel
probability measure on the space $C[0,1]$; it is, besides,
a compactly supported measure as its support
consists of equibounded H{\"o}lder functions with exponent
$\theta_2/\theta_1-\varepsilon$.

Consider the set
$\MM_1$ of probability measures $\mm$ on $C[0,1]$ satisfying, for $\xi\in C[0,1]$,
$\xi=\xi(t)$,  the conditions:
\begin{enumerate}
\item the equality $\xi(0)=0$ holds $\mm$-almost surely;

\item for all $\tau$ we have $\ee_{\mm}\xi(\tau)=0$:

\item we have $Var_{\mm}\xi(1)=1$ and for any $\tau\neq 0$ we have $Var_{\mm}\xi(\tau)\neq 0$.
\end{enumerate}

It will be proved in what follows that ${\cal D}_2^+(\HH^{\prime})\subset \MM_1$.

Consider a semi-flow $J_s$ on the space $C[0,1]$
defined by the formula
$$
J_s\xi(t)=\xi(e^{-s}t), \ s\geq 0.
$$

Introduce a semi-flow $G_s$ on $\MM_1$ by the formula
\begin{equation}
G_s\mm=\frac{(J_s)_*\mm}{Var_{\mm}(\xi(\exp(-s))}, \mm\in \MM_1.
\end{equation}

By definition, the diagram
$$
\begin{CD}
\label{Dtwo}
\HH^{\prime}@ >{\cal D}_2^+>> {\MM_1} \\
@VV{\bf g}_sV           @AAG_sA   \\
\HH^{\prime}@ >{\cal D}_2^+>> {\MM_1} \\
\end{CD}
$$
is commutative.

Let $d_{LP}$ be the L{\'e}vy-Prohorov metric and let $d_{KR}$ be the Kantorovich-Rubinstein metric
on the space of probability measures on $C[0,1]$ (see \cite{billingsley}, \cite{bogachev} and the Appendix B).

We are now ready to formulate the main result of this subsection.

\begin{proposition}
\label{limthmmoduli-simple}
Let $\Prob$ be a ${\bf g}_s$-invariant ergodic probability measure on ${\HH}$
such that the second Lyapunov exponent of the Kontsevich-Zorich
cocycle is positive and simple with respect to $\Prob$.

There exists a positive measurable function $C: \HH\times \HH\to {\mathbb R}_+$
and a positive constant $\alpha$ depending only on $\Prob$
such that
for $\Prob$-almost every ${\bf X}^{\prime}\in\HH^{\prime}$, ${\bf X}^{\prime}=({\bf X}, v)$,
and any $f\in Lip_{w,0}^+({\bf X})$
satisfying
$m_{2, {\bf X}^{\prime}}^-(f)>0$ we have
\begin{equation}
\label{lpexpmoduli}
d_{LP}(\mm[f, s], {\cal D}_2^+({\bf g}_s^{\prime}{\bf X}^{\prime}))\leq
C({\bf X}, {\bf g}_s{\bf X})\exp(-\alpha s).
\end{equation}
\begin{equation}
\label{krexpmoduli}
d_{KR}(\mm[f, s], {\cal D}_2^+({\bf g}_s^{\prime}{\bf X}^{\prime}))\leq
C({\bf X}, {\bf g}_s{\bf X})\exp(-\alpha s).
\end{equation}
\end{proposition}

Now fix $\tau\in {\mathbb R}$ and let ${\mm}_2({\bf X}^{\prime}, \tau)$ be the distribution
of the ${\mathbb R}$-valued random variable
$$
\frac{\Phi^+_{2,{\bf X}^{\prime}}(x, \tau)}{\sqrt{Var_{{\bf m}}\Phi^+_{2,{\bf X}^{\prime}}(x, \tau)}}.
$$
For brevity, write ${\mm}_2({\bf X}^{\prime}, 1)={\mm}_2({\bf X}^{\prime})$.
\begin{proposition}
For $\Prob$-almost any ${\bf X}^{\prime}\in\HH^{\prime}$, the measure
${\mm}_2({\bf X}^{\prime}, \tau)$ admits atoms for a dense set of times $\tau\in {\mathbb R}$.
\end{proposition}

A more general proposition on the existence of atoms will be formulated in the following subsection.

Proposition \ref{limthmmoduli-simple} implies that the omega-limit set of the family $\mm[f, s]$
can generically assume at most two values. More precisely, the ergodic measure $\Prob$ on $\HH$ is
naturally lifted to its ``double cover''
on the space $\HH^{\prime}$: each point in the fibre is assigned equal weight;
the resulting measure is denoted  $\Prob^{\prime}$. By definition,
the measure $\Prob^{\prime}$  has no more than two ergodic components. We therefore arrive
at the following
\begin{corollary}\label{nonentwo}
Let $\Prob$ be a ${\bf g}_s$-invariant ergodic probability measure on ${\HH}$
such that the second Lyapunov exponent of the Kontsevich-Zorich
cocycle is positive and simple with respect to $\Prob$.

There exist two closed sets ${\mathfrak N}_1, {\mathfrak N}_2\subset {\mathfrak M}$ such that
for $\Prob$-almost every ${\bf X}\in {\mathcal H}$ and any $f\in Lip_{w,0}^+({\bf X})$
satisfying $\Phi_f^+\neq 0$ the omega-limit set of the family  $\mm[f, s]$ either
coincides with ${\mathfrak N}_1$ or with ${\mathfrak N}_2$.
If, additionally, the measure $\Prob^{\prime}$ is ergodic, then ${\mathfrak N}_1={\mathfrak N}_2$.
\end{corollary}

{\bf Question}. Do the sets ${\mathfrak N}_i$ contain measures with non-compact support?

For horocycle flows on compact surfaces of constant negative curvature, compactness of support 
 for all weak accumulation points of ergodic integrals has been obtained by Flaminio and Forni \cite{flafo}.

{\bf Question}. Is the measure $\Prob^{\prime}$ ergodic when ${\Prob}$  is the Masur-Veech measure?

As we shall see in the next subsection, in general, the omega-limit sets of the distributions of
the ${\mathbb R}$-valued random variables ${\mathfrak S}[f,s;1]$ contain the delta-measure at zero.
As a consequence, it will develop that, under certain assumptions on the measure $\Prob$, which are satisfied, in particular, for the Masur-Veech smooth measure, for a generic abelian differential the random variables ${\mathfrak S}[f,s;1]$
{\it do not converge} in distribution, as $s\to\infty$, for any
function $f\in Lip_{w,0}$ such that $\Phi_f^+\neq 0$.

\subsubsection{The general case}

While, by the Avila-Viana Theorem \cite{AV}, the Lyapunov spectrum of
the Masur-Veech measure is simple, there are also natural examples of invariant measures with non-simple positive second
Lyapunov  exponent due to Eskin-Kontsevich-Zorich \cite{EKZ}, G. Forni \cite{forni-nonunif},
C. Matheus (see Appendix A.1 in \cite{forni-nonunif}).
A slightly more elaborate, but similar, construction is needed to obtain limit theorems in this  general case.

Let $\mathbb{P}$ be an invariant ergodic probability measure for
the Teichm\"{u}ller flow, and let
$$\theta_1=1>\theta_2>\ldots>\theta_{l_0}>0$$ be
the distinct positive Lyapunov exponents of the Kontsevich-Zorich cocycle
with respect to $\mathbb{P}.$ We assume $l_0\geq 2$.

As before, for ${\bf X}\in {\mathcal H}$ and $i=2, \dots, l_0$, let
$E_i^u({\bf X})$ be the corresponding  Oseledets subspaces, and let
$\mathfrak{B}_{i}^+({\bf X})$ be the corresponding spaces of cocycles.
To make notation lighter, we omit the symbol ${\bf X}$ when the abelian
differential is held fixed.

For $f\in Lip_{w}^+({\bf{X}})$ we now write
$$\Phi_f^+=\Phi_{1,f}^++\Phi_{2,f}^++\ldots +\Phi_{l_0,f}^+,$$
with $\Phi_{i,f}^+\in\mathfrak{B}_{i}^+$ and,
of course, with
$$\Phi^+_{1,f}=(\int\limits_{M}fd{\bf m})\cdot\nu^+,$$
where $\nu^+$ is the Lebesgue measure on the vertical foliation.

For each $i=2,\ldots,l_0$ introduce a measurable fibre bundle
$${\bf{S}}^{(i)}{\mathcal H}=\{({\bf X},v):{\bf X}
\in{\mathcal H},v\in E_{i}^+,|v|=1\}.$$

The flow ${\bf g}_s$ is naturally lifted to the space
${\bf{S}}^{(i)}\HH$ by the formula
$${\bf g}_s^{{\bf{S}}^{(i)}}({\bf X},v)=\left({\bf g}_s {\bf X},
\frac{\mathbf{A}(s,{\bf X})v}
{|\mathbf{A}(s,{\bf X})v|}\right).$$ The growth of the norm of
vectors $v\in E_i^+$ is controlled by the multiplicative cocycle
$H_i$ over the flow ${\bf g}_s^{{\bf{S}}^{(i)}}$ defined by the formula
$$H_i(s,({\bf X},v))=\frac{\mathbf{A}(s,{\bf X})v}
{|v|}.$$

For
${\bf X}\in {\mathcal H}$ and $f\in
Lip_{w,0}^+(\bf{X})$ satisfying $\Phi_f^+\neq 0$, denote
$$
i(f)=\min\{j:\Phi_{f,j}^+\neq0\}.
$$

Define
$v_{f}\in E_{i(f)}^u$ by the formula
$$\mathcal{I}_{{\bf X}}^+(v_{f})=\frac{\Phi_{f,i(f)}^+}{|\Phi_{f,i(f)}^+|}.$$

The growth of the variance of the ergodic integral of a weakly Lipschitz function $f$ is also,
similarly to the case of the simple second Lyapunov exponent, described by the cocycle $H_{i(f)}$ in the following way.

\begin{proposition} There exists $\alpha>0$
depending only on $\mathbb{P}$ and, for any $i=2,\ldots,l_0,$
positive measurable functions
$$V^{(i)}:{\bf{S}}^{(i)}{\mathcal H}\to\mathbb{R}_+,C^{(i)}:{\mathcal H}
\times{\mathcal H}\to\mathbb{R}_+$$ such that for
$\mathbb{P}$-almost every ${\bf X}\in{\mathcal H}$
the following holds. Let $f\in
Lip_{w,0}^+(\bf{X})$ satisfy $\Phi_f^+\neq 0$.

Then for all $s>0$ we have
$$\left|\frac{Var_{{\bf m}}(\mathfrak{S}[f,s;1])}
{V^{(i(f))}(g_s^{{\bf{S}}^{(i(f))}}(\overline{\omega},v_{f}))(H_{i(f)}(s, ({\bf X},v_{f})
))^2}-1\right|\leqslant
C^{(i(f))}({\bf X},{\bf g}_s{\bf X})e^{-\alpha s}.$$
\end{proposition}

We proceed to the formulation and the proof of the limit theorem in
the general case. For $i=2,\ldots,l_0$, introduce the map
$$\mathcal{D}_i^+:{\bf{S}}^{(i)}\HH\to\mathfrak{M}_1$$
by setting $\mathcal{D}_i^+({\bf X},v)$ to be the
distribution of the $C[0,1]$-valued random variable
$$\frac{\Phi^+_v(x,\tau)}{\sqrt{Var_{\nu( \overline{\omega})}
(\Phi_v^+(x,1))}},\tau\in[0,1].$$

As before, we have a commutative diagram
$$
\begin{CD}
\label{Dtwo-diagr}
{{\bf S}^{(i)}}\HH@ >{{\mathcal D}_i^+}>> {\MM_1} \\
@VV{\bf g}_s^{{\bf S}^{(i)}}V           @AAG_sA   \\
{{\bf S}^{(i)}}\HH@ >{{\mathcal D}_i^+}>> {\MM_1} \\
\end{CD}
$$

The measure $\mathfrak{m}[f,s]$, as before, stands for the distribution of the
$C[0,1]$-valued random variable
$$\frac
        {\int\limits_0^ {\tau
\exp(s)} f\circ h_t^+(x)dt}
{\sqrt
      {Var_ {{\bf m}}(\int\limits_0^ {
\exp(s)}f\circ h_t^+(x)dt)}},\tau\in[0,1].$$

\begin{theorem}\label{limthmmoduli}
Let $\mathbb{P}$ be an invariant ergodic probability measure for
the Teichm\"{u}ller flow such that the Kontsevich-Zorich cocycle
admits at least two distinct positive Lyapunov exponents with respect to ${\mathbb P}$.
There exists a constant $\alpha>0$ depending only on $\mathbb{P}$
and a positive measurable map
$C:{\mathcal H}\times{\mathcal H}\to\mathbb{R}_+$ such
that for $\mathbb{P}-almost$ every
${\bf X}\in{\mathcal H}$ and any $f\in
Lip_w^+(\bf{X})$ we have
$$d_{LP}(\mathfrak{m}[f,s],{\mathcal D}_{i(f)}^+(g_s^{{\bf{S}}^{(i(f))}}({\bf X},v_f)))
\leqslant C({\bf X},{\bf g}_s {\bf X})e^{-\alpha
s},$$
$$d_{KR}(\mathfrak{m}[f,s],{\mathcal D}_{i(f)}^+(g_s^{{\bf{S}}^{(i(f))}}({\bf X},v_f)))
\leqslant C({\bf X},{\bf g}_s{\bf X})e^{-\alpha
s}.$$
\end{theorem}

\subsubsection{ Atoms of limit distributions.}

For $\Phi^+\in {\mathfrak B}^+({\bf X})$, let
$\mathfrak{m}[\Phi^+,\tau]$ be the distribution of the
$\mathbb{R}$-valued random variable
$$\frac{\Phi^{+}(x,\tau)}{\sqrt{Var_{\bf{m}}\Phi^{+}(x,\tau)}}\ .$$

\begin{proposition} For $\mathbb{P}-almost$ every
$\bf{X}\in\mathcal{H},$ there exists a dense set
$T_{atom}\subset\mathbb{R}$ such that if $\tau\in T_{atom},$
 then for any $\Phi^+\in {\mathfrak B}^+({\bf X})$, $\Phi^+\neq 0$, the measure
$\mathfrak{m}(\Phi^+,\tau)$ admits atoms.
\end{proposition}

\subsubsection{Nonconvergence in distribution of ergodic integrals.}

Our next aim is to show that along certain subsequences of times the ergodic integrals of translation flows
converge in distribution  to the measure $\delta_0$, the delta-mass at zero. Weak convergence
of probability measures will be denoted by the symbol $\Rightarrow$.

We need the following additional assumption on the measure $\Prob$.
\begin{assumption}
\label{aczero}
For any $\varepsilon>0$ the set of
abelian differentials ${\bf X}=(M, \omega)$ such that  there exists an
admissible rectangle $\Pi(x, t_1, t_2)\subset M$ with
$t_1>1-\varepsilon$, $t_2>1-\varepsilon$
has positive measure with respect to $\Prob$.
\end{assumption}
Of course, this assumption holds for the Masur-Veech smooth measure.

\begin{proposition} Let $\mathbb{P}$ be an ergodic
$\bf{g}_s$ invariant measure on $\mathcal{H}$ satisfying
Assumption \ref{aczero}. Then for $\mathbb{P}$-almost
every ${\bf{X}}\in\mathcal{H}$ there exists a sequence
$\tau_n\in\mathbb{R}_+$ such that for any
$\Phi^+\in{\mathfrak{B}}^+({\bf{X}})$, $\Phi^+\neq 0$,  we have
$$\mathfrak{m}[\Phi^+, \tau_n]\Rightarrow\delta_0\ \mathrm{in} \ {\mathfrak M}({\mathbb R}) \  \mathrm{as}\
n\to\infty.$$
\end{proposition}
Theorem \ref{limthmmoduli} now implies the following
\begin{corollary} Let $\mathbb{P}$ be an ergodic
$\bf{g}_s$ invariant measure on $\mathcal{H}$ satisfying
Assumption \ref{aczero}. Then for $\mathbb{P}$-almost
every ${\bf{X}}\in\mathcal{H}$ there exists a sequence
$s_n\in\mathbb{R}_+$ such that for any $f\in
Lip_{w,0}^+({\bf{X}})$ satisfying $\Phi_f^+\neq 0$ we have
$$\mathfrak{m}[f,s_n; 1]\Rightarrow\delta_0\  \mathrm{in} \   {\mathfrak M}({\mathbb R}) \ \mathrm{as}\
n\to\infty.$$

Consequently, if $f\in Lip_{w,0}^+({\bf{X}})$ satisfies
$\Phi_f^+\neq0,$ then the family of measures
$\mathfrak{m}[f,\tau;1] $ does not converge in ${\mathfrak M}({\mathbb R})$
and the family of measures $\mathfrak{m}[f,\tau] $ does not converge in ${\mathfrak M}(C[0,1])$
as $\tau\to\infty.$
\end{corollary}

\subsection{A symbolic coding for  translation flows.}
\subsubsection{Interval exchange transformations as Vershik's automorphisms.}
Recall that, by Vershik's Theorem \cite{Vershik1}, every ergodic automorphism of a Lebesgue
probability space can be represented as a Vershik's automorphism.
The proof of Vershik's Theorem \cite{Vershik1} proceeds
by constructing an increasing sequence of Rohlin towers which intersect ``in a Markov way".
In the case of interval exchange transformations, such a sequence of towers is given, for instance,
by the Rauzy-Veech induction; as a result, one obtains an explicit representation of minimal
interval exchange transformations as Vershik's automorphisms (see \cite{gjerde}). In the next subsection,
a bi-infinite variant of this construction will give a symbolic representation for translation flows
on flat surfaces.

Let $\pi$ be an irreducible permutation on $m$ symbols and let
${\bf T}:[0,1)\to [0,1)$ be a minimal  interval exchange transformation
of $m$ intervals with
permutation $\pi$.

One can find
a sequence of intervals $I^{(n)}=[0, b^{(n)})$, $n=0, \dots$,
such that
\begin{enumerate}
\item $\lim\limits_{n\to\infty}b^{(n)}=0$;
\item $I^{(n+1)}\subset I^{(n)}$;
\item the induced map of ${\bf T}$ on $I^{(n)}$ is again an interval
exchange of $m$ subintervals.
\end{enumerate}

Denote by $\TT_n$ the induced map of $\TT$ on $I^{(n)}$, and let
$I^{(n)}_1, \dots, I^{(n)}_m$, $I_i^{(n)}=[a_i^{(n)}, b_i^{(n)})$ be the subintervals of
the interval exchange $\TT_n$. By definition, we have $a_1^{(n)}=0$,  $b_i^{(n)})=a_i^{(n+1)}$.

Now represent $I^{(n)}$ as a union of Rohlin towers over $I^{(n+1)}$ with respect to the map $\TT_{n}$:
\begin{equation}
I^{(n)}=\bigsqcup_{i=1}^m \bigsqcup_{k=0}^{N_i^{(n+1)}-1} \TT_{n}^k I_i^{(n+1)}.
\end{equation}

Here $N_i^{(n+1)}$, the height of the tower,
is the time of the first return of $I_i^{(n+1)}$ into $I^{(n)}$ under the map $\TT_{n}$.

Let
$$
{\cal E}_{n+1}=\{(i,k): i\in\{1, \dots, m\}, k\in \{0, \dots, N_i^{(n+1)}-1\}\}.
$$
and for $e\in {\cal E}_{n+1}$, $e=(i,k)$, denote
$$
J^{(n+1)}_e=\TT_{n+1}^kI_i^{(n+1)}.
$$
For any $e\in {\cal E}_{n+1}$ there exists a unique $j\in\{1, \dots, m\}$ such that
$$
J^{(n+1)}_e\subset I^{(n)}_j.
$$
We denote $j=F(e)$; we also write $i=I(e)$ if $e=(i,k)$.

We thus have
\begin{equation}
\label{ijne}
I_j^{(n)}=\bigsqcup_{e\in {\cal E}_{n+1}: F(e)=j} J^{(n+1)}_e.
\end{equation}

Now represent $I=[0,1)$ as a union of Rohlin towers over $I^{(n)}$ with respect to $\TT$:
\begin{equation}
\label{iin}
I=\bigsqcup_{i=1}^m \bigsqcup_{k=0}^{L_i^{(n+1)}-1} \TT^k I_i^{(n)}.
\end{equation}

 Substituting (\ref{ijne}) into (\ref{iin}), write
 \begin{equation}
\label{iinl}
I=\bigsqcup_{i=1}^m \bigsqcup_{k=0}^{L_i^{(n+1)}-1}  \bigsqcup_{e\in {\cal E}_{n+1}: F(e)=j}  \TT^k J^{(n+1)}_e.
\end{equation}

Denote the resulting partition of $I$ into subintervals by $\pi_{n+1}$.
By definition, the maximal length of an element of $\pi_n$ tends to $0$ as $n\to\infty$, so
the increasing sequence of partitions $\pi_n$ tends
(in the sense of Rohlin)  to the partition into points. As usual, for $x\in I$, let $\pi_n(x)$ be the element of $\pi_n$ containing $x$.

Introduce a function ${\mathfrak i}_n: I\to {\cal E}_n$ by setting
${\mathfrak i}_n(x)=e$ if $\pi_n(x)$ has the form $\TT^lJ^{(n)}_e$.
A finite string  $(e_1, \dots, e_n)$,
$e_i\in {\cal E}_i$, satisfying $F(e_i)=I(e_{i-1})$, will be called {\it admissible}.
\begin{proposition}
Let $(e_1, \dots, e_n)$, $e_i\in {\cal E}_i$, be an admissible string.
There exists a unique interval $J=J(e_1, \dots, e_n)$ such that
\begin{enumerate}
\item $J$ is an element of the partition $\pi_n$
\item for any $x\in J$ we have  ${\mathfrak i}_l(x)=e_l$, $l=1, \dots, n$.
\end{enumerate}
Conversely, any element of the partition $\pi_n$ has the form  $J(e_1, \dots, e_n)$
for a unique admissible string $(e_1, \dots, e_n)$.
\end{proposition}

The proof  is immediate by induction.

Introduce the Markov compactum
\begin{equation}
Y=\{y=y_1\dots y_n\dots: y_i\in {\cal E}_i, F(y_i)=I(y_{i-1})\}.
\end{equation}

We thus have a natural map $p: Y\to [0,1]$ which sends  $y\in Y$ to the point
\begin{equation}
\bigcap\limits_{n=1}^{\infty} {\overline J(y_1, \dots, y_n)}.
\end{equation}

(here ${\overline J}$ stands for the closure of $J$). The map $p$ is surjective
and is indeed a bijection except at the endpoints of the intervals $J=J(e_1, \dots, e_n)$,
all of which, except $0$ and $1$, have two preimages. In particular, the map $p$
is almost surely bijective with respect to the Lebesgue measure on $[0,1]$;
by definition, the image $\nu_Y$
of the Lebesgue measure on $[0,1]$ under $p^{-1}$ is a
Markov measure on $Y$.

The Markov compactum $Y$ also has  the following additional structure.
Each set ${\cal E}_n$ is partially ordered: for $e_1, e_2\in {\cal E}_n$, $e_1=(i_1, k_1)$, $e_2=(i_2, k_2)$,
we write $(i_1, k_1)<(i_2, k_2)$ if $i_1=i_2$, $k_1<k_2$.

This ordering induces a partial ordering ${\mathfrak o}$ on $Y$: we write $y<{\tilde y}$ if
there exists $n_0$ such that  $y_n={\tilde y}_n$ for $n>n_0$ while $y_{n_0}<{\tilde y}_{n_0}$.

The map $p^{-1}\circ \TT\circ p: Y\to Y$ is a Vershik's automorphism on $Y$ with respect to the partial ordering ${\mathfrak o}$ (see \cite{Vershik1, VL, solomyak}).

\subsubsection{Translation flows as symbolic flows.}

Let $(M, \omega)$ be an abelian differential such that both corresponding flows $h_t^+$ and $h_t^-$
are minimal. A rectangle  $\Pi(x, t_1, t_2)=\{h^+_{\tau_1} h^{-}_{\tau_2}x, 0\leq \tau_1< t_1, 0\leq \tau_2< t_2\}$
is called {\it weakly admissible} if for all sufficiently small  $\varepsilon>0$ the rectangle
$\Pi(h_{\varepsilon}^+h_{\varepsilon}^-x, t_1-\varepsilon, t_2-\varepsilon)$ is admissible
(in other words, the boundary of $\Pi$ may contain zeros of $\omega$ but the interior does not).

There exists a sequence of partitions
\begin{equation}
\label{recpart}
M=\Pi_1^{(n)}\sqcup \dots \Pi_m^{(n)}, \ n\in {\mathbb Z},
\end{equation}
where $\Pi_i^{(n)}$ are weakly admissible rectangles and for any $n_1, n_2\in {\mathbb Z}$, $i_1,i_2\in \{1,\dots, m\}$, the rectangles
$\Pi_{i_1}^{(n_1)}$ and $\Pi_{i_2}^{(n_2)}$ intersect in a Markov way in the following precise sense.

Take a weakly admissible rectangle $\Pi(x, t_1, t_2)$ and decompose its boundary into four parts:
$$
\partial_h^1(\Pi)={\overline{ \{h^+_{t_1} h^{-}_{\tau_2}x,  0\leq \tau_2< t_2\}}};
$$
$$
\partial_h^0(\Pi)={\overline {\{ h^{-}_{\tau_2}x,  0\leq \tau_2< t_2\}}};
$$
$$
\partial_v^1(\Pi)={\overline {\{h^-_{t_2} h^{+}_{\tau_1}x,  0\leq \tau_1< t_1\}}};
$$
$$
\partial_v^0(\Pi)={\overline {\{h^{+}_{\tau_1}x,  0\leq \tau_1< t_1\}}}.
$$
The {\it Markov condition} is then the requirement that for any $n\in {\mathbb Z}$ and $i\in \{1, \dots, m\}$
there exist $i_1, i_2, i_3, i_4\in \{1, \dots, m\}$
such that
$$
\partial_h^1(\Pi_i^{(n)})\subset \partial_h^{1}\Pi_{i_1}^{(n-1)};
$$
$$
\partial_h^0(\Pi_i^{(n)})\subset \partial_h^{0}\Pi_{i_2}^{(n-1)};
$$
$$
\partial_v^1(\Pi_i^{(n)})\subset \partial_v^{1}\Pi_{i_3}^{(n+1)};
$$
$$
\partial_v^0(\Pi_i^{(n)})\subset \partial_v^{0}\Pi_{i_4}^{(n+1)}.
$$

Furthermore, for a weakly admissible rectangle $\Pi=\Pi(x, t_1, t_2)$
we write $|\partial_h(\Pi)|=t_2$, $|\partial_v(\Pi)|=t_1$ and
require
\begin{equation}
\label{minimality}
\lim\limits_{n\to\infty} \max\limits_{i=1, \dots,m} |\partial_v\Pi_i^{(n)}|=0; \
\lim\limits_{n\to\infty} \max\limits_{i=1, \dots,m} |\partial_h\Pi_i^{(-n)}|=0.
\end{equation}

A sequence of partitions (\ref{recpart}) satisfying the Markov condition
and (\ref{minimality}) exists by the minimality of the vertical and horizontal foliations; this sequence allows us
to identify our surface $M$ with the space of paths of a non-autonomous topological Markov chain.

Indeed, for $n\in {\mathbb Z}$, let ${\cal E}_n$ be the set of connected components of intersections
$\Pi_i^{(n)}\cap \Pi_j^{(n-1)}$. If $e$ is such a connected component
then we write $i=I(e)$, $j=F(e)$ (see Figure \ref{fig:three}).

\begin{figure}
\begin{center}
\includegraphics{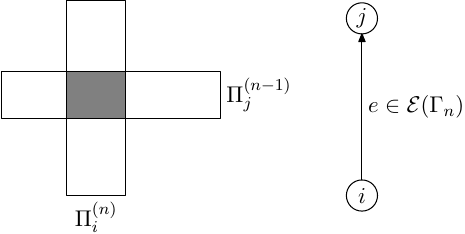}\\
\caption{Edges of graphs are connected components of intersections of rectangles}\label{fig:three}
\end{center}
\end{figure}

Now consider the {\it Markov compactum} of bi-infinite paths, that is, the set $X$ defined by the formula
$$
X=\{x=\dots x_{-n}\dots x_n\dots, x_n\in {\cal E}_n, F(x_n)=I(x_{n-1}), \ n\in {\mathbb Z}\}.
$$

Given a point $x\in X$, consider the intersection of closures of the corresponding connected components
\begin{equation}
\label{conncomp}
\pi(x)=\bigcap {\overline x_n}.
\end{equation}

Nonemptiness of this intersection follows from the Markov condition, while condition
(\ref{minimality}) implies that the intersection (\ref{conncomp}) is a point.

We therefore obtain a measurable  map $\pi: X\to M$.
It is immediate that ${\bf m}$-almost all points in $M$ have exactly one preimage,
and we thus obtain an almost sure identification of $M$ and $X$.
We illustrate the connection between $M$ and $X$ on Figure \ref{fig:four}.

\begin{figure}[ht]
\begin{center}
\includegraphics{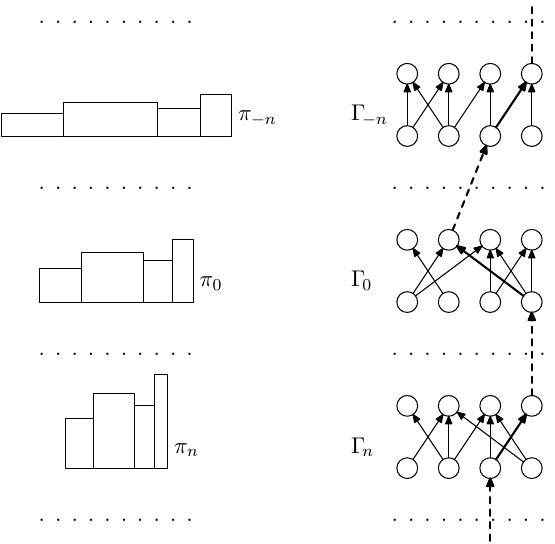}\\
\caption{From an abelian differential to a  Markov compactum}\label{fig:four}
\end{center}
\end{figure}

By construction, the measure ${\bf m}$ yields a Markov measure on $X$.
Furthermore, by construction we immediately obtain the following
\begin{proposition}
Let $n_0\in {\mathbb Z}$. If $x, x^{\prime}$ are such that $x_n=x^{\prime}_n$
for all $n\geq n_0$, then  $\pi(x)$ and $\pi(x^{\prime})$ lie on the same orbit of the flow $h_t^+$;
if   $x, x^{\prime}$ are such that $x_n=x^{\prime}_n$ for all $n\leq n_0$,
then  $\pi(x)$ and $\pi(x^{\prime})$ lie on the same orbit of the flow $h_t^-$.
\end{proposition}

In other words, the horizontal and the vertical flows of the abelian differential
$\omega$ correspond to flows along asymptotic foliations of a Markov compactum.
It will develop that H{\"o}lder cocycles
correspond to special finitely-additive measures  on the asymptotic foliations.

\subsection{Markov compacta.}
\subsubsection{Graphs and paths}
Let $m\in {\mathbb N}$ and let $\Gamma$ be an oriented graph with $m$ vertices $\{1,\dots,m\}$
and possibly multiple
edges. We assume that that for each vertex there is an edge starting
from it and an edge ending in it.

Let ${\cal E}(\Gamma)$ be the set of edges of $\Gamma$.
For $e\in {\cal E}(\Gamma)$ we denote by $I(e)$ its initial vertex and by $F(e)$ its terminal vertex.
Denote by
$A(\Gamma)$ the incidence matrix of $\Gamma$ given by the formula:
$$
A_{ij}(\Gamma)=\# \{e\in {\cal E}(\Gamma): I(e)=i, F(e)=j\}.
$$

Let ${\mathfrak G}$ be the set of all oriented graphs on $m$ vertices
such that there is an edge starting at
every vertex and an edge ending at every vertex.

Assume we are given a sequence $\{\Gamma_n\}$, $n\in {\mathbb Z}$, of graphs belonging to ${\mathfrak G}$.
To this sequence we assign the Markov compactum of paths in our sequence of graphs:
\begin{equation}
\label{markcomp}
X=\{x=\dots x_{-n}\dots x_n \dots, x_n\in {\cal E}(\Gamma_n), F(x_{n+1})=I(x_n)\}.
\end{equation}

Write $A_n(X)=A(\Gamma_n)$.

\subsubsection{Asymptotic foliations.}

For $x\in X$, $n\in {\mathbb Z}$, introduce the sets
$$
\gamma^+_n(x)=\{x^{\prime}\in X: x^{\prime}_t=x_t, t\geq n\};\
\gamma^-_n(x)=\{x^{\prime}\in X: x^{\prime}_t=x_t, t\leq n\};
$$
$$
\gamma^+_{\infty}(x)=\bigcup_{n\in {\mathbb Z}} \gamma_n^+(x);\
\gamma^-_{\infty}(x)=\bigcup_{n\in {\mathbb Z}} \gamma_n^-(x).
$$
The sets $\gamma^+_{\infty}(x)$ are leaves of the asymptotic foliation  ${\cal F}^+(X)$
on $X$ corresponding to the infinite future; the sets $\gamma^-_{\infty}(x)$
are leaves of the asymptotic foliation  ${\cal F}^-(X)$ on $X$ corresponding to the infinite past.

For $n\in {\mathbb Z}$ let ${\mathfrak C}^+_{n}(X)$ be
the collection of all subsets of $X$ of the
form
$\gamma^+_n(x)$,  $x\in X$; similarly,
${\mathfrak C}^-_{n}(X)$ is the collection of all subsets of the form  $\gamma^-_n(x)$.

By definition, the collections ${\mathfrak C}^+_{n}(X)$, ${\mathfrak C}^-_{n}(X)$ are semi-rings.
Introduce the collections
\begin{equation}
\label{cplusom}
{\mathfrak C}^+(X)=\bigcup\limits_{n\in {\mathbb Z}} {\mathfrak C}^+_{n};
{\mathfrak C}^-=\bigcup\limits_{n\in {\mathbb Z}} {\mathfrak C}^-_{n}.
\end{equation}

Since every element of ${\mathfrak C}^+_n$ is a disjoint union of elements of ${\mathfrak C}^+_{n+1}$,
the collection ${\mathfrak C}^+$
is a semi-ring as well. The same  statements hold for ${\mathfrak C}^-_n$ and  ${\mathfrak C}^-$.

Cylinders in $X$ are subsets of the form $\{x: x_{n+1}=e_1, \dots, x_{n+k}=e_{k}\}$, where
$n\in {\mathbb Z}$, $k\in {\mathbb N}$,
$e_1\in {\cal E}(\Gamma_{n+1}), \dots,  e_k\in {\cal E}(\Gamma_{n+k})$ and
$F(e_i)=I(e_{i+1})$.
The family of all cylinders forms a semi-ring which we denote by ${\mathfrak C}$.

\subsubsection{Finitely-additive measures}

Let $\mathfrak{V}^+(X)$ be the family of finitely-additive real-valued measures\, $\Phi^+$ defined on the semi-ring $\mathfrak{C}^+$ and such that if\: $x, x'\in X$ satisfy $F(x_n)=F(x_n\!\!\!'\,)$, then $\Phi^+\!\left(\gamma_n^+(x)\right)\!=\!\Phi^+\!\left(\gamma_n^+(x')\right)$.

Given a measure\, $\Phi^+\!\in\mathfrak{V}^+(X)$,\:take\;\,$l\in\mathbb{Z}$\, and choose $m$ points $x(1), x(2),\\\ldots, x(m)\in X$ such that
$$F\left({(x(i))}_{l+1}\right)=i,\quad i=1,\ldots,m.$$
Introduce a vector $v^{(l)}\!\in\mathbb{R}^m$ by the formula
$$\left({v^{(l)}}_i\right)=\Phi^+\!\left(\gamma_{l+1}^+(x(i))\right)$$
By definition, the vectors $v^{(l)}$ satisfy the relation
 \begin{equation}
 v^{(l+1)}=A_lv^{(l)},\;l\in\mathbb{Z}.
 \label{eq-seq-def}
 \end{equation}
A sequence\, $v^{(l)},\;l\in\mathbb{Z}$,\; of real vectors satisfying ({\ref{eq-seq-def}}) will be called an {\it equivariant sequence} (with respect to the sequence of matrices $A_l$ or, as we shall sometimes say, with respect to the Markov compactum $X$).

Every finitely-additive measure $\Phi^+\!\in\mathfrak{V}^+(X)$ thus defines a unique equivariant sequence $v^{(l)},\,l\in\mathbb{Z}$. Introduce the {\it evaluation map}\quad $$eval^+_0:\mathfrak{V}^+(X)\longrightarrow\mathbb{R}^m$$
by the formula
 \begin{equation}
 \label{evalplus-def}
 eval^+_0(\Phi^+)=v^{(0)}
 \end{equation}
Conversely, to any equivariant sequence ${\bf v}=v^{(l)},\,l\in\mathbb{Z}$, we assign a finitely-additive measure $\Phi^+_{\bf v}\in\mathfrak{V}^+(X)$ by the following formula, valid for any $x\in X$ and $n\in\mathbb{Z}$:
$$\Phi^+_{\bf v}\left(\gamma_{n+1}^+(x)\right)={\left(v^{(n)}\right)}_{F(x_{n+1})}.$$
Note that if all matrices $A_n,\,n\in\mathbb{Z}$, are invertible, then the map $eval^+_0$ is an isomorphism.

Similarly, let $\mathfrak{V}^-(X)$ be the family of finitely-additive real-valued measures $\Phi^-$ on the semi-ring $\mathfrak{C}$ such that if $I(x_n)=I(x_n\!\!\!'\,)$, then
$$\Phi^-\!\left(\gamma_n^-(x)\right)=\Phi^-\!\left(\gamma_n^-(x')\right)$$
Again, we take $l\in\mathbb{Z}$, choose $m$ points $x(1), x(2),\ldots, x(m)\in X$ such that
$$I\left({(x(i))}_l\right)=i,\quad i=1,\ldots,m,$$
and to a measure $\Phi^-\!\in\mathfrak{V}^-(X)$ assign the vector $\tilde{v}^{(l)}$ by the formula $${\left(\tilde{v}^{(l)}\right)}_i=\Phi^-\!\left(\gamma^-_l(x(i))\right).$$
By definition, the vectors $\tilde{v}^{(l)}$ satisfy the relation
 \begin{equation}
 \label{reveq-seq-def}
 \tilde{v}^{(l)}=A_l^t\,\tilde{v}^{(l+1)},\;\;l\in\mathbb{Z}.
 \end{equation}
A sequence $\tilde{v}^{(l)},\;l\in\mathbb{Z}$, of real vectors satisfying the relation (\ref{reveq-seq-def}) will be called {\it reverse equivariant} sequence of vectors (with respect to the sequence $A_l$ or the Markov compactum X).

Every finitely-additive measure $\Phi^-\!\in\mathfrak{V}^-(X)$ thus defines a unique reverse equivariant sequence\; $\tilde{v}^{(l)},\;l\in\mathbb{Z}$. Introduce the evaluation map
$$eval^-_0:\mathfrak{V}^-(X)\longrightarrow\mathbb{R}^m$$
by the formula
 \begin{equation}
 \label{evalminus-def}
 eval^-_0(\Phi^-)=\tilde{v}^{(0)}
 \end{equation}
Conversely, to every reverse equivariant sequence $\tilde{\bf v}=\left(\tilde{v}^{(l)}\right), \,l\in\mathbb{Z}$, we assign the corresponding measure $\Phi^-_{\tilde{\bf v}}$ by the following formula, valid for all $x\in X$ and all $n\in\mathbb{Z}$:
$$\Phi^-_{\tilde{\bf v}}\!\left(\gamma^-_n(x)\right)=\left(v^{(n)}\right)_{I(x_n)}.$$
Again, if all matrices $A_l,\,l\in\mathbb{Z}$, are invertible, then the map $eval^-_0$ is an isomorphism.

Similarly, let the $\mathfrak{B}^-_c\subset\mathfrak{V}^-(X)$ be the subspace of finitely-additive measures $\Phi^-$ satisfying
$$\lim\limits^{}_{n\rightarrow\infty}\max\limits^{}_{x\in X}\left|
\Phi^-\!\left(\gamma^-_n(x)\right)\right|=0.$$

We introduce two subspaces of finitely-additive measures in
$\mathfrak{V}^+({X})$ having additional continuity
properties.

Let

$$\mathfrak{B}^+_c({X})=\left\{\Phi^+\in\mathfrak{V}^+({X}):
\lim\limits_{n\rightarrow\infty}
\max\limits_{x\in{X}}\left|\Phi^+\left(\gamma^+_{-n}(x)\right)\right|=0\right\},$$

$$\mathfrak{B}^+ ({{X}})=\biggl\{\Phi^+\in\mathfrak{V}^+({{X}}): \text{there exists}\ \alpha>0,C>0$$

$$\left.\text{such that } \max\limits_{x\in{X}}\left|\Phi^+\left(\gamma^+_{-n}(x)\right)\right|\leq Ce^{-\alpha n} \text{ for all } n\geq0\right\}.$$

{ \bf Remark}.\: Note that by holonomy-invariance the maximum
in both definitions is taken over a finite set of values.

Similarly, let $\mathfrak{B}^-_c\subset\mathfrak{V}^-(X)$ be the subspace of finitely-additive measures $\Phi^-$ satisfying
$$\lim\limits^{}_{n\rightarrow\infty}\max\limits^{}_{x\in X}\left|
\Phi^-\!\left(\gamma^-_n(x)\right)\right|=0,$$

and set

$$\mathfrak{B}^- ({{X}})=\biggl\{\Phi^-\in\mathfrak{V}^-({{X}}): \text{there exists}\ \alpha>0,C>0$$

$$\left.\text{such that } \max\limits_{x\in\mathbf{X}}\left|\Phi^-\left(\gamma^-_{n}(x)\right)\right|\leq Ce^{-\alpha n} \text{ for all } n\geq0\right\}.$$

 \subsubsection{Product measures and duality}

 Given $\Phi^+\in {\mathfrak V}^+(X)$, $\Phi^-\in {\mathfrak V}^-(X)$,
 introduce a finitely-additive measure $\Phi^+\times \Phi^-$
on the semi-ring ${\mathfrak C}(X)$ of cylinders in $X$ as follows:
 for $C\in {\mathfrak C}$ and $x\in  C$, let ${\tilde \gamma}^+$ be the largest by inclusion set of the 
 semi-ring ${\mathfrak C}^+$ containing $x$ and contained in $C$, let ${\tilde \gamma}^-$ be the largest by inclusion set of the
 semi-ring ${\mathfrak C}^-$ containing $x$ and contained in $C$, and 
 set
\begin{equation}
\label{prodmeasx}
\Phi^+\times \Phi^-(C)=\Phi^+({\tilde \gamma}^+)\cdot \Phi^-({\tilde \gamma}^-).
\end{equation}

Observe that, by holonomy-invariance of $\Phi^+$ and $\Phi^-$, the right-hand side does not depend on
the specific choice of $x\in C$ and the left-hand side is thus well-defined.

Introduce a pairing $\langle,\rangle$ between the spaces ${\mathfrak V}^+(X)$ and ${\mathfrak V}^-(X)$ by writing
\begin{equation}
\label{xduality}
\langle \Phi^+, \Phi^-\rangle=\Phi^+\times \Phi^-(X).
\end{equation}

If ${\bf v}=v^{(n)}$ is the equivariant sequence corresponding to $\Phi^+$,
 ${\tilde {\bf v}}={\tilde v}^{(n)}$  the reverse equivariant sequence corresponding to $\Phi^-$,
then we clearly have
\begin{equation}
\label{xdualitytriv}
\langle \Phi^+, \Phi^-\rangle=\sum\limits_{i=1}^m v^{(0)}_i {\tilde v}^{(0)}_i.
\end{equation}

 In particular, if all matrices $A_n,\,n\in\mathbb{Z}$, are invertible, then the pairing (\ref{xdualitytriv}) is nondegenerate on the pair of subspaces $\mathfrak{V}^+(X), \mathfrak{V}^-(X)$.

\subsubsection{Unique ergodicity.}

Assume that each space ${\mathfrak V}^+(X)$, ${\mathfrak V}^-(X)$ contains a unique positive measure up to scaling.
Assume furthermore that the positive measure $\nu^+\in {\mathfrak V}^+(X)$ satisfies
\begin{equation}
\lim\limits_{n\to\infty}\max_{x\in X} \nu^+(\gamma_{-n}^+(x))=0; \  \lim\limits_{n\to\infty}\min_{x\in X} \nu^+(\gamma_{n}^+(x))=\infty,
\end{equation}
while the positive measure $\nu^-\in {\mathfrak V}^-(X)$ satisfies
\begin{equation}
\lim\limits_{n\to\infty}\max_{x\in X} \nu^-(\gamma_{n}^-(x))=0; \  \lim\limits_{n\to\infty}\min_{x\in X} \nu^-(\gamma_{-n}^-(x))=\infty.
\end{equation}
The Markov compactum $X$ will then be called  {\it uniquely ergodic}.
Unique ergodicity can be equivalently reformulated as the following condition going back to
H. Furstenberg (see, e.g., formula $(16.13)$ in Furstenberg \cite{furst}).
\begin{assumption}
\label{uniquergodic}
\begin{enumerate}
\item For any $l\in{\mathbb Z}$, there exists a vector
$\lal=(\lal_1, \dots, \lal_m)$, all whose coordinates are positive,
such that
$\lal=A_l^t\la^{(l+1)}$ and
$$
\bigcap_{n\in{\mathbb N}} A^t_{l+1}\dots A^t_{l+n}{\mathbb R}_+^m={\mathbb R}_+\la^{(l)}.
$$

\item For any $l\in{\mathbb Z}$, there exists a  vector
$\hhl=(\hhl_1, \dots, \hhl_m)$, all whose coordinates are positive, such
that
$\hhl=A_lh^{(l-1)}$ and
$$
\bigcap_{n\in{\mathbb N}} A_{l-1}\dots A_{l-n}{\mathbb R}_+^m={\mathbb R}_+\hhl.
$$

\item $|\lal|\to 0$ as $l\to\infty$,  $|\hhl|\to 0$ as $l\to-\infty$.
\item $\min\limits_i \lal_i\to\infty$ as $l\to-\infty$, $\min\limits_i \hhl\to\infty$ as $l\to\infty$.
\end{enumerate}
\end{assumption}

The sequences of vectors $\lal$ and $\hhl$ are defined up to a multiplicative constant
(independent of $l$). The sequence $(\hhl)$ is equivariant while the sequence
$(\lal)$ is reverse equivariant.

To normalize the vectors $(\lal)$, $(\hhl)$ given by Assumption \ref{uniquergodic},
we write
\begin{equation}
\label{lahnorm}
|\la^{(0)}|=1, \langle \la^{(0)}, h^{(0)} \rangle =1.
\end{equation}

By equivariance we have
\begin{equation}
\langle \lal, \hhl \rangle =1 \ {\rm for \ all} \ l\in {\mathbb Z}.
\end{equation}

Denote by $\nu^+_{X}$ the measure corresponding to the
equivariant sequence $\hhl$; observe that it is a positive
sigma-finite sigma-additive measure on the sigma-algebra generated by
the semi-ring ${\C}^+$. Similarly, denote
by $\nu^-_{X}$ the measure corresponding
to the equivariant sequence $\lal$.
Furthermore, we define a probability measure $\nu_X$ on $X$ by the formula
$$
\nu_X=\nu^+_{X}\times \nu^-_{X}.
$$

\subsubsection{Weakly Lipschitz functions.}

Given a uniquely ergodic compactum $X$, we
introduce a function space  $Lip_w^+(X)$ in the following way.
A bounded Borel-measurable function $f:X \to {\mathbb R}$ belongs to the space $Lip_w^+(X)$ if there exists a
constant $C>0$ such that for all $n\geq 0$ and any $x, x^{\prime} \in X$ satisfying
$F(x_{n+1})=F(x^{\prime}_{n+1})$, we have
\begin{equation}
\label{weaklip}
\left|\int_{\gamma_n^+(x)}fd\nu^+  - \int_{\gamma_n^+(x^{\prime})} fd\nu^+\right|\leq C.
\end{equation}
Let $C_f$ be the infimum of all $C$ satisfying (\ref{weaklip}) and norm the space $Lip_w^+(X)$ by setting
$$
||f||_{Lip_w^+}=\sup_X |f|+C_f.
$$
As before, let $Lip_{w,0}^+(X)$ be the subspace of $Lip_w^+(X)$ of functions
whose integral with respect to $\nu$ is zero.

 \subsubsection{Vershik's orderings.}
 The r{\^o}le of Vershik's orderings in the symbolic coding of translation flows can  informally be summarized as follows.

 The Markov compactum and its asymptotic foliations encode the translation surface and its vertical and horizontal foliations; in order, however, to recover the translations flows themselves, one needs a linear ordering on the leaves of the foliations. This linear ordering is induced by Vershik's orderings of the edges of graphs defining the Markov compactum. We proceed to formal definitions.

 Let $\Gamma\in {\mathfrak G}$. Following S.~Ito \cite{ito}, A.M.~Vershik \cite{Vershik1, Vershik2},
assume that for  any $i\in\{1,\dots, m\}$ a linear ordering is given on the set
$$
\{e\in {\cal E}(\Gamma): I(e)=i\}.
$$
Such an ordering will be called  a {\it Vershik's ordering} on $\Gamma$.

Let $X$ be a uniquely ergodic Markov compactum corresponding to the sequence of graphs $\Gamma_l$.

If a Vershik's ordering is given on each $\Gamma_l$, $l\in {\mathbb Z}$, then  a linear ordering is
induced on any leaf of the foliation ${\cal F}^+_X$.
Indeed, if $x^{\prime}\in\gamma_{\infty}^+(x)$, $x^{\prime}\neq x$, then there exists $n$ such that $x_t=x^{\prime}_t$
for $t>n$ but $x_n\neq x^{\prime}_n$. Since $I(x_n)=I(x^{\prime}_n)$, the edges $x_n$ and $x^{\prime}_n$ are
comparable with respect to our ordering; if $x_n<x^{\prime}_n$, then we write $x<x^{\prime}$.
This ordering will be called a Vershik's ordering  on a Markov compactum $X$ and denoted $\oo$.

An edge will be called {\it maximal} (with respect to $\oo$)
if there does not exist a greater edge; {\it minimal}, if there does not exist a smaller edge;
and an edge $e$ will be called {\it the successor} of $e^{\prime}$
if $e>e^{\prime}$ but there does not exist $e^{\prime\prime}$ such that
$e>e^{\prime\prime}>e^{\prime}$.
We denote by $[x,x^{\prime}]$ the (closed) interval of points $x^{\prime\prime}$ satisfying
$x\leq x^{\prime\prime}\leq x^{\prime}$;  by $(x,x^{\prime})$ the (open) interval of points $x^{\prime\prime}$ satisfying
$x< x^{\prime\prime}< x^{\prime}$.

\subsection{Random Markov compacta}
\subsubsection{The space of Markov compacta}

Recall that ${\mathfrak G}$ is the space of oriented graphs on $m$ vertices, possibly with multiple edges, and
such that for any vertex there is an edge coming into it and going out of it.

Now let $\Omega$ be the space of sequences of bi-infinite sequences of graphs $\Gamma_n\in {\mathfrak G}$.
We write
$$
\Omega=\{\omega=\dots \omega_{-n}\dots\omega_n\dots, \omega_i\in {\mathfrak G}, i\in {\mathbb Z} \}.
$$
For $\omega\in\Omega$, we denote by $X(\omega)$ the Markov compactum corresponding to $\omega$.

As in the previous sections, for any $\omega$ the Markov compactum $X(\omega)$ carries a pair of foliations
$\F^+$, $\F^-$, the semi-rings ${\mathfrak C}^+$ and ${\mathfrak C}^-$ and so forth; to underline
the dependence  on $\omega$ we shall write $\F_{\omega}^+$, ${\mathfrak C}^+_{\omega}$ and so forth.

The right shift $\sigma$ on the  space $\Omega$ is defined by the
formula $(\sigma\omega)_n=\omega_{n+1}$.

\subsubsection{The renormalization cocycle.}

We have a natural cocycle ${\mathbb A}$ over the dynamical system
$(\Omega, \sigma)$
defined, for $n>0$, by the formula
$$
{\mathbb A}(n,\omega)=A(\omega_{n})\dots A(\omega_1).
$$
The cocycle ${\mathbb A}$ will be called the {\it renormalization cocycle}.

Let $\Omega_{inv}\subset \Omega$ be the subset of all sequences $\omega$ such that all matrices
$A(\omega_n)$ are invertible.

For $\omega\in\Omega_{inv}$ and $n<0$ set
$$
{\mathbb A}(n,\omega)=A^{-1}(\omega_{-n})\dots A^{-1}(\omega_0);
$$
set ${\mathbb A}(0,\omega)$ to be the identity matrix.


Let $\mu$ be an ergodic $\sigma$-invariant probability measure on $\Omega$ satisfying the following
\begin{assumption}
\label{asos}
\begin{enumerate}
\item  There exists $\Gamma_0 \in {\mathfrak G}$ such that all entries of the
matrix $A(\Gamma_0)$ are positive and that
$$
\mu(\{\omega:\omega_0=\Gamma_0\})>0.
$$
\item The matrices $A(\omega_n)$ are almost surely invertible with respect to $\mu$.
\item The logarithm of the norm of the renormalization cocycle,
as well as that of its inverse, is integrable with respect to $\mu$.
\end{enumerate}
\end{assumption}

Our assumptions on the measure $\mu$ imply the existence of a matrix $Q$, all whose entries are positive, such that for $\mu$-almost every $\omega\in\Omega$, the sequence of matrices $A(\omega_n),\;n\in\mathbb{Z}$, contains infinitely many occurrences of the matrix $Q$ both in the past and in the future. It follows that for $\mu$-almost every $\omega\in\Omega$ the Markov compactum $X(\omega)$ is uniquely ergodic.

Furthermore, for $\mu$-almost every $\omega\in\Omega$, let $E^u_{\omega}$ be the strictly expanding Oseledets subspace of the renormalization cocycle $\mathbb{A}$ at the point $\omega$.

\begin{proposition}
For $\mu$-almost all $\omega\in\Omega$,\, the map $eval^+_0$ induces an isomorphism between the spaces $\mathfrak{B}^+(X(\omega))$ and $E^u_{\omega}$.
\end{proposition}
Proof. This is  a direct corollary of the Oseledets Multiplicative Ergodic Theorem (see \cite{pesinbarreira} and Appendix A).

Now assume that for $\mu$-almost every $\omega\in\Omega$ a Vershik's ordering $\mathfrak{o}(\omega)$ is given on the edges of each graph $\omega_n,\;n\in\mathbb{Z}$; assume also that the ordering $\mathfrak{o}(\omega)$ is $\sigma$-invariant in the sense that the ordering $\mathfrak{o}(\omega)$ on the edges of the graph $\omega_{n+1}$ is the same as the ordering $\mathfrak{o}(\sigma\omega)$ on the edges of the graph $(\sigma\omega)_n=\omega_{n+1}$.

As before, we consider the induced linear ordering on the leaves of the asymptotic foliations $\mathfrak{F}^+_{\omega}$, and we use notation $[x,x']$ for closed intervals, with respect to this linear ordering, $(x',x")$ for open intervals, and so on.

\subsubsection{Approximation of weakly Lipschitz functions}

For $\omega\in\Omega$, let $\mathfrak{C}^+_{\omega}(\mathfrak{o})$ be the semi-ring of arcs of the form
$[x, x^{\prime})$, $(x, x^{\prime}]$, $[x, x^{\prime}]$, $(x, x^{\prime})$. By definition, $\mathfrak{C}^+_{\omega}(\mathfrak{o})\supset\mathfrak{C}^+_{\omega}$.

\begin{proposition}
For $\mu$-almost all $\omega\in\Omega$, and any $\Phi^+\in\mathfrak{B}^+(X(\omega))$, the finitely-additive measure $\Phi^+$ admits a unique continuation to the semi-ring $\mathfrak{C}^+(\mathfrak{o})_{\omega}$ such that the function $\Phi^+\left([x',x")\right)$ is continuous both in $x'$ and $x"$.
\end{proposition}

We are now ready to formulate the symbolic analogue of the Approximation Theorem \ref{multiplicmoduli}.

\begin{proposition}
\label{prelim-multiplic}
For any $\varepsilon>0$ there exists a positive measurable function $C_{\varepsilon}:\Omega\rightarrow\mathbb{R}_{>0}$ such that the following holds. For $\mu$-almost every $\omega\in\Omega$ there exists a continuous mapping $\Xi^+_{\omega}:Lip^+_w(X(\omega))\rightarrow\mathfrak{B}^+(X(\omega))$ such that for any $f\in Lip^+_w(X(\omega))$, any $x',x"\in X(\omega)$ satisfying $x^{\prime}<x^{\prime\prime}$ we have
$$\left|\,\int\limits^{}_{[x^{\prime},x^{\prime\prime}]}\!\!f\,d\nu^+\,-\:\Xi^+_{\omega}\left(f;[x^{\prime},x^{\prime\prime}]\right)\right|\leq C_{\varepsilon}(\omega)\cdot{\|f\|}_{Lip^+_w}\cdot\left(1\!+\!\nu^+\!\left([x^{\prime},x^{\prime\prime}]\right)\right) ^{\varepsilon}.$$
\end{proposition}

\subsubsection{The transpose cocycle and duality}
The map $\Xi_{\omega}^+$ of Proposition \ref{prelim-multiplic} admits an
explicit description in terms of the ${\it duality}$ between
H\"{o}lder cocycles on the asymptotic foliations corresponding to
the past and to the future of our Markow compactum. This duality
in one of the central constructions of the paper.

\begin{proposition}
\label{dual-rmc}
For $\mu$-almost every
$\omega\in\Omega$ the pairing $\left< \right>$ of Section 1.9.4is nondegenerate on the pair of
subspaces $\mathfrak{B}^+({{X}}(\omega))$ and
$\mathfrak{B}^-({{X}}(\omega)).$
\end{proposition}

Using this duality, we now obtain

\begin{proposition}
\label{xi-identif-rmc}
 For $\mu-almost$ every
$\omega\in\Omega,$ any $f\in Lip_w^+({{X}}(\omega)),$ and any
$\Phi^-\in\mathfrak{B}^-(X(\omega)),$ we have
\begin{equation}
\label{dual-first}
\left
<\Xi_{\omega}^+(f),\Phi^-\right>=\int_{X(\omega)}fd\nu^+\times\Phi^-.
\end{equation}
\end{proposition}

In view of the previous proposition, the relation (\ref{dual-first}) determines
$\Xi_{\omega}^+(f)$ uniquely.

The proof of Proposition \ref{dual-rmc}
proceeds as follows.

In the same way with the identification of the space
$\mathfrak{B}^+(X(\omega))$ with the strictly unstable
Lyapunov subspace of the renormalization cocycle $\mathbb{A},$ the
space $\mathfrak{B}^-({{X}}(\omega))$ is identified with the
strictly unstable Lyapunov subspace of the ${\it transpose}$
cocycle $\mathbb{A}^t$ defined in the following way.

{\bf Definition.} The {\it transpose} cocycle  ${\mathbb A}^t$ over the dynamical system $(\Omega, \sigma^{-1}, \Prob)$
is defined, for $n>0$, by the formula
$$
{\mathbb A}^t(n,\omega)=A^t(\omega_{1-n})\dots A^t(\omega_0).
$$

If $\omega\in\Omega_{inv}$, then for $n<0$ write
$$
{\mathbb A}^t(n,\omega)=(A^t)^{-1}(\omega_{-n})\dots (A^t)^{-1}(\omega_1).
$$
and set ${\mathbb A}^t(0,\omega)$ to be the identity matrix.

For $\mu$-almost every $\omega\in\Omega,$ let $\widetilde{E}^u_w$
be the strictly expanding  Oseledets subspace of the transpose
cocycle $\mathbb{A}^t$ at the point $\omega.$

The Oseledets Multiplicative Ergodic theorem again immediately implies the following
\begin{proposition} For $\mu$-almost all $\omega$
the map $eval_0^-$ induces an isomorphism between the spaces
$\mathfrak{B}^-({\bf{X}}(\omega))$ and $\widetilde{E}^u_{\omega}.$
\end{proposition}

Furthermore,  the Oseledets Multiplicative Ergodic Theorem implies that the
standard inner product on $\mathbb{R}^m$ induces a non-degenerate
pairing between the subspaces $E^u_{\omega}$ and $\widetilde{E}^u_{\omega}$ for
$\mu$-almost every $\omega\in\Omega.$

Proposition \ref{dual-rmc} is proved.
Proposition \ref{xi-identif-rmc} will be
proven in the following Section.

In the following Section, we shall construct a special flow $h_t^+$ on a uniquely ergodic
Markov compactum $X$ endowed with a Vershik's ordering ${\mathfrak o}$
in such a way that flow arcs for $h_t^+$ are precisely  intervals of the form $[x^{\prime},x^{\prime\prime}]$.
Proposition \ref{prelim-multiplic} will then yield an asymptotic expansion of ergodic integrals for such symbolic flows.

The right shift $\sigma$ on the space $\Omega$ of Markov compacta naturally intertwines the spaces $\BB^{\pm}_{\omega}$
and $\BB^{\pm}_{\sigma\omega}$; the action of the shift on the corresponding equivariant and reverse equivariant
sequences is then governed by the Oseledets Multiplicative Ergodic Theorem applied to the cocycles ${\mathbb A}$, ${\mathbb A}^t$. This renormalization action of the shift will play a key r{\^o}le in the proof of limit theorems for our symbolic flows and, consequently, for translation flows.

{\bf Acknowledgements.}
W. A.~Veech made the suggestion that G. Forni's invariant distributions for the vertical flow
should admit a description in terms of cocycles for the horizontal flow, and I am greatly
indebted to him. G.~Forni pointed out that cocycles are dual objects
to invariant distributions and suggested the analogy with F. Bonahon's work \cite{bonahon1};
I am deeply grateful to him.
I am deeply grateful to A.~Zorich for his kind explanations
of Kontsevich-Zorich theory and for discussions of the relation
between this paper and F.~Bonahon's work \cite{bonahon1}.
I am deeply grateful to H.~Nakada who pointed out the reference
to S.~Ito's work \cite{ito} and to B. ~Solomyak who pointed out the reference
to the work \cite{dumontkamae} of P.~Dumont, T.~Kamae and S.~Takahashi and the work \cite{kamae} of T.~Kamae.

I am deeply grateful to J.~Chaika, P.~Hubert, Yu.S.~Ilyashenko, H.~Kr{\"u}ger,
E.~Lanneau, S.~Mkrtchyan and R.~Ryham for many helpful suggestions on improving the presentation.
I am deeply grateful to A.~Avila, X.~Bressaud, B.M.~Gurevich, A.B.~Katok, A.V.~Klimenko, S.B. Kuksin, C.~McMullen,  V.I.~Oseledets,
Ya.G.~Sinai, I.V.~Vyugin, J.-C.~Yoccoz for useful discussions. I am deeply grateful to
N. Kozin, D.~Ong, S. Sharahov and R. Ulmaskulov for typesetting
parts of the manuscript. Part of this paper was written while I was visiting the
Max Planck Institute of Mathematics in Bonn.
During the work on this paper,
I was supported in part by the National
Science Foundation under grant DMS~0604386, by the Edgar Odell Lovett Fund at Rice University
and by the Programme on Mathematical Control Theory of the Presidium of the Russian Academy of Sciences.

\section{The Symbolic Approximation Theorem}
\subsection{The Preparatory Approximation Lemma}
\subsubsection{The case of one-sided Markov compacta}
Our first approximation lemma only requires the structure of a {\it
one-sided} Markov compactum.

In the same way as above, to a one-sided infinite sequence
$\Gamma_n\in\mathfrak{G},n\in\mathbb{N}$ we assign the
one-sided Markov compactum
$$Y=\{y=y_1{\dots}y_n{\dots},\ y_n\in\mathcal{E}(\Gamma_n),\
I(y_n)=F(y_{n+1})\}.$$

In the same way as above, for $y\in Y$ denote
$$\gamma_n^+(y)=\{y'\in Y:y_t'=y_t\ \mathrm{for}\ t \geqslant n\}$$ $$\gamma_{\infty}^+(y)=\bigcup^{\infty}_{n=0}\gamma^+_n(y).$$

The sets $\gamma^+_n(y)$ are finite, and the set $\gamma_{\infty}^+(y)$ is
countable.

Let $Lip_w^+(Y)$ be the space of functions
$\varphi:Y\rightarrow\mathbb{R}$ admitting a positive constant
$C>0$ such that for any $n\in\mathbb{N},$ and any $y',y''\in Y$
satisfying $F(y_n')=F(y_n'')$ we have
$$\left|\sum_{y\in\gamma^+_n(y')}\varphi(y)\ -\sum_{y\in\gamma^+_n(y'')}\varphi(y)\right|\leqslant C.$$

Let $C_{\varphi}$ be the smallest possible constant such that the
above formula is valid and norm $Lip_w^+(Y)$ by setting
$$||\varphi||_{Lip_w^+}=C_{\varphi}+\sup_{y\in Y}|\varphi(y)|.$$

A particular case of weakly Lipschitz functions is given by {\it
piecewise-constant functions} defined as follows.

Let $v\in\mathbb{R}^m,v=(v_1,\ldots,v_m),$ and set
$$\Phi_v(y)=v_{F(y_1)}.$$

The next lemma shows that sums of weakly Lipchitz functions can be
approximated, up to a subexponential error, by piecewise constant
functions. We shall need the following assumption on the matrices
$A_n,n\geqslant0.$

\begin{assumption}\label{asexpdec} There exists $\alpha>0$ and, for
every $n\geqslant1$, a direct-sum decomposition
$$\mathbb{R}^m=E^u_n\oplus E^n_{cs}$$ satisfying the following.
\begin{enumerate}
\item $A_nE_n^u=E^u_{n+1}$ and $A_n|_{E^u_n}$ is injective.

\item $A_nE_n^{cs}\subset E_{n+1}^{cs}.$

\item  For any $\varepsilon>0$ there exists $C_{\varepsilon}>0$ such
that for any $n\geqslant 1,\ k\geqslant0$ we have
$$\left|\left|(A_{n+k}\cdot\ldots\cdot A_n)^{-1}|_{E^u_{n+k+1}}\right|\right|\leqslant
C_{\varepsilon}e^{\varepsilon n-\alpha k}$$

$$\left|\left|A_{n+k}\cdot\ldots\cdot A_n|_{E^{cs}_{n}}\right|\right|\leqslant
C_{\varepsilon}e^{\varepsilon(n+k)}.$$

\end{enumerate}
\end{assumption}

\begin{lemma} \label{thetaxapprox}
Let $Y$  be a one-sided Markov
compactum whose sequence of adjacency matrices $A_n,n\geqslant 1,$
satisfies Assumption \ref{asexpdec}.

Then there exists a continuous map
$$\Xi^+:Lip_w^+(Y)\rightarrow E_0^u$$

and for any $\varepsilon>0,$ a constant
$\widetilde{C}_{\varepsilon}$ such that for any $\varphi\in
Lip_w^+(Y),$ any $n\in\mathbb{N}$ and any $y'\in Y$ we have
$$\left|\sum_{y\in\gamma^+_n(y')}(\varphi(y)-\Phi_{\Xi^+(\varphi)}(y))\right|\leqslant
{\tilde C}_{\varepsilon}\cdot||\varphi||_{Lip_w^+}\cdot e^{\varepsilon
n}.$$
\end{lemma}

First we prepare
\begin{lemma}
\label{vngeneral}
Let $A_n$ be a sequence of matrices satisfying Assumption \ref{asexpdec}.
Let $v_1, \dots$ be a sequence of vectors such that for any $\varepsilon>0$ a constant
$C_{\varepsilon}$   can be chosen in such a way that for all $n$ we have
$$
|A_nv_n-v_{n+1}|\leq C_{\varepsilon}\exp(\varepsilon n).
$$
Then there exists a unique vector $v\in E_1^u$ such that
$$
|A_n\dots A_1v-v_{n+1}|\leq C^{\prime}_{\varepsilon}\exp(\varepsilon n).
$$
\end{lemma}

Proof: Denote $u_{n+1}=v_{n+1}-A_nv_n$ and
decompose $u_{n+1}=u_{n+1}^+ + u_{n+1}^-$, where
$u_{n+1}^+\in E_{n+1}^u$, $u_{n+1}^-\in E_{n+1}^{cs}$.
Let
$$
v_{n+1}^+=u_{n+1}^+ + A_{n} u_n^+ +
A_{n}A_{n-1}u_{n-1}^+ + \dots + A_{n}\dots A_1u_1^+;
$$
$$
v_{n+1}^-=u_{n+1}^- + A_{n} u_n^- +
A_{n}A_{n-1}u_{n-1}^- + \dots + A_{n}\dots A_1u_1^-.
$$
We have $v_{n+1}\in E_{n+1}^u$, $v_{n+1}^-\in E_{n+1}^{cs}$,
$v_{n+1}=v_{n+1}^+ + v_{n+1}^-$.
Now introduce a vector
$$
v=u_1^+ + A_1^{-1}u_1^+ + \dots + (A_n\dots A_1)^{-1}u_{n+1}^+ +\dots
$$

By our assumptions, the series defining $v$ converges exponentially fast
and, moreover, we have
$$
|A_n\dots A_1v-v_{n+1}^+|\leq C^{\prime}_{\varepsilon}\exp(\varepsilon n)
$$
for some constant $C_{\varepsilon}^{\prime}$.

Since, by our assumptions we also have $|v_{n+1}^-|\leq
C_{\varepsilon}\exp(\varepsilon n)$, the Lemma is proved completely.

Uniqueness of the vector $v$ follows from the fact that, by our assumptions,
for any ${\tilde v}\neq 0$, ${\tilde v}\in E^{u}_0$ we have
$$
|A_n\dots A_1{\tilde v}|\geq C^{\prime\prime} \exp(\alpha n).
$$

We proceed to the proof of Lemma \ref{thetaxapprox}.

Let $i=1, \dots, m$ and take arbitrary points $y(i)\in Y$ in such a way
that $$
F(y(i)_1)=i.
$$
Introduce a sequence of vectors $v(n)\in {\mathbb R}^m$ by the formula
$$
(v(n))_i=\sum\limits_{y\in \gamma_{n+1}^+(y(i))} \varphi(y).
$$
By the Lipschitz property of $\varphi$, for any $y\in Y$ we have
$$
\left|\sum\limits_{{\tilde y}\in \gamma_{n+1}^+(y)} \varphi({\tilde y})-v(n)_{F(y_1)}\right|\leq C,
$$
whence by additivity it also follows that
$$
\left|A_nv(n)-v(n+1)\right|\leq C_{\varepsilon}\exp(\varepsilon n).
$$

Lemma \ref{thetaxapprox} follows now from Lemma \ref{vngeneral}.

\subsubsection{Duality}
We go back to two-sided Markov compacta. Let ${X}$ be a Markov
compactum corresponding to the bi-infinite sequence of graphs
$\Gamma_n,n\in\mathbb{Z},$ and let $Y$ be the one-sided Markov
compactum corresponding to the graphs $\Gamma_n,n\geqslant1.$

We have a natural forgetting map $\Pi_Y^{\bf{X}}$ which to a
bi-infinite path $(x_n),n\in\mathbb{Z},$ assigns the one-sided
path $(x_n),n\geqslant1.$

A finitely-additive measure $\Theta$ defined on the semi-ring
$\mathfrak{C}_0^+({X})$ will be called ${\it weakly\ Lipschitz}$
if the function $\varphi:Y\to\mathbb{R}$ defined by the formula
$$\varphi(y)=\Theta(\gamma_1^+(x)),\ \ \ x\in(\Pi_Y^{{X}})^{-1}y,$$ is
weakly Lipschitz.

Note that the function $\varphi$ is well-defined, since the set
$\gamma_1^+(x)$ does not depend on a specific choice of
$x\in(\Pi_Y^{{X}})^{-1}y.$ We denote by $LipMeas^+({{X}})$
the space of weakly Lipschitz measures on ${X}.$

The function $\varphi$ determines the measure $\Theta$ uniquely,
and we norm the space $LipMeas^+({{X}})$ by setting
$$||\Theta||_{LipMeas^+}=||\varphi||_{Lip_w^+}.$$

A pairing between the spaces $LipMeas^+({X})$ and
$\mathfrak{B}^- ({X})$ can be introduced as follows. Take
$\Theta\in LipMeas^+({X})$ and let
$\Phi^-\in\mathfrak{B}^- ({X})$ be defined by the reverse
equivariant sequence of vectors
$\tilde{v}^{(n)},\,n\in\mathbb{Z}$.

For any $n\in\mathbb{N}$, choose $m$ points
$x^{(n)}(1),\ldots,x^{(n)}(m)\in {X}$ satisfying
$$F\left(\left(x^{(n)}(i)\right)_n\right)=i,\,i=1,\ldots,m.$$

Consider the Riemann sum
$$S_{Riem}\left(\Theta,\Phi^-,x^{(n)}(1),\ldots,x^{(n)}(m)\right)=
\sum\limits^m_{i=1}\Theta\left(\gamma^+_n\left(x^{(n)}(i)\right)\right)\cdot \tilde{v}^{(n-1)}_i\,,$$
and set
$$\langle\Theta,\Phi^-\rangle=\lim\limits_{n\rightarrow\infty}S_{Riem}\left(\Theta,\Phi^-,x^{(n)}(1),\ldots,x^{(n)}(m)\right).$$

Existence of the limit and its independence of the particular
choice of the points $x^{(n)}(1),\ldots,x^{(n)}(m)$ are clear from
the Lipschitz property of $\Theta$ and the exponential decay of
norms of the vectors $\tilde{v}^{(n)}$. It is also clear that if
$\Theta\in\mathfrak{V}^+({X})$, then the new definition of
pairing is consistent with the previous one.

Now let ${X}$ be a bi-infinite Markov compactum whose
adjacency matrices $A_n,\,n\geqslant0$, satisfy Assumption \ref{asexpdec}.
For $\Theta\in LipMeas^+({X})$, let $v\in
E^u_0$ be given by the equality
$$\Phi_{{v}}=\Xi^+\left(\Theta\right)\,.$$

From the definitions it is immediate that for any
that $\Phi^-\in\mathfrak{B}^- $, we have

\begin{equation}
\label{dualprelim}
\langle\Theta,\Phi^-\rangle=\langle \Xi^+(\Theta),\Phi^-\rangle.
\end{equation}

Indeed, write $\Phi^-=\Phi^-_{\mathbf{\tilde{v}}}$, where $\tilde{\mathbf{v}}=\tilde{v}^{(n)}$ is a
reverse equivariant sequence. By definition,
$$\langle\Theta,\Phi^-\rangle=\sum\limits^m_{i=1}\Xi^+\left(\Theta, \gamma^+_n\left(x^{(n)}(i)\right)\right)\cdot \tilde{v}^{(n-1)}_i.$$

Consequently, since $\Phi^-\in\mathfrak{B}^- $, we have
$$\left|\langle\Theta -\Xi^+(\Theta),\Phi^-\rangle\right|\leq $$
$$\leq \sum\limits^m_{i=1}\left|\Theta\left(\gamma^+_n\left(x^{(n)}(i)\right)\right)-\Xi^+\left(\Theta, \gamma^+_n\left(x^{(n)}(i)\right)\right)\right|\cdot \left|\tilde{v}^{(n-1)}_i\right|\leq \exp(-\alpha n)$$
for some positive $\alpha$, and (\ref{dualprelim}) is established.

In the study of random Markov compacta we shall be mainly
concerned with the case when both the sequence of matrices
$A_n,\,n\in\mathbb{Z}$, and the sequence of transpose matrices
$A_n^t,\,n\in\mathbb{Z}$, admit a decomposition into the expanding
and the central part similar to that given by Assumption \ref{asexpdec}.
More precisely, we formulate the following
assumption on a sequence of $m\times m$-matrices
$A_n,\,n\in\mathbb{Z}$.

\begin{assumption}\label{asdualexp} For any $n\in\mathbb{Z}$ there
exist direct-sum decompositions
$$\mathbb{R}^m=E^u_n\oplus E^{cs}_n\,,$$
$$\mathbb{R}^m=\widetilde{E}^u_n\oplus \widetilde{E}^{cs}_n\,$$
such that the following holds.
\begin{enumerate}

\item \quad $A_n E^u_n = E^u_{n+1}$, \;$A_n^t \widetilde{E}^u_{n+1} =
\widetilde{E}^u_n$, and $\left.A_n\right|_{E^u_n}$ and
$\left.A_n^t\right|_{\widetilde{E}^u_{n+1}}$ are injective for all
$n\in\mathbb{Z}$;

\item \quad $A_n E^{cs}_n \subset E^{cs}_{n+1}$, \;$A_n^t
\widetilde{E}^{cs}_{n+1} \subset \widetilde{E}^{cs}_n$ for all
$n\in\mathbb{Z}$;

\item \quad $\widetilde{E}^{cs}_n=Ann\left(E^u_n\right),
\widetilde{E}^u_n=Ann\left(E^{cs}_n\right)$ for all
$n\in\mathbb{Z}$;

\item \quad there exists $\alpha>0$ and, for every $\varepsilon>0$, a
positive constant $C_{\varepsilon}$ such that for all
$k\in\mathbb{N},\,n\in\mathbb{Z}$ we have
$$\left\|\left.\left(A_{n+k}\ldots A_n\right)^{-1}\right|_{E^u_{n+k+1}}\right\|\leqslant C_{\varepsilon}e^{\varepsilon|n|-\alpha k}\,;$$
$$\left\|\left.A_{n+k}\ldots A_n\right|_{E^{cs}_n}\right\|\leqslant C_{\varepsilon}e^{\varepsilon(|n|+k)}\,;$$
$$\left\|\left.\left(A_{n}^t\ldots A_{n+k}^t\right)^{-1}\right|_{\widetilde{E}^u_n}\right\|\leqslant C_{\varepsilon}e^{\varepsilon|n|-\alpha k}\,;$$
$$\left\|\left.A_{n}^t\ldots A_{n+k}^t\right|_{\widetilde{E}^{cs}_{n+k+1}}\right\|\leqslant C_{\varepsilon}e^{\varepsilon(|n|+k)}\,;$$
\end{enumerate}
\end{assumption}

Assumption \ref{asdualexp} implies that to each vector $v\in
E^u_0$ there corresponds a unique equivariant sequence ${\bf v}$ and, consequently, a unique
finitely-additive measure $\Phi^+_{{\bf v}}\in\mathfrak{V}^+({X})$ which, furthermore,
satisfies $\Phi^+_{{\bf v}}\in\mathfrak{B}^+ ({X})$; while, similarly, to
each vector $\tilde{v}\in\widetilde{E}^u_0$ there corresponds a
unique finitely-additive measure
$\Phi^-_{\mathbf{\tilde{v}}}\in\mathfrak{V}^-({X})$ which,
furthermore, satisfies
$\Phi^-_{\mathbf{\tilde{v}}}\in\mathfrak{B}^- ({X})$.
Finally, Condition 3 of Assumption \ref{asdualexp} implies that
the pairing $\langle,\rangle$ is non-degenerate on the pair of
subspaces $\mathfrak{B}^+ ({X}),
\mathfrak{B}^- ({X})$. Consequently, we arrive at the
following Preparatory Approximation Lemma.

\begin{lemma} \label{prepapprox}
Let $X$ be a
bi-infinite Markov compactum whose sequence
$A_n,\,n\in\mathbb{Z}$, of a adjacency matrices satisfies
Assumption \ref{asdualexp}. Then there exists a continuous mapping
$$\Xi^+\:: LipMeas^+({X})\rightarrow\mathfrak{B}^+ ({X})$$
and, for any $\varepsilon>0$, a positive constant
$C_{\varepsilon}$ such that for any $x\in{X}$, any
$n\in\mathbb{N}$ and any $\Theta\in LipMeas^+({X})$ we have
$$\left|\Theta\left(\gamma^+_n(x)\right)-\Xi^+\left(\Theta;\gamma^+_n(x)\right)\right|\leq C_{\varepsilon}\cdot\|\Theta\|_{LipMeas^+}\cdot e^{\varepsilon n}$$

The map $\Xi^+$ is uniquely determined by the requirement that for
any $\Theta\in LipMeas^+({X})$ and any
$\Phi^-\in\mathfrak{B}^- ({X})$ we have
$$\langle\Theta,\Phi^-\rangle=\langle\Xi^+(\Theta),\Phi^-\rangle.$$
{\bf Remark} Above, we write
$\Xi^+\left(\Theta,\gamma^+_n(x)\right)$ instead of
$\Xi^+(\Theta)\left(\gamma^+_n(x)\right).$
\end{lemma}

\subsection{Extension of measures.}

\subsubsection{The ring of well-approximable arcs.}
Let ${\R}_n^+$ be the ring generated by the semiring $\C_n^+$.
For an arbitrary subset $A$ of a leaf $\gamma_{\infty}^+$ of the foliation $\F^+$, let
${\hat \gamma}_n^+(A)$ be the minimal (by inclusion) element of the ring ${\R}_n^+$ containing $A$ and similarly
let ${\check \gamma}_n^+(A)$ be the maximal (by inclusion) element of the ring ${\R}_n^+$ contained in $A$.
The set difference ${\hat \gamma}_n^+(A)\setminus {\check \gamma}_n^+(A)$ can be represented in a unique way as
 a finite union of elements of the semi-ring $\C_n^+$ (that is, of arcs of the type $\gamma_n^+(x)$).
The number of these arcs is denoted $\delta_n^+(A)$.
We say that a subset $A$ of a leaf $\gamma_{\infty}^+$  of the foliation $\F^+$ belongs to
the family  ${\overline \R}^+$ if for any $\varepsilon>0$ there exists a constant $C_{\varepsilon}$
such that we have
$$
\delta_{-n}^+(A)\leq C_{\varepsilon} \exp(\varepsilon n).
$$

\begin{proposition}
The family  ${\overline \R}^+$ is a ring.
\end{proposition}

Indeed, this is clear from the inclusion
\begin{equation}
\label{hatcheckinclusion}
{\hat \gamma}_n^+(A\cup B)\setminus {\check \gamma}_n^+(A\cup B)\subset \left({\hat \gamma}_n^+(A)\setminus {\check \gamma}_n^+(A)
\right) \bigcup \left({\hat \gamma}_n^+(B)\setminus {\check \gamma}_n^+(B)
\right)
\end{equation}
and similar inclusions for the intersection and difference of sets.

\begin{proposition}
Let $X$ be a Markov compactum satisying Assumption \ref{asexpdec}.
Then any element $\Phi^+\in {\mathfrak B}^+(X)$ extends to a finitely-additive measure on the
ring ${\overline \R}^+$.
\end{proposition}

Indeed, if $A\in {\overline \R}^+$, then the set
${\hat \gamma}_{-n}^+(A)\setminus {\check \gamma}_{-n}^+(A)$ is a union of elements of the
semi-ring ${\C}_{-n}^+$, whose number grows subexponentially as $n\to\infty$.

Thus, if the Markov compactum is Lyapunov regular and
$\Phi^+\in {\mathfrak B}^+$,
then the quantities
\begin{equation}
\Phi_{{\bf v}}^+({\hat \gamma}_{-n-1}^+(A)\setminus {\hat \gamma}_{-n}^+(A))
\end{equation}
\begin{equation}
\Phi_{{\bf v}}^+({\hat \gamma}_{-n}^+(A)\setminus {\check \gamma}_{-n}^+(A))
\end{equation}
decay exponentially fast as $n\to\infty$.

We therefore set
$$
\Phi^+(A)=\lim\limits_{n\to\infty} \Phi^+({\hat \gamma}_{-n}^+(A))=
\lim\limits_{n\to\infty} \Phi^+({\check \gamma}_{-n}^+(A)).
$$

The resulting extension is finitely additive: indeed, if $A, B\in {\overline {\cal R}}^+$ are disjoint,
then
$$
{\check \gamma}_n^+(A)\bigsqcup{\check \gamma}_n^+(B)\subset {\check \gamma}_n^+(A\bigsqcup B);
$$
$$
{\hat \gamma}_n^+(A\bigsqcup B)\subset {\hat \gamma}_n^+(A)\bigcup{\hat \gamma}_n^+(B).
$$

Moreover, the set-difference between the right-hand side and the left-hand side in
both the above inclusions
consists of a finite number of arcs, which, by definition of ${\overline \R}^+$,
grows subexponentially with $n$. Since, by definition,
the value of $\Phi^+$ on each arc decays exponentially, we obtain
that the quantity
$$
\Phi^+( {\hat \gamma}_n^+(A\bigsqcup B)\setminus {\hat \gamma}_n^+(A)\bigcup{\hat \gamma}_n^+(B)).
$$
decays exponentially, and finite additivity is
established.

The statement of the Preparatory Approximation Lemma \ref{prepapprox}  can
also be extended to arcs from the ring ${\overline \R}^+$. For $\gamma\in {\overline \R}^+$, let $n^+(\gamma)$
be the largest integer $n$ such that $\gamma$ contains a nonempty arc from ${\mathfrak C}_n^+$.
We also need the following

{\bf{Definition}}. A sequence of nonnegative
$m\times m$ matrices $A_n,n\in\mathbb{Z},$ is said to have
sub-exponential growth if for any $\varepsilon>0$ there exists a
constant $C_{\varepsilon}$ such that for all $n\in\mathbb{Z}$ we
have
$$\sum_{i=1}^m\sum_{j=1}^m(A_n)_{ij}<C_{\varepsilon}e^{\varepsilon
|n|}.$$
We are now ready to formulate the approximation of Lipschitz measures on arcs from the ring ${\overline \R}^+$.

\begin{corollary}
\label{approxrplus}
Let $X$ be a Markov compactum that satisfies Assumption \ref{asdualexp} and whose sequence of adjacency matrices has subexponential growth. Let $\Xi^+$ be the map given by Lemma \ref{prepapprox}. For any $\varepsilon>0$ there exists
a positive constant $C_{\varepsilon}$ such that for any  $\gamma\in {\overline \R}^+$ and $\Theta\in LipMeas^+(X)$
we have
$$
|\Theta(\gamma)-\Xi^+(\Theta; \gamma)|\leq C_{\varepsilon} ||\Theta||_{LipMeas^+}e^{\varepsilon n^+(\gamma)}.
$$
\end{corollary}
Proof. For arcs from the semiring ${\mathfrak C}_n^+$ this is precisely the statement of Lemma \ref{prepapprox}.
Approximating an arc $\gamma\in {\overline \R}^+$ by arcs from ${\mathfrak C}_n^+$ and summing the resulting geometric
series, we arrive at the result of the corollary.

\subsubsection{Extension to arcs under Vershik's ordering.}

In particular, let $\bf {X}$ be a Markov compactum endowed with a
Vershik's ordering $\mathfrak{o}.$
Our aim is to extend
finitely-additive measures from the $\mathfrak{B}^+ $ to the
ring generated by the class intervals
$[x,x'],(x,x'],[x,x'),(x,x')$ with respect to the ordering
$\mathfrak{o}.$
Let ${\mathfrak C}({\mathfrak o})$ be the semi-ring of arcs of the form
$[x,x'],(x,x'],[x,x'),(x,x')$ and
${\mathcal R}({\mathfrak o})$ be the ring of sets generated by the
the semi-ring ${\mathfrak C}({\mathfrak o})$.

We shall now see that a sufficient condition for the existence of
an extension of the measures from ${\mathfrak B}^+$ is sub-exponential growth of the matrices $A_n.$

\begin{proposition} \label{measuretoorderarc}
Let $X$ be a Markov
compactum endowed with a Vershik's ordering $\mathfrak{o}.$ If
the sequence $A_n,n\in\mathbb{Z},$ of adjacency matrices of
$X$ has sub-exponential growth, then for any
$x\in{X},x'\in\gamma_{\infty}^+(x),$ we have
$$[x,x'],(x,x'),[x,x'),(x,x']\in \overline{\mathcal{R}^+}.$$
\end{proposition}

This proposition is based on a variant of the
Denjoy-Koksma argument, which in our context becomes the following simple observation on the decomposition
of arcs $[x,x'],(x,x'),[x,x'),(x,x']$.

\begin{proposition}
\label{checkempty}

For any $l\in {\mathbb Z}$ there exists an integer $N_l$
satisfying the inequality
\begin{equation}
\label{nlsubexp}
N_l\leq C_{\varepsilon} \exp(\varepsilon |l|)
\end{equation}
and such that the following holds.

Let $l\in {\mathbb Z}$ and let $\gamma$ be an arc from the semiring ${\mathfrak C}({\mathfrak o})$
such that ${\check \gamma}_l(\gamma)=\emptyset.$
Then
$$
\gamma=\gamma^{\prime} \sqcup \bigsqcup\limits_{k=1}^{N_l} \gamma_{l,k} \sqcup
\gamma^{\prime\prime},
$$
where $\gamma_{l,k}\in \C^+_{l-1}$, and
${\check \gamma}_{l-1}(\gamma^{\prime})={\check \gamma}_{l-1}(
\gamma^{\prime\prime})=\emptyset$ (some of the arcs may be empty).
\end{proposition}

In other words, if an arc from the semiring ${\mathfrak C}({\mathfrak o})$ does not contain arcs from
the semi-ring $\C_l^+$, then it cannot contain more than
$C_{\varepsilon} \exp(\varepsilon |l|)$ arcs of the semi-ring $\C_{l-1}^+$.

Proof. One may take
$$
N_l=2\max\limits_k \sum\limits_{i,k=1}^m (A_l)_{ik}+1.
$$

It is immediate that if an arc from the semiring ${\mathfrak C}({\mathfrak o})$  contains $N_l$ arcs from $\C_{l-1}^+$, then it also
contains an arc from $\C_l^+$.  The inequality (\ref{nlsubexp}) follows from the definition of subexponential growth.

Proposition \ref{checkempty} immediately
implies Proposition \ref{measuretoorderarc}.

\subsection {The H{\"o}lder  upper bound.}

Now let $X$ be a uniquely ergodic Markov compactum, let
$\nu^+\in{\mathfrak{V}}^+(X)$ be the positive measure and let
$h^{(n)}$ be the corresponding equivariant sequence, both
normalized in the usual way. We now check that if the components
of the vectors $h^{(n)}$ decay not faster than exponentially, then
each finitely-additive measure $\Phi^+\in\mathfrak{B}^+ $
satisfies a H\"{o}lder-type inequality with respect to $\nu ^+.$

\begin{assumption}\label{lowexpbnd} There exist $C>0,\gamma>0$ such
that the vectors $h^{(n)}$  satisfy
$$\min_{i\in\{1,\ldots,m\}}h_i^{(-n)}\geqslant C e^{-\gamma n}.$$
\end{assumption}

\begin{proposition}\label{hoeldersimple} Let $X$ be a uniquely
Markov compactum satisfying Assumption \ref{lowexpbnd}.
For any $\Phi^+\in\mathfrak{B} ^+$ there exist
$\varepsilon>0,\theta>0$ such that if $\gamma\in{\mathfrak{C}}^+$
satisfies $\nu^+(\gamma)\leqslant\varepsilon,$ then
$$|\Phi^+(\gamma)|\leqslant(\nu^+(\gamma))^{\theta}.$$
\end{proposition}

Proof. Indeed, take
$\gamma\in{\mathfrak{C}}^+,\gamma=\gamma_{-n}^+(x).$ Without
losing generality we may assume $n>0.$

Then, by definition of the space $\mathfrak{B} ^+(X),$
there exist constants $\widetilde{C}>0,\alpha>0$ depending only on
$\Phi^+$ such that
$$|\Phi^+(\gamma)|\leqslant
\widetilde{C}e^{-\alpha n}.$$

Conversely, by Assumption \ref{lowexpbnd} we have
$$\nu^+(\gamma)=h_{F(x_{-n})}^{(-n-1)}\geqslant Ce^{-\gamma n}.$$

Combining these two bounds and choosing $\varepsilon$ small
enough, we arrive at the desired estimate with any $\theta$
satisfying $$0<\theta<\frac{\alpha}{\gamma}.$$

Proposition \ref{hoeldersimple} admits the following partial converse.
\begin{proposition}
\label{bpluscharacter}
Let $X$ be a uniquely ergodic Markov compactum such that $\nu^+\in {\mathfrak B}^+$.
Take $\Phi^+\in {\mathfrak V}^+(X)$ and assume that there exist
$\varepsilon>0,\theta>0$ such that if $\gamma\in{\mathfrak{C}}^+$
satisfies $\nu^+(\gamma)\leqslant\varepsilon,$ then
\begin{equation}
\label{hoelderest}
|\Phi^+(\gamma)|\leqslant(\nu^+(\gamma))^{\theta}\end{equation}
Then $\Phi^+\in {\mathfrak B}^+$.
\end{proposition}
Proof. The condition $\nu^+\in {\mathfrak B}^+$ implies that the norms of the vectors of the positive equivariant sequence $h^{(n)}$ decay exponentially as $n\to-\infty$ and (\ref{hoelderest}) now implies that the same is true
for $\Phi^+$, whence the claim.

{\bf Remark.} Exactly the same propositions are of course valid for the finitely-additive measures on the foliations ${\mathcal F}^-$.

\begin{proposition}
\label{hoelderupper}
Let $X$ be a uniquely
ergodic Markov compactum such that

\begin{enumerate}
\item $\nu^+\in\mathfrak{B} ^+(X).$

\item  Assumption \ref{lowexpbnd} holds for
$X;$

\item The sequence $A_n({X}),n\in\mathbb{Z},$ has
sub-exponential growth.
\end{enumerate}

Let $\mathfrak{o}$ be a Vershik's ordering on $X.$ For any
$\Phi^+\in\mathfrak{B} ^+(X)$ there exist
$\theta>0,\varepsilon>0$ such that for any
$x\in{X},x'\in\gamma^+_{\infty}(x)$ satisfying
$$\nu^+([x,x'])\leqslant\varepsilon$$

we have

$$\Phi^+([x,x'])\leqslant\nu^+([x,x'])^{\theta}.$$

\end{proposition}

To prove the proposition, we first reformulate  Proposition \ref{checkempty} in the following way.
\begin{lemma}
\label{arcdecomposition}
Let $X$ be a
Markov compactum whose sequence $A_n({X}),n\in\mathbb{Z},$ of adjacency matrices
has sub-exponential growth.

For any $l\in {\mathbb Z}$ there exists an integer $M_l$
satisfying the inequality
\begin{equation}
\label{nlsubexp2}
M_l\leq C_{\varepsilon} \exp(\varepsilon |l|)
\end{equation}
such that the following holds.

Let $\gamma$ be an arc of the   from the semiring ${\mathfrak C}({\mathfrak o})$
Then there is a decomposition
\begin{equation}
\label{arcdecom}
\gamma=\bigsqcup\limits_{l=n^+(\gamma)}^{-\infty} \bigsqcup\limits_{k=1}^{M_l} \gamma_{l,k}
\end{equation}
such that $\gamma_{l,k}\in \C_l^+$.
\end{lemma}

Informally, Lemma \ref{arcdecomposition} says that any arc of our
symbolic flow is approximable by ``Markovian'' arcs with sub-exponential error; we illustrate this by Figure \ref{fig:five} .

\begin{figure}
\begin{center}
\includegraphics{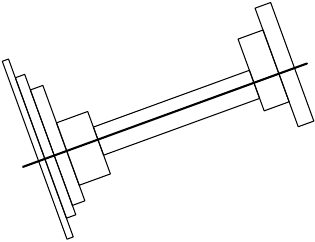}\\
\caption{The number of small arcs grows at most subexponentially}\label{fig:five}
\end{center}
\end{figure}
Proposition \ref{hoelderupper} is now clear from Proposition\ref{hoeldersimple} and Lemma
\ref{arcdecomposition}.

\subsection{Concatenation and aggregation.}

For two graphs $\Gamma, \Gamma^{\prime}\in\mathfrak G$, their concatenation $\Gamma\Gamma'\in\mathfrak G$ is defined as follows.
The set of edges ${\cal E} (\Gamma\Gamma')$ is given by the formula
$$
{\cal E} (\Gamma\Gamma')=\{(e,e^{\prime}),\ e\in{\cal E}(\Gamma),e\in\mathcal E(\Gamma^{\prime}), I(e)=F(e^{\prime})\},
$$
and we define $F(e,e^{\prime})=F(e), I(e,e^{\prime})=I(e^{\prime})$.
We clearly have $A(\Gamma\Gamma^{\prime})=A(\Gamma^{\prime})A(\Gamma)$.

Denote by $W({\mathfrak G})$ the set of all finite words over the alphabet ${\mathfrak G}$, and for
a word $w\in W({\mathfrak G})$, $w=w_0\dots w_n$, $w_i\in {\mathfrak G}$,
let $\Gamma(w)$ stand for the concatenation $w_0\dots w_n$ of the graphs $w_0, \dots, w_n$.

Now take a sequence $\Gamma_n\in {\mathfrak G}$, $n\in {\mathbb Z}$, and a strictly increasing
sequence of indices $i_n\in {\mathbb Z}$, $n\in {\mathbb Z}$.
Consider the concatenations
$$
{\check \Gamma}_n=\Gamma_{i_{n}}\dots \Gamma_{i_{n+1}-1}.
$$

The sequence ${\check \Gamma}_n$ will be called an {\it aggregation} of the sequence $\Gamma_n$, while the
sequence $\Gamma_n$ will be called a {\it refinement} of ${\check \Gamma}_n$.

Let $X$ be the Markov compactum corresponding to the sequence $\Gamma_n$,
${\check X}$ the Markov compactum corresponding to the sequence ${\check \Gamma}_n$.

We have a natural ``tautological'' homeomorphism
\begin{equation}
{\mathfrak t}_{Ag(i_n)}: X \to {\check X}.
\end{equation}

By definition, the homeomorphism  ${\mathfrak t}_{Ag(i_n)}$ sends foliations $\F^{\pm}_X$ to the respective
foliations $\F^{\pm}_{\check X}$ and identifies the spaces ${\mathfrak V}^{\pm}(X)$ and ${\mathfrak V}^{\pm}({\check X})$.

The Markov compactum $X$ is uniquely ergodic if and only if ${\check X}$ is, and
if $\Phi^+\in \BB^+(X)$, then $\left({\mathfrak t}_{Ag(i_n)}\right)_*\Phi^+\in \BB^+({\check X})$.

\subsection{Symbolic flows.}

Let $X$ be a uniquely ergodic Markov compactum corresponding to the sequence of graphs $\Gamma_l$ and assume that a Vershik's ordering is given on each $\Gamma_l$ and, thus, a linear ordering is induced on each leaf $\gamma_{\infty}^+(x)$.

Let $Max({\oo})$ be the set of points $x\in X$, $x=(x_n)_{n\in {\mathbb Z}}$, such that
each $x_n$ is a maximal edge.
Similarly,  $Min({\oo})$ denotes the set of points $x\in X$, $x=(x_n)_{n\in {\mathbb Z}}$, such that
each $x_n$ is a minimal edge. Since edges starting at a given vertex are ordered linearly,
the cardinalities of $Max({\oo})$ and $Min({\oo})$ do not exceed $m$.

If a leaf $\gamma^+_{\infty}$ does not intersect  $Max({\oo})$, then it does not have a maximal
element; similarly, if $\gamma^+_{\infty}$ does not intersect  $Min({\oo})$, then it does not have a
minimal element.

\begin{proposition}
\label{maxt}
Let $x\in X$.  If $\gamma^+_{\infty}(x)\cap Max(\oo)=\emptyset$, then for any $t\geq 0$ there exists a point $x^{\prime}\in
\gamma^+_{\infty}(x)$ such that
\begin{equation}
\label{hplust}
\nu_X^+([x, x^{\prime}])=t.
\end{equation}
\end{proposition}

Proof. Let $V(x)=\{t: \exists x^{\prime}\geq x: \nu_X^+([x, x^{\prime}])=t\}$.
Since  $\gamma^+_{\infty}(x)\cap Max(\oo)=\emptyset$, for any $n$ there exists
$x^{\prime\prime}\in \gamma^+_{\infty}(x)$ such that all points in $\gamma_n^+(x^{\prime\prime})$ are greater
than $x$. By Assumption \ref{uniquergodic}, the quantity
$\nu_X^+(\gamma_n^+(x^{\prime\prime}))$ goes to $\infty$
uniformly in $x^{\prime\prime}$,
as $n\to\infty$.
The set $V(x)$ is thus unbounded. Furthermore,  by Assumption \ref{uniquergodic}, the quantity $\nu_X^+(\gamma_n^+(x^{\prime\prime}))$ decays
to $0$, uniformly in
$x^{\prime\prime}$, as $n\to-\infty$, whence the set $V(x)$ is dense in ${\mathbb R}_+$.
Finally, by compactness of $X$, the
set $V(x)$ is closed, which concludes the proof of the Proposition.

A similar proposition, proved in the same way, holds for negative $t$.
\begin{proposition}
\label{mint}
Let $x\in X$.  If $\gamma^+_{\infty}(x)\cap Min(\oo)=\emptyset$,
then for any $t\geq 0$ there exists a point $x^{\prime}\in \gamma^+_{\infty}(x)$
such that
\begin{equation}
\label{hminust}
\nu^+_X([x^{\prime}, x])=t.
\end{equation}
\end{proposition}

Our next aim is to construct a flow $h_t^+$ such that for all
$t\geq 0$ we have $h_t^+x\in \gamma_{\infty}^+(x)$ and $\nu^+_X([x, h_tx])=t$.
Note, however, that the above conditions do not determine the point $h_t^+x$ uniquely.
We therefore modify the Markov compactum $X$ by gluing together the points $x,x^{\prime}$
such that $x<x^{\prime}$ but $(x,x^{\prime})=\emptyset$.

Define an equivalence relation $\sim$ on $X$ by writing $x\sim x^{\prime}$ if $x\in \gamma_{\infty}^+(x^{\prime})$
and $(x, x^{\prime})=(x^{\prime}, x)=\emptyset$. The equivalence classes admit the following explicit description, which is
clear from the definitions.
\begin{proposition}
\label{eqrel}
Let  $x, x^{\prime}\in X$ be such
that $x\in \gamma_{\infty}^+(x^{\prime})$, $x<x^{\prime}$ and
$\nu_X^+([x, x^{\prime}])=0$.
Then there exists $n\in {\mathbb Z}$ such that
\begin{enumerate}
\item $x^{\prime}_n$ is a successor of $x_n$;
\item $x$ is the maximal element in $\gamma_n(x)$;
\item  $x^{\prime}$ is the minimal element in $\gamma_n(x^{\prime})$.
\end{enumerate}
\end{proposition}
In other words, $\nu_X^+([x, x^{\prime}])=0$ if and only if
$(x, x^{\prime})=\emptyset$.
In particular, equivalence classes consist at most of two points and,
$\nu$-almost surely, of only one point.

Denote $X_{\oo}=X/{\sim}$, let $\pi_{\oo}: X\to X_{\oo}$ be the projection map and set
$\nu_{\oo}=(\pi_{\oo})_*\nu$. The probability spaces $(X_{\oo}, \nu_{\oo})$ and $(X,\nu)$ are measurably isomorphic;
in what follows, we shall often omit the index $\oo$. The foliations ${\cal F}^+$ and ${\cal F}^-$ descend to the space $X_{\oo}$;
we shall denote their images on $X_{\oo}$ by the same letters and, as before, denote by
$\gamma^+_{\infty}(x)$, $\gamma_{\infty}^-(x)$ the leaves containing $x\in X_{\oo}$.

Now let $x\in X_{\oo}$ satisfy $\gamma^+_{\infty}(x)\cap Max(\oo)=\emptyset$.
By Proposition \ref{maxt}, for any $t\geq 0$ there exists a unique $x^{\prime}$ satisfying
(\ref{hplust}). Denote $h_t^+(x)=x^{\prime}$.
Similarly, if $x\in X_{\oo}$ satisfy $\gamma^+_{\infty}(x)\cap Min(\oo)=\emptyset$.
By Proposition \ref{mint}, for any $t\geq 0$ there exists a unique $x^{\prime}$ satisfying
(\ref{hminust}). Denote $h_{-t}^+(x)=x^{\prime}$.

We thus obtain a flow $h_t^+$, which is well-defined on the set
$$
X_{\oo}\setminus \Big(\bigcup\limits_{x\in Max(\oo)\cup Min(\oo)}\gamma_{\infty}^+(x)\Big),
$$
and, in particular, $\nu$-almost surely on $X_{\oo}$.
By definition the flow $h_t^+$ preserves the measure $\nu$.

The flow $h_t^+$ is a suspension flow over the Vershik's automorphism corresponding to
the one-sided Markov compactum $Y$ given by  the sequence $\Gamma_n$, $n\geq 1$.
The roof function is simply the piecewise constant function $h^{(1)}_{F(y_1)}$.

A Vershik's ordering on two graphs $\Gamma, \Gamma^{\prime}$ yields an ordering on their concatenation $\Gamma\Gamma^{\prime}$:
one sets $(e,e^{\prime})<({\tilde e}, {\tilde e}^{\prime})$ if $e^{\prime}<{\tilde e}^{\prime}$ or if
$e^{\prime}={\tilde e}^{\prime}$, $e<{\tilde e}$.

Thus, if a Markov compactum $X$ is endowed with a Vershik's ordering $\oo$, and the Markov compactum ${\check X}$ is obtained
from $X$ by concatenation with respect to a strictly increasing sequence $(i_n)$, then ${\check X}$ is automatically also
endowed with a Vershik's ordering ${\check \oo}$, and
the map ${\mathfrak t}_{Ag(i_n)}$ sends the flow $h_t^{+, \oo}$ on $X$ to the flow $h_t^{+, {\check \oo}}$ on ${\check X}$.

In a similar way, assume that for every graph $\Gamma_n$, $n\in {\mathbb Z}$,
a linear ordering ${\tilde o}$ is given on all the edges ending at a given vertex.
Such an ordering will be called a {\it reverse Vershik's ordering}.
In the same way as above, a  reverse Vershik's ordering induces a $\nu$-preserving
flow on the leaves of the foliation $\F^-$.

\subsection{H{\"o}lder cocycles.}

As before, we consider a uniquely ergodic Markov compactum $X$  endowed with a Vershik's
ordering ${\mathfrak o}$, and we denote by $h_t^+$ the resulting flow.
By an {\it arc} of the flow $h_t^+$ we mean a set of the
type
\begin{equation}
\label{gammaxt}
\gamma(x,t)=\{y\in \gamma^+(x), x\leq y < h_t^+(x)\}, x\in X, t\geq 0.
\end{equation}

In other words, an arc is the image, under the quotient map
by the equivalence relation $\sim_{\oo}$, of an interval
$[x, x^{\prime})=\{x^{\prime\prime}: x\leq x^{\prime\prime}< x^{\prime}\}$.

By Proposition \ref{measuretoorderarc}, if the adjacency matrices of our Markov compactum have subexponential growth, then all the finitely-additive measures from ${\mathfrak B}^+$ can be extended
to all arcs of the flow $h_t^+$, or, in other words, we have the following
\begin{proposition}
\label{measuretococycle}
Any arc of the flow $h_t^+$ belongs to the ring ${\overline \R}^+$.
\end{proposition}

Since every measure
$\Phi^+\in {\mathfrak B}^+$ is defined
on every arc of the flow $h_t^+$, it follows that
such a measure defines a  cocycle
on the orbits of the flow $h_t^+$
by the formula
$$
\Phi^+(x,t)=\Phi^+([x, h_tx]).
$$
Slightly abusing notation, we denote the measure and
the corresponding cocycle by the same letter; we identify the measure and the
cocycle and we speak, for instance,  of the norm of the cocycle, meaning the
norm  of the corresponding finitely-additive measure, etc.

\subsection{Balanced, Lyapunov regular and hyperbolic Markov compacta.}

More precise statements about the H{\"o}lder behaviour of the cocycles can be given under stronger
assumptions on adjacency matrices of our Markov compactum.  The assumptions that we give here can certainly be weakened; nonetheless,  they
hold for random Markov compacta and are thus sufficient for our
purposes.
We proceed to formal definitions.

A Markov compactum will be called
{\it balanced} if the following holds.
\begin{assumption}
\label{balanced}
There exists a  positive constant $C$, a
strictly increasing sequence of indices
$i_n\in {\mathbb Z}$, $n\in {\mathbb Z}$,
such that
\begin{equation}
\label{seqin}
\lim\limits_{n\to-\infty} i_n=-\infty, \  \lim\limits_{n\to\infty} i_n=\infty,
\end{equation}
such that
\begin{enumerate}
\item  The matrices
$A_{i_{n+1}}\dots A_{i_n+1}$ have sub-exponential growth in $n\in {\mathbb Z}$;
\item For any $n\in {\mathbb Z}$ all entries of the matrix
$A_{i_{n+1}}\dots A_{i_n+1}$ are positive, and for all
$j,k,l\in\{1,\dots,m\}$ we have
$$
\frac{(A_{i_{n+1}}\dots A_{{i_n}+1})_{jk}}
{(A_{i_{n+1}}\dots A_{{i_n}+1})_{lk}}\leq C.
$$
\end{enumerate}
\end{assumption}

A sufficient condition for the second requirement is that
there exist a matrix $Q$ all whose entries are positive
and a sequence $i_n$ satisfying (\ref{seqin}) such that
$A_{{i_n}}=Q$. We shall see that Markov compacta random with respect to a $\sigma$-invariant ergodic probability
measure $\mu$ on $\Omega$ satisfying Assumption \ref{asos} are automatically balanced.

Note also that a balanced Markov compactum is automatically uniquely ergodic.

A Markov compactum satisfying Assumption \ref{asexpdec}
will be called {\it Lyapunov regular} if the following additional assumption holds.

\begin{assumption}
\label{lyapreg}
There exist positive numbers
$\theta_1>\theta_2>\dots>\theta_{l_0}>0$ and, for any
$n\in {\mathbb Z}$, a direct sum decomposition
$$
E_n^u=E_n^1\oplus E_n^2\dots\oplus E_n^{l_0}
$$
such that $A_nE_n^i=E_{n+1}^i$ and for any nonzero $v\in E_n^i$
we have
\begin{equation}
\label{lyapone}
\lim\limits_{k\to\infty} \frac{\log|A_{n-k}\dots A_nv|}{k}=\theta_i;
\end{equation}
\begin{equation}
\label{lyaptwo}
\lim\limits_{k\to\infty} \frac{\log|(A_{n-k}\dots A_{n})^{-1}v|}{k}=-\theta_i.
\end{equation}
The convergence in (\ref{lyapone}), (\ref{lyaptwo}) is uniform
on the unit sphere $\{v\in E_n^i, |v|=1\}$.
\end{assumption}

Note that the positive equivariant sequence  $h^{(n)}$ of a
balanced Lyapunov regular Markov compactum
automatically satisfies Assumption \ref{lowexpbnd}.

Now let $X$ be uniquely ergodic, take $v\in E_0^u$ and consider the corresponding finitely additive measure
$\Phi^+\in {\mathfrak B}^+$. Decompose
$v=v^{(1)}+\dots +v^{(l_0)}$, $v^{(i)}\in E_0^i$
and let $j$ be the lowest index such that $v^{(j)}\neq 0$.
Then $\theta_j$ will be called the {\it Lyapunov exponent} of the
finitely-additive measure   $\Phi^+$.
For instance, the positive measure has exponent $\theta_1$.
Similarly, let $k$ be the highest index such that $v^{(k)}\neq 0$.
Then $\theta_k$ will be called the {\it lower Lyapunov exponent} of $\Phi^+$.
It will develop that the Lyapunov exponent of $\Phi^+$ governs its growth at infinity, while the lower Lyapunov exponent of $\Phi^+$ controls its H{\"o}lder behaviour at zero.

We shall often need the dual of Assumption \ref{lyapreg}.
Take a Markov compactum $X$ satisfying Assumptions \ref{asexpdec},
\ref{asdualexp}, \ref{lyapreg}.
For $i=1, \dots, l_0$, define
$$
{\tilde E}_n^i=Ann\left(\oplus_{j\neq i} E_n^j\right).
$$
We have then
$$
{\tilde E}_n^u={\tilde E}_n^1\oplus {\tilde E}_n^2\dots\oplus
{\tilde E}_n^{l_0}
$$
and $A^t_n{\tilde E}_{n+1}^i={\tilde E}_{n}^i$.

\begin{assumption}
\label{duallyapreg}
For any $n\in {\mathbb Z}$  and any nonzero $v\in {\tilde E}_n^i$
we have
\begin{equation}
\label{duallyapone}
\lim\limits_{k\to\infty} \frac{\log|A^t_{n-k}\dots A^t_{n}v|}{k}= {\theta}_i;
\end{equation}
\begin{equation}
\label{duallyaptwo}
\lim\limits_{k\to\infty} \frac{\log|(A^t_{n+k}\dots A^t_{n})^{-1}v|}{k}=-\theta_i.
\end{equation}
The convergence in (\ref{duallyapone}), (\ref{duallyaptwo})
is uniform on the unit sphere $\{v\in {\tilde E}_n^i, |v|=1\}$.
\end{assumption}

A Markov compactum satisfying Assumptions \ref{lyapreg}, \ref{duallyapreg} will be called {\it Lyapunov
bi-regular}.

{\bf Remark.} The {\it uniform} convergence in (\ref{duallyapone}), (\ref{duallyaptwo}) is guaranteed by the Oseledets Multiplicative Ergodic Theorem (see \cite{pesinbarreira} and the Appendix A).

A case of special interest for us will be when
all Lyapunov exponents are simple, i.e., when the following holds.
\begin{assumption}
We have $l_0={\rm dim} E_n^u={\rm dim} {\tilde E}_n^u$ and
$$
{\rm dim} E_n^i={\rm dim} {\tilde E}_n^i=1, \ i=1, \dots, l_0.
$$
\end{assumption}
In the latter case we shall say that the Markov
compactum $X$ has {\it simple Lyapunov spectrum}.

Finally, a Lyapunov regular Markov compactum $X$ will be called {\it hyperbolic} if $\BB_c^+(X)=\BB^+(X)$.

\subsection{The H{\"o}lder property for cocycles}

\subsubsection{The upper bound.}

Proposition \ref{hoelderupper} can now be reformulated as follows.

\begin{corollary}
\label{hoeldprop}
Let $X$ be a Markov compactum with top Lyapunov
exponent $\theta_1$. For any $\varepsilon>0$ there exists a positive constant
$C_{\varepsilon}$ depending only on $X$ such that the following is true.
Let $\Phi^+\in {\mathfrak B}^+$ have
Lyapunov exponent $\theta>0$. Then for any $x\in X$ and any $t\in {\mathbb R}$ we have
\begin{equation}
|\Phi^+(x,t)|\leq C_{\varepsilon}\cdot |\Phi^+| \cdot |t|^{\theta/\theta_1-\varepsilon}.
\end{equation}
\end{corollary}

\subsubsection{The logarithmic asymptotics at infinity.}

The upper bound of Corollary \ref{hoeldprop} is precise, as is shown by the following

\begin{proposition}
\label{convhoeldprop}
Let $X$ be a Lyapunov regular balanced Markov compactum with top Lyapunov
exponent $\theta_1$, and
let $\Phi^+\in {\mathfrak B}^+$ have
Lyapunov exponent $\theta>0$. Then for any $x\in X$ we have
\begin{equation}
\limsup_{t\to\infty} \frac{\log |\Phi^+(x,t)|}{\log t}=\theta/\theta_1.
\end{equation}
\end{proposition}

Proof. The upper bound is furnished by Proposition \ref{hoeldprop}, and we proceed to the proof of the lower bound.
Let $v^{(n)}$ be the equivariant sequence corresponding to
$\Phi^+.$ Take $\varepsilon>0.$ There exists $n_0$ such that for
every $n>n_0$ there exists $i\in\{1,\ldots,m\}$ such that
$|v_i^{(n)}|\geqslant e^{(\theta-\varepsilon)n}$.

Now take $x\in X$ and let $x'\in\gamma_{\infty}^+(x)$ be the
smallest element of $\gamma_{\infty}^+(x)$ satisfying the
requirements:
\begin{enumerate}
\item $F(x_{n+1}')=i.$

\item There exists $\widetilde{x}\in(x,x')$ such that
$F(\widetilde{x}_{n+1})\neq i .$

\end{enumerate}

Let $x''>x'$ be the smallest element of
$\gamma_{\infty}^+(x')=\gamma_{\infty}^+(x)$ such that
$F(x_{n+1}'')\neq  ~ i.$

By definition, we have $$\gamma_{n+1}(x')=[x',x'').$$

Now, by Lyapunov regularity of the Markov compactum $X$ we
have $$\nu^+([x,x'])\leqslant e^{(\theta_1+\varepsilon)n},$$
$$\nu^+([x',x''])\leqslant e^{(\theta_1+\varepsilon)n}.$$

We also have $$|\Phi^+([x',x''])|=|v_i^{(n)}|\geqslant
e^{(\theta-\varepsilon)n}.$$

Consequently, $$\max(|\Phi^+([x,x'])|,|\Phi^+([x,x''])|)\geqslant
\frac{1}{2}e^{(\theta-\varepsilon)n},$$ and, therefore, either
$$|\Phi^+([x,x'])|\geqslant\frac{1}{2}(\nu^+([x,x']))^{\frac{\theta-\varepsilon}{\theta_1+\varepsilon}}$$
or
$$|\Phi^+([x,x''])|\geqslant\frac{1}{2}(\nu^+([x,x'']))^{\frac{\theta-\varepsilon}{\theta_1+\varepsilon}}$$

The desired lower bound is established. We illustrate the argument in Figure \ref{fig:one}.

\begin{figure}
\begin{center}
\includegraphics{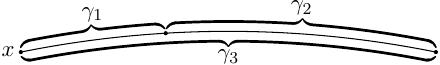}\\
\caption{Proof of the lower bound in Proposition \ref{convhoeldprop}: $\gamma_1=[x,x^{\prime}]$, $\gamma_2=[x^{\prime}, x^{\prime\prime}]$. Either $\gamma_1$ or $\gamma_3$ satisfies the lower bound. }\label{fig:one}
\end{center}
\end{figure}

The same argument yields precise H{\"o}lder behaviour at zero, and we have

\begin{proposition}
\label{zerohoeldprop}
Let $X$ be a Lyapunov regular balanced Markov compactum with top Lyapunov
exponent $\theta_1$, and
let $\Phi^+\in {\mathfrak B}^+$ have lower
Lyapunov exponent $\theta>0$. Then for any $x\in X$ we have
\begin{equation}
\limsup_{t\to 0} \frac{\log |\Phi^+(x,t)|}{\log t}=\theta/\theta_1.
\end{equation}
\end{proposition}

Proof. The upper bound follows from Corollary \ref{hoeldprop}. The lower bound is established by the same argument
as that used in  Proposition \ref{convhoeldprop} (except that now one must take $n\to-\infty$ instead of $n\to\infty$):
first one finds a Markovian arc $[x^{\prime}, x^{\prime\prime}]$ satisfying the lower bound, and then one notes that, in view of Lyapunov regularity, one of the arcs $[x,x^{\prime}]$, $[x, x^{\prime\prime}]$ must also satisfy the desired lower bound.

\subsubsection{Expectation and variance of H{\"o}lder cocycles}
\begin{proposition}
\label{phiexpzero}
For any $\Phi^+\in \BB^+$ and any $t_0\in {\mathbb R}$ we have
$$
\ee_{\nu}(\Phi^+(x,t_0))=\langle \Phi^+, \nu^-\rangle\cdot t_0.
$$
\end{proposition}

Proof: Since the Proposition is clearly valid for $\Phi^+=\nu^+$, it suffices to prove it
in the case $\langle \Phi^+, \nu^-\rangle=0$.
But indeed, if $\ee_{\nu}(\Phi^+(x,t))\neq 0$, then the Ergodic  Theorem implies
$$
\limsup_{T\to\infty} \frac{\log |\Phi^+(x,T)|}{\log T}=1,
$$
and then $\langle \Phi^+, \nu^-\rangle\neq 0$.

\begin{proposition}
\label{phivarnotzero}
For any $\Phi^+\in \BB^+$ not proportional to $\nu^+$  and any $t_0\neq 0$ we have
$$
Var_{\nu} \Phi^+(x,t_0)\neq 0.
$$
\end{proposition}

Taking $\Phi^+-\langle \Phi^+, \nu^-\rangle\cdot\nu^+$ instead of $\Phi^+$, we may
assume $\ee_{\nu}(\Phi^+(x,t_0))=0$. If $Var_{\nu} \Phi^+(x,t_0)=0$, then $\Phi^+(x,t_0)=0$
identically, but then
 $$
\limsup_{T\to\infty} \frac{\log |\Phi^+(x,T)|}{\log T}=0,
$$
whence $\Phi^+=0$, and the Proposition is proved.

{\bf Remark.} In the context of substitutions,
related cocycles have been studied by P.~Dumont, T.~Kamae and
S.~Takahashi in \cite{dumontkamae}
as well as by T.~Kamae in \cite{kamae}.

\subsection{Approximation of weakly Lipschitz functions.}

The weak Lipshitz property of a function $f$ implies uniform estimates on difference of its integrals
along any two arcs of the flow $h_t^+$ that stay within a fixed Markovian rectangle (see Figure \ref{fig:arcsinonerect}).

\begin{proposition}
There exists a constant $C>0$ such that for any  $f\in Lip_w^+(X)$, any $T>0$ and any pair of
points $x, x^{\prime}$ the following is true. If there exists $i\in \{1, \dots, m\}$ and  $n\in {\mathbb N}$ such that
for all $t: 0\le t\le T$ we have $F((h_tx)_n)=F((h_tx^{\prime})_n)=i$, then we have
\begin{equation}
\left|\int_0^T f\circ h_t^+(x)dt-\int_0^T f\circ h_t^+(x^{\prime})dt\right|\leq C ||f||_{Lip_w^+}.
\end{equation}
\end{proposition}

\begin{figure}
\begin{center}
\includegraphics{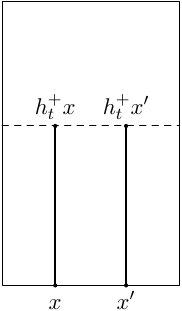}\\
\caption{Flow arcs staying in the same Markovian rectangle}\label{fig:arcsinonerect}
\end{center}
\end{figure}

We are now ready to prove the main result of this subsection:
\begin{lemma}
\label{mainindmarkcomp}
Let $X$ be a uniquely ergodic Lyapunov biregular  balanced Markov compactum.
There exists a continuous mapping $\Xi^+: Lip_w^+(X)\to \BB^+(X)$
such that the following holds.
For any $\varepsilon>0$ there exists
a constant $C_{\varepsilon}$ such that for any
$f\in Lip_w^+(X)$, any $x\in X$
and any $T>0$ we have
\begin{equation}
\label{indmarkapprox}
\left| \int_0^T
f\circ h_t^+(x) dt -\Xi^+(f; x,T)\right|\leq
C_{\varepsilon}||f||_{Lip_w^+}(1+T^{\varepsilon}).
\end{equation}

The map $\Xi^+$ is uniquely defined by the requirement that for any $\Phi^-\in \BB^-(X)$
we have
\begin{equation}
\langle \Xi^+(f), \Phi^- \rangle =\int\limits_X fd\nu^+\times \Phi^-.
\end{equation}

\end{lemma}

{\bf Remark.}  We write $\Xi^+(f; x,T)$ instead of $\Xi^+(f)( x,T)$.

Proof: If $f\in Lip_w^+(X)$, then the measure $fd\nu^+$ is a weakly Lipschitz measure on the foliation $\F^+$.
Let $\Xi^+$ be the continuous map given by the Preparatory Approximation Lemma \ref{prepapprox} and, slightly abusing notation, write
$\Xi^+(f)=\Xi^+(fd\nu^+)$. The bound (\ref{indmarkapprox}) for ``Markovian'' arcs of the form $\gamma_n^+(x)$ is contained in Lemma \ref{prepapprox}, while for general arcs of the flow $h_t^+$ the claim follows from the approximation bound of
Corollary \ref{approxrplus} which is applicable to flow arcs by  Proposition \ref{measuretoorderarc}.
Lemma \ref{mainindmarkcomp} is proved completely.

\subsection{The Approximation Theorem for random Markov compacta.}
\subsubsection{Skew-products associated to the shift and to the renormalization cocycle}

We now derive the Approximation Theorem for random Markov compacta and study the action
that the right shift $\sigma$ on the space $\Omega$ of Markov compacta induces on the spaces
${\mathfrak B}^+$ of H{\"o}lder cocycles of individual Markov compacta. This action will
play a key r{\^o}le in the proof of limit theorems.

For $\omega\in\Omega$ we have a natural map ${\mathfrak t}_{\sigma}: X(\omega)\to X(\sigma\omega)$ which
to
a point $x\in X$ assigns the point ${\tilde x}\in X(\sigma \omega)$
given by ${\tilde x}_n=x_{n+1}$.

The map ${\mathfrak t}_{\sigma}$ sends the foliations
$\F^+_{\omega}$, $\F^-_{\omega}$ to $\F^+_{\sigma\omega}$, $\F^-_{\sigma\omega}$;
the semirings $\C^+_{\omega}$, $\C^-_{\omega}$ to $\C^+_{\sigma\omega}$, $\C^-_{\sigma\omega}$;
and induces an isomorphism $({\mathfrak t}_{\sigma})_*: \VV^+(X_{\omega})\to \VV^+(X_{\sigma\omega})$
given by the usual formula
$$
({\mathfrak t}_{\sigma})_*\Phi^+(\gamma)=\Phi^+(({\mathfrak t}_{\sigma})^{-1}\gamma), \ \gamma\in \C^+_{\sigma\omega}.
$$

Introduce the space
$$
\mathfrak X\Omega=\{ (\omega,x),x\in X(\omega)\}
$$
and endow it with a skew-product map ${\sigma}^{{\mathfrak X}}$
given by the formula $$
{\sigma}^{{\mathfrak X}}(\omega,x)=(\sigma\omega, \mathfrak t_{\sigma}x).
$$

Next, to the renormalization cocycle ${\mathbb A}$
assign the corresponding skew-product transformation
$\sigma^{{\mathbb A}}: \Omega\times {\mathbb R}^m\to \Omega\times {\mathbb R}^m$
defined by the formula
$$
\sigma^{{\mathbb A}}(\omega, v)=(\sigma\omega, {\mathbb A}(1, \omega)v).
$$

Consider the space
$$
\VV^+\Omega=\{(\omega, \Phi^+): \omega\in\Omega_{inv},  \Phi^+\in \VV^+(X(\omega))\},
$$
and endow $\VV^+\Omega$ with the automorphism ${\mathfrak T}_{\sigma}$ given by the formula
$$
{\mathfrak T}_{\sigma}(\omega, \Phi^+)=(\sigma \omega, ({\mathfrak t}_{\sigma})_*\Phi^+).
$$

Let $\omega\in\Omega_{inv}$, $v\in {\mathbb R}^m$.
Consider the equivariant sequence ${\bf v}=v^{(n)}$ such that $v^{(0)}=v$.
Let $\Phi_{{\bf v}}^+\in \VV^+(X(\omega))$ be the corresponding finitely-additive measure.
Since $v^{(n)}$ is uniquely determined
by $v^{(0)}$, we obtain the isomorphism
$$
{\cal I}^+_{\omega}: {\mathbb R}^m \to {\mathfrak V}^+(X(\omega))
$$
given by ${\cal I}^+_{\omega}v=\Phi^+_{{\bf v}}$, ${\bf v}=v^{(n)}$,  $v^{(0)}=v$.

For any $\omega\in\Omega_{inv}$ the diagram

$$
\begin{CD}
 {\mathbb R}^m@>{\cal I}^+_{\omega}>> {\mathfrak V}^+(X(\omega)) \\
@VV{\mathbb A}(1, \omega)V           @VV({\mathfrak t}_{\sigma})_*V   \\
 {\mathbb R}^m@>{\cal I}^+_{\sigma\omega}>> {\mathfrak V}^+(X(\sigma\omega)) \\
\end{CD}
$$
is commutative.

\subsubsection{Properties of random Markov compacta}

\begin{proposition}
If the measure $\mu$ satisfies Assumption \ref{asos}, then for almost
every $\omega$ the Markov compactum $X(\omega)$ is uniquely ergodic and
Lyapunov bi-regular.
\end{proposition}

Indeed, unique ergodicity is clear by the first condition, while,
in view of the second and the third, the Oseledets theorem
implies bi-regularity.

Given $\omega\in\Omega$,
let $E^u_{\omega}$ be the Lyapunov subspace at the point $\omega$ corresponding to
positive Lyapunov exponents of ${\mathbb A}$.
The first condition in Assumption \ref{asos} implies the nontriviality of the subspace $E^u_{\omega}$:
indeed, we have $h^{(0)}_{\omega}\in E^u_{\omega}$.

\begin{proposition}
\label{lyapidentif}
For $\mu$-almost every $\omega\in\Omega$ the transformation ${\cal I}^+_{\omega}$
maps  the subspace $E^u_{\omega}$ isomorphically onto
${\mathfrak B}^+(X(\omega))$.
\end{proposition}

This is immediate from the Oseledets Multiplicative Ergodic Theorem.

For the transpose cocycle, in the same way as above, for $\omega\in\Omega_{inv}$, $v\in {\mathbb R}^m$,
we have
the isomorphism
$$
{\cal I}^-_{\omega}: {\mathbb R}^m \to {\mathfrak V}^-(X(\omega))
$$
given by ${\cal I}^-_{\omega}{\tilde v}=\Phi^+_{{\tilde {\bf v}}}$,
where ${\tilde {\bf v}}={\tilde v}^{(n)}$  is a reverse equivariant subsequence satisfying ${\tilde v}^{(0)}={\tilde v}$.

As before, for any $\omega\in\Omega_{inv}$ the diagram
$$
\begin{CD}
 {\mathbb R}^m@>{\cal I}^-_{\omega}>> {\mathfrak V}^-(X(\omega)) \\
@AA{\mathbb A}^t(1, \sigma\omega)A           @VV({\mathfrak t}_{\sigma})_*V   \\
 {\mathbb R}^m@>{\cal I}^-_{\sigma\omega}>> {\mathfrak V}^-(X(\omega)) \\
\end{CD}
$$
is commutative.

Continuing, again we let $\mu$ satisfy Assumption \ref{asos}, and, for $\omega\in\Omega$,
set ${\tilde E}^u_{\omega}$ to be the Lyapunov subspace at the point $\omega$ corresponding to
positive Lyapunov exponents of the cocycle ${\mathbb A}^t$. By definition of the spaces ${\mathfrak B}^-(X(\omega))$ and the Oseledets Theorem,
we have
\begin{proposition}
\label{lyapidentiftranspose}
For $\mu$-almost every $\omega\in\Omega$ the transformation ${\cal I}^-_{\omega}$
maps  the subspace ${\tilde E}^u_{\omega}$ isomorphically onto
${\mathfrak B}^-(X(\omega))$.
\end{proposition}

\subsubsection{Duality}

Take $\omega\in \Omega_{inv}$. For $v, {\tilde v}\in {\mathbb R}^m$ we clearly have
$$
\langle {\cal I}_{\omega}^+(v), {\cal I}_{\omega}^-({\tilde v})\rangle=\sum\limits_{i=1}^m v_i {\tilde v}_i.
$$

If $\mu$ satisfies Assumption \ref{asos}, then,
by the Oseledets Theorem, for almost every $\omega\in\Omega$ the Markov compactum $X(\omega)$ 
satisfies Assumption \ref{asdualexp}.
In particular,
the standard Euclidean inner product yields a nondegenerate pairing of the subspaces
$E^u_{\omega}$ and ${\tilde E}^u_{\omega}$. Consequently, we have the following
\begin{corollary} If a probability $\sigma$-invariant ergodic  measure $\mu$ satisfies Assumption \ref{asos},
then for $\mu$-almost every $\omega\in\Omega$ the pairing $\langle, \rangle$ is nondegenerate  on the pair of subspaces
$\BB^+(X(\omega))$, $\BB^-(X(\omega))$.
\end{corollary}

\subsubsection{Balanced random Markov compacta and the Approximation Theorem.}

\begin{proposition}
\label{uebalanced}
If the measure $\mu$ satisfies Assumption \ref{asos}, then for almost all $\omega$ the Markov compactum
$X(\omega)$ is balanced.

\end{proposition}
Proof: We verify the requirements of Assumption \ref{balanced} one by one.
Let $\Gamma_0\in {\mathfrak G}_+$ have positive probability with respect to $\mu$ (the existence of such $\Gamma_0$ is given by
 the first condition of Assumption \ref{asos}). Let $i_n$, $n\in
{\mathbb Z}$ be consecutive moments of time
such that $\omega_{i_n}=\Gamma_0$ (the sequence $i_n$ is almost surely unbounded both in the positive and
in the negative direction). The positivity of $A(\Gamma_0)$ immediately implies the second
requirement of Assumption \ref{balanced}, and it remains to verify the first requirement.
Denote by $\sigma_{\Gamma_0}$  the induced map of $\sigma$ on the set $\{\omega:\omega_0=\Gamma_0\}$. The
renormalization cocycle ${\mathbb A}$ naturally yields the induced cocycle ${\mathbb A}_{\Gamma_0}$
over $\sigma_{\Gamma_0}$ (to obtain the matrix of ${\mathbb A}_{\Gamma_0}$ one needs to multiply all the
matrices
of ${\mathbb A}$ occurring between two consecutive appearances of $\Gamma_0$). If the logarithm of
of the norm of ${\mathbb A}$ is integrable, then the same is true of ${\mathbb A}_{\Gamma}$, whence
the first requirement    of Assumption \ref{balanced} follows immediately.

Summing up, we see that under Assumption \ref{asos}
for almost every $\omega$ the Markov compactum $X(\omega)$ is uniquely ergodic, Lyapunov
biregular and balanced; Lemma \ref{mainindmarkcomp} is thus applicable and yields
\begin{corollary}[The Approximation Theorem for Random Markov Compacta]
\label{symbmultdiscrtime}
Let $\mu$ be an  ergodic, $\sigma$-invariant probability measure on $\Omega$
satisfying Assumption \ref{asos}. For any $\varepsilon>0$ there exists
a constant $C_{\varepsilon}$ depending only on $\mu$
such that the following holds for almost every $\omega\in\Omega$.
There exists a continuous mapping $\Xi^+_{\omega}: Lip_{w}^+(X(\omega))\to \BB^+(X(\omega))$
such that
\begin{enumerate}
\item for any $t_0\in {\mathbb R}$ we have $\Xi^+_{\oomega}(f\circ h_{t_0}^+)=\Xi^+_{\oomega}(f)$;
\item the diagram
$$
\begin{CD}
Lip_{w}^+(X(\omega))@ >\Xi^+_{\oomega}>> \BB^+(X(\omega)) \\
@AA({\mathfrak t}_{\sigma})^*A           @VV({\mathfrak t}_{\sigma})_*V   \\
Lip_{w}^+(X(\sigma\omega))@ >\Xi^+_{\sigma\omega}>>
\BB^+(X(\sigma\omega)) \\
\end{CD}
$$
is commutative
\item
 for any
$f\in Lip_w^+(X)$, any $x\in X(\omega)$ and any $T>0$ we have
\begin{equation}
\label{transflowapprox}
\left| \int_0^T
f\circ h_t^+(x) dt -\Xi^+(f; x,T)\right|\leq
C_{\varepsilon}||f||_{Lip_w^+}(1+T^{\varepsilon}).
\end{equation}
\end{enumerate}
\end{corollary}

\subsubsection{Hyperbolic random Markov compacta}

We now give a sufficient condition for hyperbolicity of random Markov compacta.
Let $({\mathcal X}, \mu)$ be a probability space endowed with a $\mu$-preserving transformation $T$ or flow $g_s$ and
an integrable linear cocycle $A$ over $g_s$  with values in ${\mathrm GL}(m, {\mathbb R})$.

For $p\in{\mathcal X}$ let $E_{0, x}$ be the the {\it neutral} subspace of $A$ at $p$, i.e., the
Lyapunov subspace of the cocycle ${A}$ corresponding to the
Lyapunov exponent $0$.
We say  that ${A}$ {\it acts isometrically on its neutral subspaces}
if for almost any $p$ there exists
an inner product $\langle\cdot\rangle_{p}$ on ${\mathbb R}^m$ which depends on
$p$ measurably and satisfies
$$
\langle { A}(1, p)v,{ A}(1, p)v\rangle_{g_sp}=\langle v, v\rangle_{p}, \ v\in E_{0, p}
$$
for all $s\in {\mathbb R}$ (again, in the case of a transformation, $g_s$ should be replaced by $T$ in this formula).

The following proposition is clear from the definitions.
\begin{proposition}
\label{hyperboliccriterion}
Let $\nu$ be a $\sigma$-invariant ergodic probability measure on $\Omega$ satisfying Assumption \ref{asos}
and such that the renormalization cocycle ${\mathbb A}$ acts isometrically on its neutral subspaces
with respect to $\nu$. Then for almost all $\omega$ the Markov compactum $X(\omega)$ is hyperbolic.
\end{proposition}

In other words, any continuous finitely-additive measure must in fact be H{\"o}lder.
Note that the assumptions of the proposition are verified,
in particular, for the symbolic counterpart of the Masur-Veech
smooth measure on the moduli space of abelian differentials.

\section{The Renormalization Flow on the Space of Measured Markov Compacta.}
\subsection{The space of measured Markov compacta.}

Let $\Omega_{ue}\subset \Omega$ be the subset of $\omega$ such that the Markov compactum $X(\omega)$ is
uniquely ergodic.
For $\omega\in\Omega_{ue}$, $r\in {\mathbb R}_+$ define
\begin{equation}
\label{nur}
\nu^+_{(\omega, r)}=\frac{\nu^+_{\omega}}{r};\
\nu^-_{(\omega, r)}=r{\nu^-_{\omega}}.
\end{equation}

It is clear that for any $r>0$ we have
\begin{equation}
\nu_{\omega}=\nu^+_{(\omega, r)}\times\nu^-_{(\omega, r)}.
\end{equation}

Furthermore, from the definitions it is clear that for $l\in {\mathbb Z}$ we have
\begin{equation}
({\mathfrak t}_{\sigma})^l_*\nu^+_{(\omega,r)}=\nu^+_{(\sigma^l\omega, r|\la^{(l)}_{\omega}|)}.
\end{equation}

We now introduce an equivalence relation $\sim$ on the set of pairs
$(\omega,r)\in\Omega_{ue}\times {\mathbb R}_+$ in the following way:
$$
(\omega, r)\sim (\omega^{\prime}, r^{\prime})
$$
if there exists $l\in {\mathbb Z}$ such that
\begin{equation}
\omega^{\prime}=\sigma^l\omega, \
r^{\prime}/r=|\la^{(l)}(\omega)|.
\end{equation}

Since, by definition,
$$
\la^{(n)}_{\sigma^l\omega}=\frac{\la_{\omega}^{(n+l)}}{|\la_{\omega}^{(l)}|},
$$
we have
$$
|\la^{(l)}_{\omega}|=\frac 1{|\la^{(-l)}_{\sigma^l\omega}|},
$$
and the fact that $\sim$ is an equivalence relation is clear.

\subsubsection{The renormalization flow and the renormalization cocycle.}

Let ${\overline \Omega}$ be the set of equivalence classes of $\sim$.
Introduce a flow $g_s$ on ${\overline \Omega}$  by the formula
$$
g_s(\omega,r)=(\omega, e^{s}r).
$$
The flow $g_s$ will be called the {\it renormalization} flow.

An explicit fundamental domain for $\sim$ is given by the set
\begin{equation}
\label{funddomain}
\Omega_0=\{(\omega, r): \omega\in\Omega_{ue},  1\leq r< |\la^{(1)}(\omega)|^{-1}\}.
\end{equation}

To every pair $(\omega, r)\in\Omega_0$ we assign
\begin{enumerate}
\item the Markov compactum $X(\omega,r)=X(\omega)$;
\item the foliations $\F_{(\omega,r)}^+=\F_{\omega}^+$, $\F_{(\omega,r)}^-=\F_{\omega}^-$;
\item the measures $\nu^+_{(\omega, r)}=\frac{\nu^+}{r}$, $\nu^-_{(\omega, r)}=r{\nu^+}$;
\item the spaces of finitely-additive measures $\BB^+_{(\omega,r)}=\BB^+_{\omega}$,
$\BB^-_{(\omega,r)}=\BB^-_{\omega}$.
\end{enumerate}

We identify ${\overline \Omega}$ with $\Omega_0$ and for ${\overline \omega}\in {\overline \Omega}$ we
speak of the corresponding Markov compactum $X({\overline \omega})$, the foliations,
the measures etc., meaning those objects for the corresponding point $(\omega,r)\in \Omega_0$.

The identification of ${\overline \Omega}$ with $\Omega_0$ yields a representation of
$g_s$ as a suspension flow over the shift $\sigma$ with roof function
\begin{equation}
\label{rooffunction}
\tau^1(\omega)=-\log{|\la^{(1)}(\omega)|}.
\end{equation}
It is important to note that the roof function $\tau^1(\omega)$ given by (\ref{rooffunction})
only depends on the {\it future} of the sequence $\omega$.

Given
$(\omega,r)\in\Omega_0$ and $s\in {\mathbb R}$, define an integer ${\tilde n}(\omega,r,s)$
by the formula
\begin{equation}
\label{ntilde}
(\omega, e^sr)\sim (\sigma^{{\tilde n}(\omega,r,s)}\omega, r^{\prime}), \ (\sigma^{{\tilde
n}(\omega,r,s)}\omega, r^{\prime})\in\Omega_0
\end{equation}

For every $s\in{\mathbb R}$, we have a natural map
$$
{\mathfrak t}_s: X({\overline\omega})\to X(g_s{\overline\omega})
$$
given, for ${\overline \omega}=(\omega,r)$, $(\omega,r)\in\Omega_0$, by the formula
$$
{\mathfrak t}_s({\overline \omega})={\mathfrak t}_{\sigma}^{{\tilde n}(\omega,r,s)}.
$$

Introduce the space
$$
\mathfrak X{\overline \Omega}=\{ (\oomega,x),x\in X(\oomega)\}
$$
and endow it with the skew-product flow ${g}_{s}^{{\mathfrak X}}$
given by the formula $$
{g}_{s}^{{\mathfrak X}}(\oomega,x)=({g}_s^{{\mathfrak X}}\oomega, {\mathfrak t}_{s}x).
$$

It is immediate from the definitions that if
$(\omega, r)\sim (\omega^{\prime}, r^{\prime})$ and
$\omega^{\prime}=\sigma^l\omega$, then
\begin{equation}
{\mathfrak t}_{\sigma}^l\nu^+_{(\omega, r)}=\nu^+_{(\omega^{\prime}, r^{\prime})}; \
{\mathfrak t}_{\sigma}^l\nu^-_{(\omega, r)}=\nu^-_{(\omega^{\prime}, r^{\prime})},
\end{equation}
and we thus have
\begin{equation}
({\mathfrak t}_{s})_*\nu^+_{{\overline \omega}}=\nu^+_{g_s{\overline\omega}}; \
({\mathfrak t}_s)_*\nu^-_{{\overline\omega}}=\nu^-_{g_s{\overline\omega}}.
\end{equation}

For ${\overline \omega}=(\omega,r)$, $(\omega,r)\in\Omega_0$, write
\begin{equation}
\label{overlinea}
{\overline {\mathbb A}}(s, {\overline \omega})={\mathbb A}({\tilde n}(\omega,r,s), \omega).
\end{equation}

We thus obtain a matrix cocycle ${\overline {\mathbb A}}$ over the  flow $g_s$.

Let $g_s^{{\overline {\mathbb A}}}: \Oomega\times {\mathbb R}^m\to \Oomega\times {\mathbb R}^m$ be the skew-product
transformation corresponding to the cocycle ${\overline {\mathbb A}}$ by the formula
$$
g_s^{{\overline {\mathbb A}}}(\oomega, v)=(g_s\oomega, {\overline {\mathbb A}}v).
$$

As before, for any $\oomega=(\omega, r)$, we have the isomorphism
$$
{\cal I}^+_{\oomega}: {\mathbb R}^m \to {\mathfrak V}^+(X(\oomega))
$$
given by ${\cal I}^+_{\oomega}v=\Phi^+_{{\bf v}}$, ${\bf v}=v^{(n)}$,  $v^{(0)}=v$ (recall that, by definition, $\omega\in\Omega_{inv}$).

For any $s\in {\mathbb R}$
the diagram
$$
\begin{CD}
 {\mathbb R}^m@>{\cal I}^+_{\oomega}>> {\mathfrak V}^+(X(\oomega)) \\
@VV{\mathbb A}(s, \oomega)V           @VV({\mathfrak t}_{s})_*V   \\
 {\mathbb R}^m@>{\cal I}^+_{g_s\oomega}>> {\mathfrak V}^+(X(g_s\oomega)) \\
\end{CD}
$$
is commutative.

\subsubsection{Characterization of finitely-additive measures.}

Introduce the space ${\mathfrak V}^+{\overline\Omega}$ by the formula
$$
{\mathfrak V}^+{\overline\Omega}=\{(\oomega, \Phi^+), \Phi^+\in {\mathfrak V}^+_{\oomega}\}
$$
and a flow ${\mathfrak T}_s$ on  ${\mathfrak V}^+{\overline\Omega}$ by the formula
$$
{\mathfrak T}_s(\oomega, \Phi^+)=(g_s\oomega, \left({\mathfrak t}_s\right)_*\Phi^+).
$$

The trivialization map
$$
{\rm Triv}: {\mathfrak V}^+{\overline\Omega}\to \Oomega\times {\mathbb R}^m
$$
is given by the formula
$$
{\rm Triv}(\oomega, \Phi^+)=(\oomega, ({\cal I}^+_{\oomega})^{-1}\Phi^+).
$$
The diagram
$$
\begin{CD}
 {\mathfrak V}^+{\overline\Omega}@>{\rm Triv}>> \Oomega\times {\mathbb R}^m \\
@VV{\mathfrak T}_sV           @VVg_s^{{\overline {\mathbb A}}}V   \\
 {\mathfrak V}^+{\overline\Omega}@>{\rm Triv}>> \Oomega\times {\mathbb R}^m\\
\end{CD}
$$
is commutative.

Recall that $W({\mathfrak G})$ stands for the set of all finite words over the alphabet ${\mathfrak G}$, and that
for a word $w\in W({\mathfrak G})$, $w=w_0\dots w_n$, $w_i\in {\mathfrak G}$,
we write $\Gamma(w)$ for the concatenation $w_0\dots w_n$.

Let ${\mathfrak G}_+$ be the set of all $\Gamma\in {\mathfrak G}$ such that all entries of the matrix $A(\Gamma)$ are positive.
Take $\Gamma\in {\mathfrak G}_+$ and a word $w\in W({\mathfrak G})$, $w=w_0\dots w_n$,
satisfying $\Gamma(w)=\Gamma$.
Let ${\mathscr M}(w, \infty)$ be the family of
Borel ergodic $\sigma$-invariant measures $\mu$ on $\Omega$ (finite or infinite)
such that
\begin{enumerate}
\item
$
\mu(\{\omega:\omega_0=w_0, \dots, \omega_n=w_n\})>0;
$
\item $
\int_{\Omega} \tau^1(\omega) d\mu(\omega)=1.
$
\end{enumerate}

The first condition together with ergodicity of $\mu$ implies that $\mu(\Omega\setminus\Omega_{ue})=0$.
The second condition yields that the measure
$$
\Prob_{\mu}=\mu\times \frac{dr}{r}
$$
is a well-defined probability $g_s$-invariant measure on ${\overline \Omega}$.

Denote
$$
{\mathscr M}(\Gamma, \infty)=\bigcup\limits_{w\in W({\mathfrak G}), \Gamma(w)=\Gamma} {\mathscr M}(w, \infty),
$$
and let
${\mathscr P}^+$ be the space of ergodic probability $g_s$-invariant measures $\Prob$ on ${\overline \Omega}$
having the form $\Prob=\Prob_{\mu}$, $\mu\in {\mathscr M}(\Gamma, \infty)$, $\Gamma\in {\mathfrak  G}_+$.

By definition, for any $\Prob\in{\mathscr P}^+$
the cocycle ${\overline {\mathbb A}}$ is integrable with respect
to $\Prob$, and the Oseledets Theorem is applicable to ${\overline {\mathbb A}}$.
For ${\overline \omega}\in \Oomega$, let  $E^u_{\oomega}$ be
the strictly unstable space of the cocycle ${\overline {\mathbb A}}$ at the point $\oomega$.

\begin{proposition}
\label{oomegaidentif}
For almost every $\oomega$ the map ${\cal I}^+_{\oomega}$ induces an isomorphism from
the strictly unstable space $E^u_{\oomega}$ of the cocycle ${\overline {\mathbb A}}$ at the point $\oomega$ to the space
${\BB}^+(X({\oomega}))$.
\end{proposition}

We start with the case when the measure $\mu$ is finite.
In this case, for $\oomega\in\Oomega$, $\oomega=(\omega, r)$,
the unstable space $E^u_{\oomega}$ of the cocycle ${\overline {\mathbb A}}$ at $\oomega$
coincides with the unstable space $E^u_{\omega}$ of the cocycle ${{\mathbb A}}$ at $\omega$.
Concatenating the graphs between consecutive occurrences of the word $w$ and considering the induced
map of the shift $\sigma$
on the cylinder $\{\omega:\omega_0=w_0, \dots, \omega_n=w_n\}$, we obtain
Proposition \ref{oomegaidentif} as an immediate corollary of Proposition \ref{lyapidentif}.

In this case the Lyapunov exponents of  ${\mathbb A}$ and ${\overline {\mathbb A}}$ are related by the following
\begin{proposition}
If the positive Lyapunov exponents of ${\mathbb A}$ are
\begin{equation}
\theta_1> \theta_2\dots > \theta_{l_0},
\end{equation}
then the positive Lyapunov exponents ${\overline \theta}_i$ of ${\overline {\mathbb A}}$ are
\begin{equation}
{\overline \theta}_i=\theta_{i}/\theta_1.
\end{equation}
\end{proposition}

{\bf Remark.} In particular, we always have ${\overline \theta}_1=1$.

We proceed to the case when $\mu\in {\mathscr M}(w, \infty)$ is infinite.
We will again consider the induced
map of the shift $\sigma$ on the cylinder $\{\omega:\omega_0=w_0, \dots, \omega_n=w_n\}$,
and check that the induced measure is finite.

More precisely, let ${\mathfrak G}_{\Gamma}\subset  {\mathfrak G}$ be defined by the formula
$$
{\mathfrak G}_{\Gamma}=\{\Gamma^{\prime}\in {\mathfrak G}: \Gamma^{\prime}=\Gamma\Gamma^{\prime\prime} \ {\rm for \ some } \ \Gamma^{\prime\prime}\in {\mathfrak G}\}.
$$

Let $\Omega_{\Gamma}\subset \Omega$ be the subspace of sequences such that all their symbols lie in ${\mathfrak G}_{\Gamma}$, and let ${\overline \Omega}_{\Gamma}\subset {\overline \Omega}$ be the set of all pairs $(\omega, r)\in{\overline \Omega}$ satisfying $\omega\in\Omega_{\Gamma}$.

Let ${\mathscr M}_{\Gamma}$ be the space of Borel ergodic $\sigma$-invariant measures $\mu$
on  $\Omega_{\Gamma}$ such that
$$
\int_{{\Omega_{\Gamma}}} \tau^1(\omega) d\mu(\omega)=1.
$$

If $\Gamma\in {\mathfrak G}_+$, then
any measure in ${\mathscr M}_{\Gamma}$ is necessarily finite since the function $\tau(\omega)$ is
bounded away from $0$ on $\Omega_{\Gamma}$.

Let $\Omega^{\prime}(w, \infty)\subset \Omega$ be the subset of
bi-infinite sequences containing infinitely many occurrences of $w$, both in the past and in the future, and let
$\Omega(w, \infty)\subset\Omega^{\prime}(w, \infty)$ be the subset of
$\omega\in \Omega^{\prime}(w, \infty)$ satisfying the additional condition $\omega_0=w_0, \dots, \omega_n=w_n$.
Let $\sigma(w, \infty)$ be the induced map of the shift $\sigma$ to $\Omega(w, \infty)$.

Write
$$
{\overline \Omega}(w, \infty)=\{(\omega, r)\in {\overline \Omega}, \omega\in\Omega^{\prime}(w, \infty)\}.
$$

For a measure $\mu\in {\mathscr M}(w, \infty)$, let $\mu(w, \infty)$ be its restriction
to $\Omega(w, \infty)$.

Concatenating the graphs between consecutive occurrences of $w$, we obtain a natural aggregating surjection
$$
Ag^{w}: \Omega(w, \infty)\to \Omega_{\Gamma}
$$
such that the diagram
$$
\begin{CD}
\Omega(w, \infty)@ >\sigma(w, \infty)>> \Omega(w, \infty) \\
@VV Ag^{w} V           @VV Ag^{w} V   \\
\Omega_{\Gamma}@ >\sigma>>\Omega_{\Gamma} \\
\end{CD}
$$
is commutative.

The map $Ag^{w}$ lifts to a map ${\overline {Ag^w}}: {\overline \Omega}(w, \infty)\to {\overline \Omega}_{\Gamma}$
such that the diagram
$$
\begin{CD}
{\overline \Omega}(w, \infty)@ >g_s>> {\overline \Omega}(w, \infty) \\
@VV {\overline {Ag^w}} V           @VV {\overline {Ag^w}} V   \\
{\overline \Omega}_{\Gamma}@ >g_s>>{\overline \Omega}_{\Gamma} \\
\end{CD}
$$
is commutative.

The map ${\overline {Ag^w}}$ preserves the cocycle ${\overline {\mathbb A}}$  in the following sense:
if the map ${\overline {Ag^w}}\times {\rm Id}: {\overline \Omega}(w, \infty)\times {\mathbb R}^m\to
{\overline \Omega}_{\Gamma}\times {\mathbb R}^m$  is given by
$$
{\overline {Ag^w}}\times {\rm Id}(\oomega, v)=({\overline {Ag^w}}\oomega, v),
$$
then the diagram
$$
\begin{CD}
{\overline \Omega}(w, \infty)\times {\mathbb R}^m@ >g_s^{{\overline {\mathbb A}}}>> {\overline \Omega}(w, \infty)\times {\mathbb R}^m \\
@VV {\overline {Ag^w}}\times {\rm Id} V           @VV {\overline {Ag^w}}\times {\rm Id} V   \\
{\overline \Omega}_{\Gamma}\times {\mathbb R}^m@ >g_s^{{\overline {\mathbb A}}}>>{\overline \Omega}_{\Gamma}\times {\mathbb R}^m \\
\end{CD}
$$
is commutative.

For $\mu\in {\mathscr M}(w, \infty)$, set
$$
\mu^{w}=\left(Ag^w\right)_*\mu.
$$

The correspondence
\begin{equation}
\mu\rightarrow \mu^{w}
\end{equation}
induces an affine map from ${\mathscr M}(w, \infty)$ to ${\mathscr M}_{\Gamma}$.
From the definitions we clearly have
\begin{proposition}
For any $\mu\in{\mathscr M}(w, \infty)$ the dynamical systems $({\overline \Omega}, \Prob_{\mu}, g_s)$ and
$({\overline \Omega}_{\Gamma}, \Prob_{\mu^{w}}, g_s)$ are measurably isomorphic.
\end{proposition}

We have thus reduced the case of infinite measures to the case of finite measures, in which Proposition \ref{lyapidentif} is applicable.
Proposition \ref{oomegaidentif} is proved completely.

\subsubsection{The space of translation flows.}

Now assume that for almost all $\omega$ there is a
Vershik's ordering $\oo(\omega)$
on the edges of each graph $\omega_n$, $n\in {\mathbb Z}$. Furthermore,
assume that   the ordering is shift-invariant in the following
sense: the ordering $\oo(\omega)$ on the edges of the graph
$\omega_{n+1}$ is the same as the ordering $\oo(\sigma\omega)$
on the edges of the graph $(\sigma\omega)_n=\omega_{n+1}$.

In this case for almost every $\omega\in\Omega$ we obtain
a flow $h_t^{+, \omega}$ on $X(\omega)$.
Given $(\omega, r)\in \Omega_0$, set $h_t^{+, (\omega,r)}=h^{+, \omega}_{t/r}$.

Identifying, as before, the spaces $\Omega_0$ and ${\overline \Omega}$
we obtain, for almost every ${\overline \omega}$, a flow $h_t^{+, {\overline \omega}}$
on $X({\overline \omega})$ in such a way that
the following diagram is
commutative:

$$
\begin{CD}
X({\overline\omega})@ >h^+_{\exp(s)t,{\overline\omega}}>> X({\overline\omega}) \\
@VV{\mathfrak t}_{s}V           @VV{\mathfrak t}_{s}V   \\
X(g_s{\overline\omega})@ >h^+_{t, {\overline\omega}}>>
X(g_s{\overline\omega}) \\
\end{CD}
$$

Recall that ${\mathscr P}^+$ is the space of ergodic probability $g_s$-invariant measures $\Prob$ on ${\overline \Omega}$
having the form $\Prob=\Prob_{\mu}$, $\mu\in {\mathscr M}(\Gamma, \infty)$, $\Gamma\in {\mathfrak  G}_+$.
Corollary \ref{symbmultdiscrtime} now  implies
\begin{corollary}
\label{symbmultiplic}
Let $\Prob\in {\mathscr P}^+$. For any $\varepsilon>0$ there exists
a constant $C_{\varepsilon}$ depending only on $\Prob$
such that the following holds for almost every $\oomega\in\Oomega$.
There exists a continuous mapping $\Xi^+_{\oomega}: Lip_{w}^+(X(\oomega))\to \BB^+(X(\oomega))$
such that
\begin{enumerate}
\item for any $t_0\in {\mathbb R}$ we have $\Xi^+_{\oomega}(f\circ h_{t_0}^+)=\Xi^+_{\oomega}(f)$;
\item the diagram
$$
\begin{CD}
Lip_{w}^+(X({\overline\omega}))@ >\Xi^+_{\oomega}>> \BB^+(X({\overline\omega})) \\
@AA({\mathfrak t}_{s})^*A           @VV({\mathfrak t}_{s})_*V   \\
Lip_{w}^+(X(g_s{\overline\omega}))@ >\Xi^+_{g_s\oomega}>>
\BB^+(X(g_s{\overline\omega})) \\
\end{CD}
$$
is commutative
\item
 for any
$f\in Lip_w^+(X)$, any $x\in X(\oomega)$ and any $T>0$ we have
\begin{equation}
\label{symbflowapprox}
\left| \int_0^T
f\circ h_t^+(x) dt -\Xi^+(f; x,T)\right|\leq
C_{\varepsilon}||f||_{Lip_w^+}(1+T^{\varepsilon}).
\end{equation}
\end{enumerate}
\end{corollary}

\subsection{Limit theorems in the case of the simple second Lyapunov exponent.}
\subsubsection{The leading term in the asymptotic for the ergodic integral.}
We assume that the first and the second Lyapunov exponents of the cocycle ${\overline {\mathbb A}}$ are simple,
and we consider the corresponding  subspaces $E_{1, \oomega}^u={\mathbb R}h_{\oomega}$ and $E_{2, \oomega}^u$.
Furthermore, let $E_{\geq 3, \oomega}^u$ be the subspace corresponding to the remaining Lyapunov exponents.

We have then the decomposition
$$
E_{\oomega}^u=E_{1, \oomega}^u\oplus E_{2, \oomega}^u\oplus E_{\geq 3, \oomega}^u.
$$

A similar decomposition holds for ${\tilde E}^u$:
$$
{\tilde E}_{\oomega}^u={\tilde E}_{1, \oomega}^u\oplus {\tilde E}_{2, \oomega}^u\oplus {\tilde E}_{\geq 3, \oomega}^u.
$$

Choose $\Phi_2^+\in {\cal I}_{\oomega}(E_{2, \oomega}^u)$, $\Phi_2^-\in
{\cal I}_{\oomega}({\tilde E}_{2, \oomega}^u)$ in such a way that
$$
\langle \Phi_2^+, \Phi_2^-\rangle =1.
$$
Take $f\in Lip_{w}^+(X)$, $x\in X$, $T\in {\mathbb R}$  and observe that the expression
\begin{equation}
\label{mphitwo}
m_{\Phi_2^-}(f)\Phi_2^+(x,T)
\end{equation}
does not depend on the precise choice of $\Phi_2^{\pm}$ (we have the freedom of multiplying $\Phi_2^+$
by an arbitrary scalar, but then $\Phi_2^-$ is divided by the same scalar).

Now for $f\in Lip_{w}^+(X)$ write
$$
\Phi_f^+(x,T)=(\int_X fd\nu)\cdot T + m_{\Phi_2^-}(f)\Phi_2^+(x,T)+\Phi_{3, f}^+(x,T),
$$
where $\Phi_{3, f}^+\in {\cal I}_{\oomega}({\tilde E}_{\geq 3, \oomega}^u)$.

In particular, there exist two positive constants  $C$ and $\alpha$ depending only on $\Prob$ such that
for any  $f\in Lip_{w}^+(X)$, $\int_X fd\nu=0$, we have
\begin{equation}
\label{phitwoapprox}
\left|\int_0^T f\circ h_t^+(x) dt - m_{\Phi_2^-}(f)\Phi_2^+(x,T)\right|\leq C||f||_{Lip} T^{{\overline \theta_2}-\alpha}.
\end{equation}

\subsubsection{The growth of the variance.}

In order to estimate the variance of the random
variable $\int_0^T f\circ h_t^+(x) dt$, we start by studying the growth of
the variance of the random variable $\Phi_{2, \oomega}^+(x,T)$ as $T\to\infty$.

Recall that $\ee_{\nu(\oomega)}\Phi_{2, \oomega}^+(x,T)=0$ for all $T$, while
$Var_{\nu(\oomega)}\Phi_{2, \oomega}^+(x,T)\neq 0$ for $T\neq 0$.
Recall that for a cocycle $\Phi^+\in\BB_{\oomega}^+$,
$\Phi^+={\cal I}^+_{\oomega}(v)$,
we have defined its norm $|\Phi^+|$ by the formula $|\Phi^+|=|v|$.
Introduce a multiplicative cocycle $H_2(s, \oomega)$ over the flow $g_s$ by the formula
\begin{equation}
H_2(s, \oomega)=\frac{|{\overline \A}(s, \oomega)v|}{|v|}, \ v\in E_{2, \oomega}^u,\  v\neq 0.
\end{equation}
 Observe that the right-hand side does not depend on the specific choice of $v\neq 0$.

By definition, we now have
\begin{equation}
\lim_{s\to\infty}\frac{\log H_2(s, \oomega)}{s}={\overline \theta}_2.
\end{equation}

\begin{proposition}
There exists a positive measurable
function $V:\Oomega\to {\mathbb R}_+$ such that
the following holds for $\Prob$-almost all $\oomega\in\Oomega$.
\begin{equation}
{Var_{\nu} {\Phi_2^+}(x,T)}=V(g_s\oomega)|\Phi_2^+|^2(H_2(s, \oomega))^2.
\end{equation}
\end{proposition}
Indeed, the function $V(\oomega)$ is given by
$$
V(\oomega)=\frac{{Var_{\nu} {\Phi_2^+}(x,1)}}{|\Phi^+_2|^2}.
$$
Observe that the right-hand side does not depend on  a particular choice of $\Phi_2^+\in {\BB}_{2, \oomega}^+$,
$\Phi_2^+\neq 0$.

Using (\ref{phitwoapprox}), we now proceed to
estimating the growth of the variance of the ergodic integral  $$\int_0^T f\circ h_t^+(x) dt.$$

We use the same notation as in the Introduction: for
$\tau\in [0,1]$, $s\in {\mathbb R}$, a real-valued
$f\in Lip_{w,0}^+(X)$ we write
\begin{equation}
\label{sfstaux-symb}
{\mathfrak S}[f,s;\tau, x]=\int_0^{\tau\exp(s)} f\circ h^{+}_t(x)dt.
\end{equation}

As before, as $x$ varies
in the probability space $(X, \nu)$, we obtain a random element
of $C[0,1]$. In other words, we have a
 random variable
\begin{equation}
{\mathfrak S}[f,s]: (X, \nu)\to C[0,1]
\end{equation}
defined by the formula (\ref{sfstaux-symb}).

For any fixed $\tau\in [0,1]$  the formula (\ref{sfstaux-symb}) yields
a real-valued random variable
\begin{equation}
{\mathfrak S}[f,s; \tau]: (X, \nu)\to {\mathbb R},
\end{equation}
whose expectation, by definition, is zero.

\begin{proposition}
\label{varf}
There exists $\alpha>0$ depending only on $\Prob$ and a positive measurable
function $C:\Oomega\times\Oomega\to {\mathbb R}_+$ such that
the following holds for $\Prob$-almost all $\oomega\in\Oomega$.
Let $\Phi_{2, \oomega}^+\in\BB^+$, $\Phi_{2, \oomega}^-\in\BB^-$ be such that
$\langle \Phi^+_{2, \oomega}, \Phi^-_{2, \oomega}\rangle=1$.
Let $f\in Lip_{w,\oomega}^+$ be such that
$$
\int_{X(\oomega)} f d\nu(\oomega)=0,\  m_{\Phi_{2, \oomega}^-}(f)\neq 0.
$$
Then
\begin{equation}
\left|\frac{Var_{\nu(\oomega)} {\mathfrak S}[f,s;1]}{V(g_s\oomega) (m_{\Phi_2^-}(f)|\Phi^+_2|H_2(s, \oomega))^2}-1\right|\leq
C(\oomega, g_s\oomega)\exp(-\alpha s).
\end{equation}
\end{proposition}
{\bf Remark.} Observe that the quantity $(m_{\Phi_2^-}(f)|\Phi^+_2|)^2$ does
not depend on the specific choice of $\Phi_2^+\in\BB^+_2$, $\Phi_2^-\in\BB^-_2$ such that $\langle \Phi^+_2, \Phi^-_2\rangle=1$.

Indeed, the proposition is immediate from the inequality
$$
|\ee(\xi_1^2)-\ee(\xi_2^2)|\leq {\rm sup}|\xi_1+\xi_2| \cdot \ee|\xi_1-\xi_2|,
$$
which holds for any two bounded random variables $\xi_1, \xi_2$ on any probability space,
and the following clear Proposition.
\begin{proposition}
There exists a constant  $\alpha>0$ depending only on $\Prob$, a positive measurable
function $C:\Oomega\times\Oomega\to {\mathbb R}_+$ and a positive measurable function
$V^{\prime}:\Oomega\to {\mathbb R}_+$ such that
\begin{equation}
\max |\Phi_2^+(x,e^s)|=V^{\prime}(g_s\oomega)H_2(s, \oomega);
\end{equation}
\begin{equation}
 \left|\frac{\max {\mathfrak S}[f,s;1]}{V^{\prime}(g_s\oomega) (m_{\Phi_2^-}(f)|\Phi^+|H_2(s, \oomega))^2}-1\right|\leq
C(\oomega, g_s\oomega)\exp(-\alpha s).
\end{equation}

\end{proposition}

\subsubsection{Formulation and proof of the limit theorem.}

We now turn to the asymptotic behaviour of the distribution of the
random variable ${\mathfrak S}[f,s]$ as $s\to\infty$.

Again, we will use the notation
$\mm[f,s]$ for the distribution of the normalized random variable
\begin{equation}
\frac{{\mathfrak S}[f,s]}{{\sqrt{Var_{{\bf m}} {\mathfrak S}[f,s;1]}}}.
\end{equation}
The measure $\mm[f,s]$ is thus a probability distribution on the space $C[0,1]$ of continuous functions on the unit interval.

For $\tau\in {\mathbb R}$, $\tau\neq 0$, we again let $\mm[f,s; \tau]$ be the distribution of the ${\mathbb R}$-valued random variable
\begin{equation}
\frac{{\mathfrak S}[f,s; \tau]}{{\sqrt{Var_{{\bf m}} {\mathfrak S}[f,s; \tau]}}}.
\end{equation}
If $f$ has zero average, then, by definition,  $\mm[f,s; \tau]$ is a measure on ${\mathbb R}$ of expectation $0$ and variance $1$.
Again, as in the Introduction, we take the space $C[0,1]$
of continuous functions on the unit interval endowed with
the Tchebyshev topology, we let $\MM$ be the space of Borel probability
measures on the space $C[0,1]$
endowed with the weak topology (see \cite{bogachev} or the Appendix B).

Consider the space $\HH^{\prime}$ given by
the formula
$$
\Oomega^{\prime}=\{\oomega^{\prime}=(\oomega, v), v\in E_{2, \oomega}^+, |v|=1\}.
$$
The  flow $g_s$ is lifted to $\Oomega^{\prime}$ by the formula
$$
g_s^{\prime}(X, v)=\left(g_sX, \frac{{\overline {\mathbb A}}(s,\oomega)v}{|({\overline {\mathbb A}}(s,\oomega)v|}\right).
$$

Given $\oomega^{\prime}\in {\Oomega}^{\prime}$, $\oomega^{\prime}=(\oomega, v)$,
write
$$
\Phi_{2, \oomega^{\prime}}^+={\cal I}_{\oomega}(v).
$$
As before, write
$$
V(\oomega^{\prime})=Var_{\nu(\oomega)}   \Phi_{2, \oomega^{\prime}}^+(x,1).
$$

Now introduce the map
$$
{\cal D}_2^+: \Oomega^{\prime}\to \MM
$$

by setting    ${\cal D}_2^+(\oomega^{\prime})$ to be the distribution
of the $C[0,1]$-valued normalized random variable
$$
\frac{\Phi^+_{2,\oomega^{\prime}}(x, \tau)}{\sqrt{V(\oomega^{\prime})}}, \  \tau\in[0,1].
$$

Note here that
for any $\tau_0\neq 0$ we have $Var_{\nu(\oomega)} \Phi_{2, \oomega}(x, \tau_0)\neq 0$,
so, by definition, we have
${\cal D}_2^+(\oomega^{\prime})\in\MM_1$.

Now, as before, we take a function  $f\in Lip_{w,\oomega}^+$ such that
 $$
 \int_{X(\oomega)} f d\nu(\oomega)=0,\  m_{\Phi_{2, \oomega}^-}(f)\neq 0
 $$
As before, $d_{LP}$  stands for the L{\'e}vy-Prohorov metric on $\MM$, $d_{KR}$  for the Kantorovich-Rubinstein
metric on $\MM$.
\begin{proposition}
\label{limthmmarkcomp-simple}
Let $\Prob\in {\mathscr P}^+$ be
such that both the first and the second Lyapunov exponents of the renormalization
cocycle ${\overline {\mathbb A}}$ are positive and simple with respect to $\Prob$.
There exists a positive measurable function $C: \Oomega\times \Oomega\to {\mathbb R}_+$
and a positive constant $\alpha$ depending only on $\Prob$
such that
for $\Prob$-almost every $\oomega^{\prime}\in\Oomega^{\prime}$, $\oomega^{\prime}=(\oomega, v)$,
and any $f\in Lip_{w,0}^+(X)$
satisfying
$m_{2, X^{\prime}}^-(f)>0$ we have
\begin{equation}
d_{LP}(\mm[f,s], {\cal D}_2^+(g_s^{\prime}\oomega^{\prime}))\leq C(\oomega, g_s\oomega)\exp(-\alpha s).
\end{equation}
\begin{equation}
d_{KR}(\mm[f,s], {\cal D}_2^+(g_s^{\prime}\oomega^{\prime}))\leq C(\oomega, g_s\oomega)\exp(-\alpha s).
\end{equation}
\end{proposition}

Proof: We start with the simple inequality
$$
\left|\frac{a}{b}-\frac{c}{d}\right| \leq
|a|\cdot\left|\frac{b-d}{bd}\right|+\frac{|a-c|}{d}
$$
valid for any real numbers $a,b,c,d$.
For any pair of random variables $\xi_1, \xi_2$  taking values in an arbitrary Banach space
and any positive real numbers $M_1, M_2$
we consequently have
\begin{equation}
\label{ineqlimthm}
\sup \left| \frac{\xi_1}{M_1} - \frac{\xi_2}{M_2}\right|\leq
\sup|\xi_1|\cdot\left|\frac{M_1-M_2}{M_1M_2}\right|+\frac{\sup|\xi_1-\xi_2|}{M_2}.
\end{equation}

We apply the inequality (\ref{ineqlimthm}) to the $C[0,1]$-valued random variables
$$
\xi_1={\mathfrak S}[f,s], \ \xi_2=\Phi_{2, g_s\oomega}^+(x,\tau\cdot e^s),
$$
letting $M_1$, $M_2$ be the corresponding normalizing variances: $M_1=Var_{\nu(\oomega)}{\mathfrak S}[f,s; 1]$,
$M_2=Var_{\nu(\oomega)}\mm[f,s;1]$.

Now take $\varepsilon>0$ and  let ${\tilde \xi}_1, {\tilde \xi}_2$ be
two random variables on an arbitrary probability space $(\Omega, \Prob)$ taking values in a complete metric space
and such that the distance between their values does not exceed $\varepsilon$.
In this case both the L{\'e}vy-Prohorov and the Kantorovich-Rubinstein distance between their distributions $({\tilde \xi}_1)_*\Prob$, $({\tilde \xi}_2)_*\Prob$ also does not exceed $\varepsilon$ (see Lemma \ref{dist-images} in Appendix B).

Proposition \ref{limthmmarkcomp-simple} is now immediate from Equation (\ref{phitwoapprox}) and Proposition \ref{varf}.

\subsubsection{Proof of Corollary \ref{nonentwo}.}

For $\oomega^{\prime}\in\Oomega^{\prime}$, $\Phi^+\in\BB^+_{\oomega}$
let ${\mathfrak m}[\Phi^+, \tau]$ be the distribution
of the normalized ${\mathbb R}$-valued random variable
$$
\frac{\Phi^+(x, \tau)}{\sqrt{Var_{{\nu}}\Phi^+(x, \tau)}}.
$$

\begin{proposition}
 Let $ \mathbb{P}\in\mathscr{P}^+.$
For $\Prob$-almost every $\oomega$ and any $\Phi^+\in \BB^+_{\oomega}$, $\Phi^+\neq 0$,  the correspondence
$$
\tau\rightarrow \mm[\Phi^+, \tau]
$$
yields a continuous map from ${\mathbb R}\setminus 0$ to $\MM({\mathbb R})$.
\end{proposition}
Proof. This is immediate from the H{\"o}lder property of the cocycle $\Phi^+$ and the nonvanishing of the
variance  $Var_{\nu} \Phi^+(x, \tau)$ for $\tau\neq 0$, which is  guaranteed by Proposition \ref{phivarnotzero}.

As usual, by the {\it {omega-limit set}} of a parametrized curve $p(s)$, $s\in {\mathbb R}$, taking values
in a  metric space, we mean the set of all accumulation points of our curve as $s\to\infty$.

We now use the following general statement.
\begin{proposition}
\label{omlimset}
Let $(\Omega, \mathcal{B})$ be a standard Borel space, and let $g_s$ be a measurable
flow on $\Omega$ preserving an ergodic Borel probability measure $\mu$. Let $Z$ be a separable metric space, and
let $\varphi: \Omega\rightarrow Z$\:be a measurable map such that for $\mu$-almost every $\omega\in \Omega$
the curve $\varphi(g_s\omega)$ is continuous in $s\in\mathbb{R}$. Then there exists a closed set
$\mathfrak{N}\subset Z$ such that for $\mu$-almost every $\omega\in \Omega$ the set $\mathfrak{N}$
is the omega-limit set of the curve $\varphi(g_s\omega)$, $s\in\mathbb{R}$.
\end{proposition}
The proof of Proposition \ref{omlimset} is routine.
We choose a countable base ${\mathscr U}=\{U_n\}_{n\in\mathbb{N}}$  of open sets in $Z$.
By ergodicity of $g_s$, continuity of the curves  $\varphi(g_s\omega)$ and countability of the family ${\mathscr U}$,
there exists a subset of full measure $\Omega^{\prime}\subset \Omega$, $\mu(\Omega^{\prime})=1$,
such that for any $U\in {\mathscr U}$ and any $\omega\in\Omega^{\prime}$
 the following conditions are satisfied:
\begin{enumerate}
\item if $\mu(U)>0$, then there exists an infinite sequence $s_n\rightarrow\infty$ such that $\varphi(g_{s_n}\omega)\in U$;
 \item if $\mu(U)=0$, then there exists $s_0>0$ such that $\varphi(g_s\omega)\notin U$ for all $s>s_0$.
\end{enumerate}

Now let $\mathfrak{N}$  be the set of all points $z\in Z$ such that $\mu(U)>0$ for any open set
$U\in {\mathscr U}$ containing the point $z$.
By construction, for any $\omega\in\Omega^{\prime}$, the set $\mathfrak{N}$
is precisely the omega-limit set of the curve $\varphi(g_s\omega)$.
The Proposition is proved.

Proposition \ref{omlimset} with $\Omega=\HH^{\prime}$, $\varphi={\mathcal D}_2^+$ and
$\mu$ an ergodic component of $\Prob^{\prime}$ together with the Limit Theorem given by Propositions
\ref{limthmmoduli-simple}, \ref{limthmmarkcomp-simple} immediately  implies Corollary \ref{nonentwo}.

\subsection{Limit theorems in the general case}
\subsubsection{Formulation and proof of the limit theorem}
We now proceed to the general case.
Let $\mathbb{P}\in {\mathscr P}^+$ be an ergodic $g_s$-invariant measure on
${\overline \Omega}$ and let $$\theta_1=1>\theta_2>\dots>\theta_{l_0}>0$$ be
the distinct positive Lyapunov exponents of $\overline{\mathbb{A}}$ with
respect to $\mathbb{P}$. We assume $l_0\geq 2$.

For $\overline{\omega}\in\overline{\Omega},$ let
$$E_{\overline{\omega}}^u=\mathbb{R}h_{\overline{\omega}}^{(0)}
+E_{2,\overline{\omega}}\oplus\dots\oplus
E_{l_0,\overline{\omega}}$$ be the corresponding direct-sum
decomposition into Oseledets subspaces, and let
$$\mathfrak{B}_{\overline{\omega}}^+=\mathbb{R}\nu_{\overline{\omega}}^+\oplus
\mathfrak{B}_{2,\overline{\omega}}^+\oplus\ldots\oplus
\mathfrak{B}_{l_0,\overline{\omega}}^+$$ be the corresponding
direct sum decomposition of the space
$\mathfrak{B}_{\overline{\omega}}^+.$

For $f\in Lip_{w}^+({{X}}(\overline{\omega}))$ we now write
$$\Phi_f^+=\Phi_{1,f}^++\Phi_{2,f}^++\ldots +\Phi_{l_0,f}^+,$$
where $\Phi_{i,f}^+\in\mathfrak{B}_{i,\overline{\omega}}^+$ and,
of course,
$$\Phi^+_{1,f}=(\int_{{{X}(\overline{\omega})}}fd\nu_{\overline{\omega}}
)\cdot\nu_{\overline{\omega}}^+.$$

For each $i=2,\ldots,l_0$ introduce a measurable fibre bundle
$${\bf{S}}^{(i)}\overline{\Omega}=\{(\overline{\omega},v):\overline{\omega}
\in\overline{\Omega},v\in E_{i,\overline{\omega}}^+,|v|=1\}.$$

The flow $g_s$ is naturally lifted to the space
${\bf{S}}^{(i)}\overline{\Omega}$ by the formula
$$
g_s^{{\bf{S}}^{(i)}}(\overline{\omega},v)=\left(g_s \overline{\omega},
\frac{\mathbb{A}(s,\overline{\omega})v}
{|\mathbb{A}(s,\overline{\omega})v|}\right).
$$

The growth of the norm of
vectors $v\in E_i^+$ is controlled by the multiplicative cocycle
$H_i$ over the flow $g_s^{\bf{S}^{(i)}}$ defined by the formula
$$
H_i(s, (\overline{\omega},v))=\frac{\mathbb{A}(s,\overline{\omega})v}
{|v|}.
$$
The growth of the variance of ergodic integrals is also,
similarly to the previous case, described by the cocycle $H_i.$

For
$\overline{\omega}\in\overline{\Omega}$ and $f\in
Lip_{w,0}^+({X}(\overline{\omega}))$ we write
\begin{equation}
\label{defif}
i(f)=\min\{j:\Phi_{f,j}^+\neq0\}.
\end{equation}
We now define a vector $v_f\in E_{i(f), {{\overline\omega}}}^u$ by the formula
\begin{equation}
\label{defvf}
\mathcal{I}_{{\overline\omega}}^+(v_f)=\frac{\Phi_{f,i(f)}^+}{|\Phi_{f,i(f)}^+|}.
\end{equation}

\begin{proposition} There exists $\alpha>0$
depending only on $\mathbb{P}$ and, for any $i=2,\ldots,l_0,$
positive measurable functions
$$V^{(i)}:{\bf{S}}^{(i)}\overline{\Omega}\to\mathbb{R}_+,C^{(i)}:\overline{\Omega}
\times\overline{\Omega}\to\mathbb{R}_+$$ such that for
$\mathbb{P}-almost$ every $\overline{\omega}\in\overline{\Omega}$, any  $f\in
Lip_{w,0}^+({X}(\overline{\omega}))$
and all $s>0$ we have
$$\left|\frac{Var_{\nu(\overline{\omega})}(\mathfrak{S}[f,e^s;1])}
{V^{(i(f))}(g_s^{{\bf{S}}^{(i)}}(\overline{\omega},v_{f}))(H_i(s, (\overline{\omega},v_{f})
))^2}-1\right|\leqslant
C^{(i)}(\overline{\omega},g_s\overline{\omega})e^{-\alpha s}.$$
\end{proposition}

Indeed, similarly to the case of a simple Lyapunov exponent, for $v\in E_{\oomega}^i$ we
write $\Phi_v^+=\mathcal{I}_{\overline{\omega}}^+(v)$ and set
$$V^{(i)}(\overline{\omega},v)=Var_{\nu(\overline{\omega})}
\Phi_v^+(x,1).$$

The Proposition follows now in the same way as in the case of the simple second Lyapunov exponent:
the pointwise approximation of the ergodic integral by the corresponding H{\"o}lder cocycle implies also
that the variances of these random variables are exponentially close.

We proceed to the formulation and the proof the limit theorem in
the general case. For $i=2,\ldots,l_0,$
introduce a map
$$\mathcal{D}_i^+:{\bf{S}}^{(i)}\overline{\Omega}\to\mathfrak{M}$$
by setting $\mathcal{D}_i^+(\overline{\omega},v)$ to be the
distribution of the $C[0,1]$-valued random variable
$$\frac{\Phi^+_v(x,\tau)}{\sqrt{Var_{\nu( \overline{\omega})}
(\Phi_v^+(x,1))}},\tau\in[0,1].$$

As before, by definition we have
$\mathcal{D}_i^+(\overline{\omega},v)\in\mathfrak{M}_1.$ The measure
$\mathfrak{m}[f,s]\in \MM$ is, as before, the distribution of the
$C[0,1]$-valued random variable
$$\frac
        {\int\limits_0^ {\tau
\exp(s)} f\circ h_t^+(x)dt}
{\sqrt
      {Var_ {\nu(\overline{\omega})}(\int\limits_0^ {
\exp(s)}f\circ h_t^+(x)dt)}},\ \tau\in[0,1].$$

Recall that $\mathscr{P}^+$ is the space of ergodic
$g_s$-invariant probability measures $\mathbb{P}$ on
$\overline{\Omega}$ having the form
$\mathbb{P}=\mathbb{P}_{\mu},\mu\in\mathcal{M}(\Gamma,\infty),\Gamma\in\mathfrak{G}.$

As before, for $\mathbb{P}\in\mathscr{P}^+,$ we let
$l_0=l_0(\mathbb{P})$ be the number of distinct positive Lyapunov
exponents of the measure $\mathbb{P}$. For $f\in
Lip_{w,0}^+({X}(\overline{\omega}))$ we define the number $i(f)$ by (\ref{defif})
and the vector $v_f$ by (\ref{defvf}).

\begin{theorem}\label{limthmmarkcomp} Let $ \mathbb{P}\in\mathscr{P}^+.$
There exists a constant $\alpha>0$ depending only on $\mathbb{P}$
and a positive measurable map
$C:\overline{\Omega}\times\overline{\Omega}\to\mathbb{R}_+$ such
that for $\mathbb{P}-almost$ every
$\overline{\omega}\in\overline{\Omega}$ and any $f\in
Lip_{w,0}^+({X}(\overline{\omega}))$ we have
$$d_{LP}(\mathfrak{m}[f,s],D_{i(f)}^+(g_s^{{\bf{S}}^{(i(f))}}(\overline{\omega},v_f)))
\leqslant C(\overline{\omega},g_s\overline{\omega})e^{-\alpha
s},$$
$$d_{KR}(\mathfrak{m}[f,s],D_{i(f)}^+(g_s^{{\bf{S}}^{(i(f))}}(\overline{\omega},v_f)))
\leqslant C(\overline{\omega},g_s\overline{\omega})e^{-\alpha
s}.
$$
\end{theorem}

The proof is similar to the proof of Proposition \ref{limthmmarkcomp-simple}.
Again, the ergodic integral is uniformly approximated by the
corresponding cocycle; the uniform bound on the difference yields
the uniform bound on the difference and the ratio of variances of
the ergodic integral and the cocycle considered as random
variables; we proceed, as before, by using the inequality
(\ref{ineqlimthm})
with $\xi_1=\mathfrak{m}[f,s],\ \xi_2=\Phi^+_{f,i(f)}(x, \tau)$, and
$M_1,M_2$ the corresponding normalizing variances. We conclude, again, by
noting that a uniform bound on the difference between two random
variables implies the same bound on the L{\'e}vy-Prohorov or
Kantorovich-Rubinstein distance between the distributions of the
random variables (using Lemma \ref{dist-images} in Appendix B).

\subsubsection{Atoms of limit distributions.}

\begin{proposition}
\label{bigatom}
Let $\omega\in\Omega$ satisfy $\la_{1}^{(0, \omega)}>1/2$. Then there exists a set ${\Pi}\subset X(\omega)$
such that
\begin{enumerate}
\item $\nu_{\omega}({\Pi})\geq (2\la_1^{(0, \omega)}-1)h_1^{(0, \omega)}$;
\item for any $\Phi^+\in {\mathfrak B}^+(X(\omega))$, the function $\Phi^+(x, h_1^{(0, \omega)})$ is constant
on ${\Pi}$.
\end{enumerate}
\end{proposition}

Proof. We consider $\omega$ fixed and omit it from notation.
As before, consider the flow transversal
$$
I=\{x\in X: x_n=\min\{e: I(e)=F(x_{n+1}) \ {\rm for \ all} \ n\leq 0\}.
$$
and decompose it into ``subintervals" $I_k=\{x\in I: I(x_0)=k\}$, $k=1, \dots, m$.

The transversal $I$ carries a natural conditional measure $\nu_I$ invariant under the first-return map of the flow $h_t^+$ on $I$.
The measure $\nu_I$ is given by the formula
$$
\nu_I(\{x\in I: x_1=e_1, \dots, x_n=e_n\})=\la^{(n+1)}_{I(e_n)},
$$
provided $I(e_k)=F(e_{k+1})$, $k=1, \dots, n$. In particular, $\nu_I(I_k)=\la^{(0)}_k$.
For brevity, denote $t_1=h_1^{(0)}$. By definition,  $h^+_{t_1}I_1\subset I$ and we have
$$
\nu_I(I_1\bigcap h^+_{t_1}I_1)\geq 2\la_1^{(0)}-1>0.
$$
Introduce the set
$$
\Pi=\{h^+_{\tau}x, 0<\tau<t_1, x\in I_1, h_{t_1}x\in I_1\}.
$$
The first statement of the Proposition is clear, and we proceed to the proof of the second.
Note first that for any $\Phi^+\in {\mathfrak B}^+(X(\omega))$ and any $\tau, 0\leq\tau\leq t_1$
the quantity $\Phi^+(x, \tau)$
is constant as long as $x$ varies in $I_1$.

Fix $\Phi^+\in {\mathfrak B}^+(X(\omega))$ and take an arbitrary ${\tilde x}\in \Pi$.
Write ${\tilde x}=h^+_{\tau_1}x_1$, where $x_1\in I_1$,
$0<\tau_1<t_1$. We have $h^+_{t_1-\tau_1}{\tilde x}\in I_1$, whence
$$
\Phi^+(h^+_{t_1-\tau_1}{\tilde x}, \tau_1)=\Phi^+(x_1, \tau_1)
$$
and
$$
\Phi^+({\tilde x}, t_1)=\Phi^+({\tilde x}, t_1-\tau_1)+\Phi^+(h^+_{t_1-\tau_1}{\tilde x}, \tau_1)=
\Phi^+(h^+_{\tau_1}x_1, t_1-\tau_1)+\Phi^+(x_1, \tau_1)=\Phi^+(x_1, t_1),
$$
and the Proposition is proved. We illustrate the proof by Figure \ref{fig:two}.
\begin{figure}
\begin{center}
\includegraphics{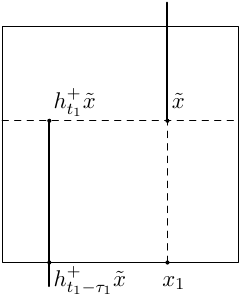}\\
\caption{Atoms of limit distributions}\label{fig:two}
\end{center}
\end{figure}

\begin{proposition}
Let ${\tilde x}, {\hat x}\in X$ be distinct and satisfy ${\tilde x}\in \gamma_{\infty}^+({\hat x})$,
${\tilde x}\in \gamma_{\infty}^-({\hat x})$. Let $t_0$ be such that ${\hat x}=h_{t_0}^+{\tilde x}$.
Then there exists a set $\Pi$ of positive measure such that
for any $x\in\Pi$ and any $\Phi^+\in\B^+(X)$ we have
$$
\Phi^+(x,t_0)=\Phi^+({\tilde x}, t_0).
$$
\end{proposition}

Proof. There exist $n_0, n_1\in {\mathbb Z}$ such that
$$
{\tilde x}_t={\hat x}_t, \ t\in (-\infty, n_0]\cup [n_1, \infty).
$$
Let $\Pi$ be the set of all $x^{\prime}$ satisfying the equality $x^{\prime}_t={\tilde x}_t$ for $t\in (n_0, n_1)$,
and, additionally, the equalities
$$
F(x^{\prime}_{n_1-1})=F({\tilde x}_{n_1-1}), \ I(x^{\prime}_{n_0})=I({\tilde x}_{n_0}),
$$
By holonomy-invariance, for any $x^{\prime}\in\Pi$ and any $\Phi^+\in\B^+(X)$ we have
$$
\Phi^+(x^{\prime},t_0)=\Phi^+({\tilde x}, t_0).
$$
To estimate from below the measure of the set $\Pi$, set $I({\tilde x}_{n_0})=i$, $F({\tilde x}_{n_1-1})=j$ and note that
by definition we have
$$
\nu(\Pi)\geq \la^{(n_0)}_{i}h^{(n_1-1)}_j.
$$
For a fixed $\oomega$, the set of ``homoclinic times'' $t_0$ for which there exist ${\tilde x}, {\hat x}\in X$  satisfying
${\tilde x}\in \gamma_{\infty}^+({\hat x})$,
${\tilde x}\in \gamma_{\infty}^-({\hat x})$, ${\hat x}=h_{t_0}^+{\tilde x}$, is countable and dense in ${\mathbb R}$.
\begin{corollary}
For almost every $\oomega\in\Oomega$, there exists a dense set of times $t_0\in {\mathbb R}$ such that for
any $\Phi^+\in\B^+$  the distribution of the random variable $\Phi^+(x, t_0)$ has an atom.
\end{corollary}

\subsubsection{Accumulation at zero for limit distributions}

Recall that for $\oomega^{\prime}\in\Oomega^{\prime}$, $\Phi^+\in\BB^+_{\oomega}$, $\Phi^+\neq 0$,
and $\tau\in {\mathbb R}$, $\tau\neq 0$,
the measure ${\mathfrak m}[\Phi^+, \tau]$ is the distribution
of the normalized ${\mathbb R}$-valued random variable
$$
\frac{\Phi^+(x, \tau)}{\sqrt{Var_{{\nu}}\Phi^+(x, \tau)}}.
$$
As before, let $\MM({\mathbb R})$ be the space of probability measures on ${\mathbb R}$ endowed with the weak topology,
and let $\delta_0\in\MM({\mathbb R})$ stand for the delta-measure at zero.
Similarly to the Introduction, we need the following additional assumption on the measure
$\Prob\in {\mathscr P}^+$.
\begin{assumption}
\label{fatzip}
For any $\varepsilon>0$ we have
$$
\Prob(\{\oomega: \la_1^{(\oomega)}>1-\varepsilon, h_1^{(\oomega)}>1-\varepsilon \})>0.
$$
\end{assumption}

By Proposition \ref{bigatom}, in view of the ergodicity of $\Prob$, for almost every $\oomega\in\Oomega$
and every $\Phi^+\in\BB^+_{\oomega}$, $\Phi^+\neq 0$, the sequence of measures ${\mathfrak m}[\Phi^+, \tau]$
admits atoms of size arbitrarily close to $1$.
The next simple Proposition shows that the corresponding measures must then accumulate at {\it zero}
(rather than at any other point of the real line).

\begin{proposition} \label{atomatzero}
Let $\mu_0$ be a probability
measure on $\mathbb{R}$ such that
$$\int_{\mathbb{R}}xd\mu_0(x)=0,\ \ \int_{\mathbb{R}}x^2d\mu_0(x)=1.$$
Let $x_0\in\mathbb{R}$ and assume that
$$\mu_0(\{x_0\})=\beta.$$
Then $$|x_0|^2\leqslant\frac{1-\beta}{\beta^2}.$$
\end{proposition}

Proof. If $x_0=0$, then there is nothing to
prove, so assume $x_0>0$ (the remaining case $x_0<0$ follows by
symmetry). We have $$\int_0^{+\infty}xd\mu_0(x)\geqslant\beta
x_0,$$ and, consequently,
$$\int_{-\infty}^0xd\mu_0(x)\leqslant-\beta x_0.$$

Using the Cauchy-Bunyakovsky-Schwarz inequality, write
$$\frac{1}{\mu_0((-\infty,0))}\int_{-\infty}^0x^2d\mu_0(x)\geqslant\left(\frac{1}{\mu_0((-\infty,0))}\int_{-\infty}^0xd\mu_0(x)\right)^2,$$
whence, recalling that the variance of $\mu_0$ is equal to $1$, we obtain
$$\mu_0((-\infty,0))\geqslant\left(\int_{-\infty}^0xd\mu_0(x)\right)^2$$
and, finally, $$1-\beta\geqslant\beta^2x_0^2,$$ which is what we
had to prove.

As before, the
symbol $\Rightarrow$ denotes weak convergence of probability
measures.

\begin{proposition}\label{convtodelta} Let $\mathbb{P}\in {\mathscr P}$ be an ergodic
${g}_s$ invariant measure on $\Oomega$ satisfying
Assumption \ref{fatzip}. Then for $\mathbb{P}-almost$
every ${{X}}\in\Oomega$ there exists a sequence
$\tau_n\in\mathbb{R}_+$ such that for any
$\Phi^+\in{\mathfrak{B}}^+({{X}})$ we have
$$\mathfrak{m}[\Phi^+,\tau_n]\Rightarrow\delta_0\ \mathrm{as}\
n\to\infty.$$
\end{proposition}

This is immediate from Proposition \ref{bigatom} and Proposition \ref{atomatzero}.

\begin{corollary}\label{nonconvergence}
Let $\mathbb{P}\in {\mathscr P}$ be an ergodic
${g}_s$-invariant measure on $\Oomega$ satisfying
Assumption \ref{fatzip}. Then for $\mathbb{P}$-almost
every ${{X}}\in\Oomega$ there exists a sequence
$s_n\in\mathbb{R}_+$ such that for any $f\in
Lip_{w,0}^+({{X}})$ satisfying $\Phi_f^+\neq0$ we have
$$\mathfrak{m}[f,s_n;1]\Rightarrow\delta_0\ \mathrm{as}\
n\to\infty.$$

Consequently, if $f\in Lip_{w,0}^+({{X}})$ satisfies
$\Phi_f^+\neq0,$ then the family of measures
$\mathfrak{m}[f,s;1]$ does not converge in the weak topology on $\MM({\mathbb R})$ as
$s\to\infty$ and the family of measures
$\mathfrak{m}[f,s]$ does not converge in the weak topology on $\MM(C[0,1])$ as
$s\to\infty$.
\end{corollary}

Proof. The first claim is clear from
Proposition \ref{convtodelta} and the Limit Theorem \ref{limthmmarkcomp}.
The second claim is obtained from the Limit
Theorem \ref{limthmmarkcomp}  in the following way.

First note that the set
\begin{equation}
\label{phitausphere}
\{\mathfrak{m}[\Phi^+,1], \Phi^+\in\mathfrak{B}^+({{X}}),|\Phi^+|=1\}
\end{equation}
is compact in the weak topology (indeed, it is clear from the {\it uniform}
convergence on spheres in the Oseledets Theorem
that the map $$\Phi^+\rightarrow\mathfrak{m}[\Phi^+,\tau]$$ is
continuous in restriction to the set $\{\Phi^+: |\Phi^+|=1\}$ whose image is therefore compact).
In particular, the set (\ref{phitausphere}) is bounded away from $\delta_0$, and the function
$$
\kappa(\oomega)=\inf\limits_{{\Phi^+: |\Phi^+|=1}} d_{LP}(\mm[\Phi^+, 1], \delta_0)
$$
is a positive measurable function on $\Oomega$. Consequently, there exists $\kappa_0>0$ such that
$$
\Prob(\{\oomega: \kappa(\oomega)>\kappa_0\})>0.
$$
From ergodicity of the measure $\Prob$ and the Limit  Theorem \ref{limthmmarkcomp} it follows that
the family  $\mathfrak{m}[f,s;1]$, $s\in {\mathbb R}$, does not converge to $\delta_0$.
On the other hand, as we have seen, the measure $\delta_0$
is an accumulation point for the family. It follows that the measures  $\mathfrak{m}[f,s;1]$ do
not converge in $\MM({\mathbb R})$ as $s\to\infty$, and, a fortiori,
that the measures $\mathfrak{m}[f,s]$ do not converge in $\MM(C[0,1])$ as $s\to\infty$.

\subsection{Ergodic averages of Vershik's automorphisms.}
\subsubsection{The space of one-sided Markov compacta.}
Recall that to a sequence $\{\Gamma_n\}$, $n\in {\mathbb N}$, of graphs belonging to ${\mathfrak G}$, we assign the one-sided Markov compactum of infinite one-sided paths in our sequence of graphs:
\begin{equation}
\label{markcomponeside}
Y=\{y= y_1\dots y_n \dots,\  y_n\in {\cal E}(\Gamma_n), F(y_{n+1})=I(y_n)\}.
\end{equation}

As before, we write $A_n(Y)=A(\Gamma_n)$.

The Markov compactum $Y$ is endowed with a natural  tail equivalence relation: $y\sim y^{\prime}$ if
there exists $n_0$ such that $y_n=y^{\prime}_n$ for all $n>n_0$.

Cylinders in $Y$ are subsets of the form $\{y: y_{n+1}=e_1, \dots, y_{n+k}=e_{k}\}$, where
$n\in {\mathbb N}$, $k\in {\mathbb N}$,
$e_1\in {\cal E}(\Gamma_{n+1}), \dots,  e_k\in {\cal E}(\Gamma_{n+k})$ and
$F(e_i)=I(e_{i+1})$.
The family of all cylinders forms a semi-ring which we denote by ${\mathfrak C}(Y)$.

Let $\VV(Y)$ be the vector space of all real-valued finitely-additive measures $\Phi$ defined on
the semi-ring ${\mathfrak C}(Y)$ and invariant under the tail equivalence relation in the following precise sense:  if $e_1\in {\cal E}(\Gamma_{1}), \dots,  e_k\in {\cal E}(\Gamma_{k})$ and
$F(e_i)=I(e_{i+1})$, then the measure
$$
\Phi(\{y: y_1=e_1, \dots, y_k=e_k\})
$$
only depends on $I(e_k)$.

As before, a sequence of  vectors ${\bf v}=v^{(l)}$, $v^{(l)}\in {\mathbb R}^m$, $l\in {\mathbb N}$ satisfying
$$
v^{(l)}=A^t_lv^{(l+1)},
$$
will be called {\it reverse equivariant}.
By definition, the vector space of all reverse equivariant sequences is isomorphic to $\VV(Y)$.

Now, in analogy to the bi-infinite case,  let $\Omega_+$ be the space of sequences of one-sided
infinite sequences of graphs $\Gamma_n\in {\mathfrak G}$. As before, we write
$$
\Omega_+=\{\omega=\omega_{1}\dots\omega_n\dots, \omega_i\in {\mathfrak G}, i\in {\mathbb N} \},
$$
and, for $\omega\in\Omega$, we denote by $Y(\omega)$ the Markov compactum corresponding to $\omega$.
The right shift $\sigma$ on the  space $\Omega$ is defined
by the formula $(\sigma\omega)_n=\omega_{n+1}$.

Let $\mu$ be an ergodic $\sigma$-invariant probability measure on $\Omega_+$
whose natural extension to the space $\Omega$ satisfies Assumption \ref{asos}.
Again we have two natural cocycles over the system
$(\Omega_+, \sigma, \mu)$:
\begin{enumerate}
\item the renormalization cocycle ${\mathbb A}$
defined, for $n>0$, by the formula
$$
{\mathbb A}(n,\omega)=A(\omega_{n})\dots A(\omega_1).
$$
\item the reverse transpose cocycle ${\mathbb A}^{-t}$ defined, for $n>0$, by the formula
$$
{\mathbb A}^{-t}(n,\omega)=(A^t)^{-1}(\omega_{n})\dots (A^t)^{-1}(\omega_1).
$$
\end{enumerate}

Our assumptions imply that both these cocycles satisfy the Oseledets Theorem.

For $\omega\in\Omega_+$, let ${\check E}_{\omega}$ be the strictly stable Lyapunov subspace of the
cocycle ${\mathbb A}^{-t}$. Each $v\in {\check E}_{\omega}$ gives rise to a reverse-equivariant sequence of vectors and thus to a measure $\Phi_v\in\VV(Y)$; denote
$$
{\mathfrak B}(\omega)={\mathfrak B}(Y(\omega))=\{\Phi_v,\  v\in {\check E}_{\omega}\}.
$$

Let $-\theta_{l_0}>-\theta_{l_0-1}>\dots>-\theta_1$ be the distinct
negative Lyapunov exponents of the inverse transpose cocycle ${\mathbb A}^{-t}$.
By the Oseledets Theorem, for almost every $\omega$ we have the corresponding  flag of subspaces
$$
{\check E}_{\theta_1}\subset {\check E}_{\theta_2}\subset \dots \subset {\check E}_{\theta_{l_0}},
$$
where for any $v\in {\check E}_{\theta_i}\setminus {\check E}_{\theta_{i-1}}$ we have
$$
\lim\limits_{n\to\infty} \frac{\log |{\mathbb A}^{-t}(n, \omega)v|}{n}=-\theta_i.
$$
Set
$$
{\mathfrak B}_{\theta_i}(\omega)={\mathfrak B}_{\theta_i}(Y(\omega))=\{\Phi_v,\  v\in {\check E}_{\theta_i, \omega}\}.
$$

For instance, our assumptions imply that  almost every $\omega$ admits a unique (up to scaling)
positive reverse equivariant sequence; the space $\VV(Y(\omega))$ then contains a unique positive countably additive probability measure $\nu_{\omega}$; we have then $\nu_{\omega}\in {\mathfrak B}_{\theta_1}(Y(\omega))$.

As before, a bounded measurable function $f:Y\to {\mathbb R}$ will be called {\it weakly Lipschitz}
if there exists a constant $C>0$ such that for any cylinder ${\cal C}\in {\mathfrak C}(Y)$ and any
points $y(1), y(2)\in {\cal C}$ we have
\begin{equation}
\label{weakliponeside}
|f(y(1))-f(y(2))|\leq C\nu_{\omega}({\cal C}).
\end{equation}
If $C_f$ is the infimum of all constants in the right-hand side of (\ref{weakliponeside}), then
the norm of $f$ is given by
$$
||f||_{Lip_w}=C_f+\sup_Y |f|.
$$

For almost all $\omega$, all $f\in Lip_w(Y(\omega))$ and all
$\Phi\in {\mathfrak B}(Y(\omega))$, the Riemann-Stieltjes integral $\int_Y fd\Phi$ is well-defined.
Given $f\in Lip_w(Y(\omega))$, set

$$
\theta(f)=\max \{\theta_i: \ {\rm there \ exists} \ \Phi\in {\mathfrak B}_{\theta_i}(Y(\omega)) \  {\rm such \ that} \
\int_Y fd\Phi\neq 0\}
$$
(if $\int fd\Phi=0$ for all $\Phi\in {\mathfrak B}(Y(\omega))$, then we set $\theta(f)=0$).

\subsubsection{Vershik's automorphisms.}

As before, assume that for almost every $\omega\in\Omega_+$ there is a
Vershik's ordering $\oo(\omega)$
on the edges of each graph $\omega_n$, $n\in {\mathbb N}$. Furthermore,
assume that   the ordering is shift-invariant in the following
sense: the ordering $\oo(\omega)$ on the edges of the graph
$\omega_{n+1}$ is the same as the ordering $\oo(\sigma\omega)$
on the edges of the graph $(\sigma\omega)_n=\omega_{n+1}$.
Almost every Markov compactum $Y(\omega)$ is now endowed with a Vershik's
automorphism $T_Y$ with respect to the ordering $\oo(\omega)$, and
Theorem \ref{symbmultiplic} implies
\begin{corollary}
Let $\mu$ be an ergodic $\sigma$-invariant probability measure on $\Omega_+$
whose natural extension to the space $\Omega$ satisfies Assumption \ref{asos}.
Then for $\mu$-almost any $\omega\in\Omega_+$, any $f\in Lip_w(Y(\omega))$ and any $y\in Y(\omega)$
we have
\begin{equation}
\limsup\limits_{N\to\infty} \frac{\log\left|\sum\limits_{k=0}^{N-1} f(T_Y^ky)\right|}{\log N}=\theta(f).
\end{equation}
\end{corollary}

Note that for any $\omega\in\Omega$ the automorphism $T_{Y(\sigma\omega)}$ on $Y(\sigma\omega)$ can be realized as
an induced automorphism of $T_{Y(\omega)}$ on $Y(\omega)$ in the following way.
Take $\Gamma(\omega_1)$ and let ${\cal E}_{min}(\Gamma(\omega_1))$ be the set of minimal
edges with respect to the ordering $\oo$. Consider the subset $Y^{\prime}(\omega)\subset Y(\omega)$ given by
$$
Y^{\prime}(\omega)=\{x\in Y(\omega): x_1\in {\cal E}_{min}(\Gamma(\omega_1))\}.
$$
The shift $\sigma$ maps $Y^{\prime}(\omega)$ bijectively onto $Y(\sigma\omega)$; the induced map
of $T_{Y(\omega)}$ on $Y^{\prime}(\omega)$ is isomorphic to $T_{Y(\sigma\omega)}$.
For $\Phi\in {\mathfrak V}(Y(\omega))$, consider its restriction $\Phi|_{Y^{\prime}(\omega)}$
; we have $\sigma_*\left(\Phi|_{Y^{\prime}(\omega)}\right)\in {\mathfrak V}(Y(\sigma\omega))$, and if
$\Phi\in {\mathfrak B}(Y(\omega))$, then we have $\sigma_*\left(\Phi|_{Y^{\prime}(\omega)}\right)\in {\mathfrak B}(Y(\sigma\omega))$.

\section{Markov Compacta and  Abelian Differentials.}
\subsection{The mapping into cohomology.}

First we show that for an arbitrary  abelian differential ${\bf X}=(M, \omega)$ the  map
$$
{\check {\cal I}}_{\bf X}:\BB_c^+(M, \omega)\to H^1(M, {\mathbb R})
$$
given by (\ref{maptocohomology})  is well-defined.
Then, in the following subsections we show that for almost all $(M, \omega)$ the image of $\BB^+({\bf X})$ under this mapping
is the unstable space of the Kontsevich-Zorich cocycle.

\begin{proposition}
\label{acthomology}
Let $\gamma_i$, $i=1, \dots, k$,  be rectangular closed curves such that the cycle $\sum_{i=1}^k \gamma_i$
is homologous to $0$. Then for any $\Phi^+\in\BB^+$ we have
$$
\sum_{i=1}^{k} \Phi^+(\gamma_i)=0.
$$
\end{proposition}
Informally, Proposition \ref{acthomology} states that the {\it relative} homology of the surface with respect
to zeros of the form is not needed for the description of cocycles.
Arguments of this type for invariant measures
of translation flows go back to Katok's work \cite{katok}.

We proceed to the formal proof.
Take a fundamental polygon $\Pi$ for $M$ such that all its sides are
simple
closed rectangular curves on $M$. Let $\del\Pi$ be
the boundary of $\Pi$, oriented counterclockwise. By definition,
\begin{equation}
\label{deltapizeroo}
\Phi^+(\del\Pi)=0,
\end{equation}
since each curve of the boundary enters $\del\Pi$ twice and with
opposite signs.

We now deform the curves $\gamma_i$ to the boundary $\del\Pi$ of our fundamental polygon.
\begin{proposition}
\label{simplezero}
Let $\gamma\subset \Pi$ be a simple rectangular closed curve.
Then
$$
\Phi^+(\gamma)=0.
$$
\end{proposition}

Proof of Proposition \ref{simplezero}.

We may assume that $\gamma$ is oriented counterclockwise and does not
contain zeros of the form $\omega$. By Jordan's theorem, $\gamma$ is the boundary of a domain
$N\subset\Pi$. Let $p_1, \dots, p_r$ be zeros of $\omega$ lying inside
$N$; let $\kappa_i$ be the order of $p_i$. Choose an arbitrary
$\varepsilon>0$, take $\delta>0$ such that $|\Phi^+(\gamma)|\leq\varepsilon$ as soon as the length of $\gamma$ does not exceed $\delta$
 and consider a partition of $N$ given by
\begin{equation}
N=\Pi_1^{(\varepsilon)}\bigsqcup\dots\bigsqcup
\Pi_n^{(\varepsilon)}\bigsqcup
{\tilde \Pi}_1^{(\varepsilon)}\bigsqcup\dots\bigsqcup
{\tilde \Pi}_{r}^{(\varepsilon)},
 \end{equation}
where all
$\Pi_i^{(\varepsilon)}$ are admissible rectangles and
${\tilde \Pi}_i^{(\varepsilon)}$ is a $4(\kappa_i+1)$-gon
containing $p_i$ and no other zeros and satisfying the additional
assumption that all its sides are no longer than $\delta$.
Let $\del\Pi_i^{(\varepsilon)}$, $\del{\tilde \Pi}_i^{(\varepsilon)}$
 stand for the boundaries of our polygons oriented counterclockwise.

We have
$$
\Phi^+(\gamma)=\sum \Phi^+(\del\Pi_i^{(\varepsilon)})+
\sum\Phi^+(\del{\tilde \Pi}_i^{(\varepsilon)}).
$$
In the first sum, each term is equal to $0$ by definition of $\Phi^+$,
whereas the second sum does not exceed, in absolute value, the quantity
$$
C(\kappa_1, \dots, \kappa_r)\varepsilon,
$$
where $C(\kappa_1, \dots, \kappa_r)$ is a positive constant depending only
on $\kappa_1, \dots, \kappa_r$.
Since $\varepsilon$ may be chosen arbitrarily small, we have
$$
\Phi^+(\gamma)=0,
$$
which is what we had to prove.

For $A,B\in\del\Pi$, let $\del\Pi_A^B$ be the part of $\del\Pi$ going
counterclockwise from $A$ to $B$.
\begin{proposition}
Let $A,B\in\del\Pi$ and let $\gamma\subset\Pi$ be an
arbitrary rectangular curve going from $A$ to $B$.
Then
$$
\Phi^+(\del\Pi_A^B)=\Phi^+(\gamma).
$$
\end{proposition}

We may assume that $\gamma$ is simple in $\Pi$, since,
by Proposition \ref{simplezero}, self-intersections of $\gamma$ (whose
number is finite) do not change the value of $\Phi^+(\gamma)$. If $\gamma$
is simple, then $\gamma$ and $\Phi^+(\del\Pi_B^A)$ together form a simple
closed curve, and the proposition follows from Proposition
\ref{simplezero}.

\begin{corollary}
If $\gamma\subset \Pi$ is a rectangular curve which yields a closed curve
in $M$
homologous to zero in $M$, then
$$
\Phi^+(\gamma)=0.
$$
\end{corollary}

Indeed, by the previous proposition we need only consider the case when $\gamma\subset
\del\Pi$. Since $\gamma$ is homologous to $0$ by assumption, the cycle $\gamma$ is in fact
a multiple of the cycle $\del\Pi$, for which the statement follows from (\ref{deltapizeroo}).

\subsection{Veech's space of zippered rectangles}

\subsubsection{Rauzy-Veech induction}

To establish the link between Markov compacta and abelian differntials, we use the expansions of interval exchange
transformations given by the Rauzy-Veech induction and the Veech representation of abelian differentials by zippered rectangles.
For a different presentation of the Rauzy-Veech formalism, see Marmi-Moussa-Yoccoz \cite{MMY}.

Let $\pi$ be a permutation of $m$ symbols, which will always be
assumed irreducible in the sense that $\pi\{1,\dots,k\}=\{1,\dots,k\}$ implies $k=m$.
The Rauzy operations $a$ and $b$ are defined by the formulas
$$
a\pi(j)=\begin{cases}
\pi j,&\text{if $j\leq\pi^{-1}m$,}\\
\pi m,&\text{if $j=\pi^{-1}m+1$,}\\
\pi(j-1),&\text{if $\pi^{-1}m+1<j\le m$;}
\end{cases}
$$
$$
b\pi(j)=\begin{cases}
\pi j,&\text{if $\pi j\leq \pi m$,}\\
\pi j+1,&\text{if $\pi m<\pi j<m$,}\\
\pi m+1,&\text{ if $\pi j=m$.}
\end{cases}
$$

These operations preserve irreducibility. The {\it Rauzy class}
$\mathcal R(\pi)$ is defined as the set of all permutations that can
be obtained from $\pi$ by application of the transformation group
generated by $a$ and $b$. From now on we fix a Rauzy class $\R$ and
assume that it consists of irreducible permutations.

For $i,j=1,\dots,m$, denote by $E^{ij}$ the $m\times m$ matrix whose
$(i,j)th$ entry is $1$, while all others are zeros. Let $E$ be the
identity $m\times m$-matrix. Following Veech \cite{veech}, introduce
the unimodular matrices
\begin{equation} \label{mat_a}
{\cal A}(a,\pi)=\sum_{i=1}^{\pi^{-1}m}E^{ii}+E^{m,\pi^{-1}m+1}+
\sum_{i=\pi^{-1}m}^{m-1}E^{i,i+1},
\end{equation}
\begin{equation} \label{mat_b}
{\cal A}(b,\pi)=E+E^{m,\pi^{-1}m}.
\end{equation}
For a vector $\la=(\la_1,\dots,\la_m)\in{\mathbb R}^m$, we write
$$
|\la|=\sum_{i=1}^m\la_i.
$$
Let
$$
\Delta_{m-1}=\{\la\in {\mathbb R}^m:|\la|=1,\ \la_i>0 \text{ for }
i=1,\dots,m\}.
$$

One can identify each pair $(\lambda,\pi)$,
$\lambda\in\Delta_{m-1}$, with the {\it interval exchange map} of
the interval $I:=[0,1)$ as follows. Divide $I$ into the
sub-intervals $I_k:=[\beta_{k-1},\beta_k)$, where $\beta_0=0$,
$\beta_k=\sum_{i=1}^k\lambda_i$, $1\le k\le m$, and then place the
intervals $I_k$ in $I$ in the following order (from left to write):
$I_{\pi^{-1}1},\dots, I_{\pi^{-1}m}$. We obtain a piecewise linear
transformation of $I$ that preserves the Lebesgue measure.

The space $\Delta(\R)$ of interval exchange maps corresponding to
$\R$ is defined by
$$
\Delta(\R)=\Delta_{m-1}\times\R.
$$
Denote
$$
\Delta_{\pi}^+=\{\la\in\Delta_{m-1}| \ \la_{\pi^{-1}m}>\la_m\},\ \
\Delta_{\pi}^-=\{\la\in\Delta_{m-1}| \ \la_m>\la_{\pi^{-1}m}\},
$$
$$
\Delta^+(\R)=\cup_{\pi\in{\cal R}}\{(\pi,\la)|\
\la\in\Delta_{\pi}^+\},$$
$$
\Delta^-(\R)=\cup_{\pi\in{\cal R}}\{(\pi,\la)|\
\la\in\Delta_{\pi}^-\},$$
$$
\Delta^\pm(\R)=\Delta^+(\R)\cup\Delta^-(\R).
$$
The {\it Rauzy-Veech induction map} ${\mathscr
T}:\Delta^\pm(\R)\to\Delta(\R)$ is defined as follows:
\begin{equation}
\label{te} {\mathscr T}(\la,\pi)=\begin{cases}
(\frac{{\mathcal A}(a,\,\pi)^{-1}\la}{|{\mathcal A}(a,\,\pi)^{-1}\la|},a\pi), &\text{if
$\la\in\Delta_\pi^+$,}\\
(\frac{{\mathcal A}(b,\,\pi)^{-1}\la}{|{\mathcal A}(b,\,\pi)^{-1}\la|},\,b\pi), &\text{if
$\la\in\Delta_\pi^-$} .
\end{cases}
\end{equation}

One can check that ${\mathscr T}(\la,\pi)$ is the interval exchange map
induced by $(\la,\pi)$ on the interval $J=[0,1-\gamma]$, where
$\gamma=\min(\la_m,\la_{\pi^{-1}m})$; the interval $J$ is then stretched to
unit length.

Denote
\begin{equation}
\label{delta+-} \Delta^\infty(\mathcal R)=\bigcap_{n\ge0}\mathscr
T^{-n} \Delta^\pm(\mathcal R).
\end{equation}
Every ${\mathscr T}$-invariant probability measure is concentrated on
$\Delta^\infty(\mathcal R)$. On the other hand, a natural Lebesgue
measure defined on $\Delta(\mathcal R)$, which is finite, but
non-invariant, is also concentrated on $\Delta^\infty(\mathcal R)$.
Veech \cite{veech} showed that $\mathscr T$ has an absolutely
continuous ergodic invariant measure on $\Delta(\R)$, which is,
however, infinite.

We have two matrix cocycles ${\mathcal A}^t$, ${\mathcal A}^{-1}$over $\mathscr T$ defined by
$$
{\mathcal A}^t(n,(\lambda, \pi))={\mathcal A}^t({\mathscr T}^n(\lambda, \pi))\cdot \ldots \cdot {\mathcal A}^t(\lambda,\pi),
$$
$$
{\mathcal A}^{-1}(n,(\lambda, \pi))={\mathcal A}^{-1}({\mathscr T}^n(\lambda, \pi))\cdot \ldots \cdot {\mathcal A}^{-1}(\lambda,\pi).
$$

We introduce the corresponding skew-product transformations
${\mathscr T}^{{\cal A}^t}: \Delta(\R)\times {\mathbb R}^m\to\Delta(\R)\times {\mathbb R}^m$,
${\mathscr T}^{{\cal A}^{-1}}: \Delta(\R)\times {\mathbb R}^m\to\Delta(\R)\times {\mathbb R}^m$,
$$
{\mathscr T}^{{\cal A}^t}((\la,\pi), v)=({\mathscr T}(\la,\pi), {\cal A}^t(\la,\pi)v);
$$
$$
{\mathscr T}^{{\cal A}^{-1}}((\la,\pi), v)=({\mathscr T}(\la,\pi), {\cal A}^{-1}(\la,\pi)v).
$$

\subsubsection{The construction of zippered rectangles}
\label{zipper}

Here we briefly recall the construction of the Veech space of
zippered rectangles. We use the notation of \cite{bufetov}.

{\it Zippered rectangles} associated to the Rauzy class $\R$ are
triples $(\la,\pi,\delta)$, where
$\la=(\la_1,\dots,\la_m)\in{\mathbb R}^m$, $\la_i>0$, $\pi\in{\cal
R}$, $\delta=(\delta_1,\dots,\delta_m)\in{\mathbb R}^m$, and the
vector $\delta$ satisfies the following inequalities:
\begin{equation}
\label{deltaone} \delta_1+\dots+\delta_i\leq 0,\ \ i=1,\dots,m-1,
\end{equation}
\begin{equation}
\label{deltatwo} \delta_{\pi^{-1}\,1}+\dots+\delta_{\pi^{-1}\,i}\geq
0, \ \ i=1, \dots, m-1.
\end{equation}
The set of all vectors $\delta$ satisfying (\ref{deltaone}),
(\ref{deltatwo}) is a cone in ${\mathbb R}^m$; we denote it by
$K(\pi)$.

For any $i=1, \dots, m$, set
\begin{equation}
\label{adelta}
a_j=a_j(\delta)=-\delta_1-\dots-\delta_j, \
h_{j}=h_j(\pi,\delta)=-\sum_{i=1}^{j-1} \delta_i+\sum_{l=1}^{\pi(j)-1}\delta_{\pi^{-1}l}.
\end{equation}

\subsubsection{Zippered rectangles and abelian differentials.}

Given a zippered rectangle $(\la,\pi,\delta)$, Veech \cite{veech} takes $m$ rectangles
$\Pi_i=\Pi_i(\la,\pi,\delta)$ of girth $\la_i$ and height $h_i$, $i=1, \dots, m$, and glues them together according to
a rule determined by the permutation $\pi$. This procedure yields a Riemann surface $M$ endowed with a holomorphic
$1$-form $\omega$ which, in restriction to each $\Pi_i$, is simply the form $dz=dx+idy$.
The union of the bases of the rectangles is an interval $I^{(0)}(\la,\pi, \delta)$ of length $|\la|$ on $M$; the
first return map of the vertical flow of the form $\omega$ is precisely the interval exchange ${\bf T}_{(\la,\pi)}$.

The {\it area} of a zippered rectangle $(\la,\pi,\delta)$ is given
by the expression
\begin{equation}
\label{area} Area\,(\la,\pi,\delta):=\sum_{r=1}^m\la_rh_r=
\sum_{r=1}^m\la_r(-\sum_{i=1}^{r-1}\delta_i+\sum_{i=1}^{\pi
r-1}\delta_{\pi^{-1}\,i}).
\end{equation}
(Our convention is $\sum_{i=u}^v...=0$ when $u>v$.)

Furthermore, to each rectangle $\Pi_i$ Veech \cite{veechamj} assigns a cycle $\gamma_i(\la,\pi, \delta)$ in the homology
group $H_1(M, {\mathbb Z})$: namely, if $P_i$ is the left bottom corner of $\Pi_i$ and $Q_i$ the left top corner, then the cycle is
the union of the vertical interval $P_iQ_i$ and the horizontal subinterval of $I^{(0)}(\la,\pi, \delta)$ joining $Q_i$ to $P_i$. It is
clear that the cycles $\gamma_i(\la,\pi, \delta)$ span $H_1(M, {\mathbb Z})$.

\subsubsection{The space of zippered rectangles.}
Denote by ${\cal V}(\R)$ the space of all zippered rectangles
corresponding to the Rauzy class $\R$, i.e.,
$$
{\mathcal V}(\R)=\{(\la,\pi,\delta):\la\in{\mathbb
R}^m_+,\,\pi\in\R,\,\delta\in K(\pi)\}.
$$
Let also
$$
{\mathcal V}^+(\R)=\{(\la,\pi,\delta)\in{\mathcal
V}(\R):\la_{\pi^{-1}m}>\la_m\},
$$
$$
{\mathcal V}^-(\R)=\{(\la,\pi,\delta)\in{\mathcal
V}(\R):\la_{\pi^{-1}m}<\la_m\},
$$
$$
{\cal V}^\pm(\R)={\cal V}^+(\R)\cup{\cal V}^-(\R).
$$
Veech \cite{veech} introduced the flow $\{P^t\}$ acting on ${\cal
V}(\R)$ by the formula
$$
P^t(\la,\pi,\delta)=(e^{t}\la,\pi,e^{-t}\delta),
$$
and the map $\U:{\cal V}^\pm(\R)\to{\cal V}(\R)$, where
$$
{\U}(\la,\pi,\delta)=\begin{cases}({{\mathcal A}}(\pi,a)^{-1}\la,a\pi,{{\mathcal A}}(\pi,
a)^{-1}\delta),
&\text{if $\la_{\pi^{-1}m}>\la_m$,}\\
({{\mathcal A}}(\pi,b)^{-1}\la,b\pi,{{\mathcal A}}(\pi,b)^{-1}\delta), &\text{if
$\la_{\pi^{-1}m}<\la_m$.}
\end{cases}
$$
(The inclusion $\U{\cal V}^\pm(\R)\subset{\cal V}(\R)$ is proved in \cite{veech}.) The
map $\U$ and the flow $\{P^t\}$ commute on ${\cal V}^\pm(\R)$ and
both preserve the measure determined on ${\cal V}(\R)$ by the volume
form $Vol=d\la_1\dots d\la_md\delta_1\dots d\delta_m$. They also
preserve the area of a zippered rectangle (see (\ref{area})) and
hence can be restricted to the set
$$
{\cal V}^{1,\pm}(\R):=\{(\la,\pi,\delta)\in{\cal V}^\pm(\R):
Area(\la,\pi,\delta)=1\}.
$$
The restriction of the volume form $Vol$ to ${\cal V}^{1,\pm}(\R)$
induces on this set a measure $\mu_\R$ which is invariant under
$\cal U$ and $\{P^t\}$.

For $(\la,\pi)\in\Delta(\R)$, denote
\begin{equation}
\label{taulambda}
\tau^0(\la,\pi)=:-\log(|\la|-\min(\la_m,\la_{\pi^{-1}m})).
\end{equation}
From (\ref{mat_a}), (\ref{mat_b}) it follows that if
$\la\in\Delta_\pi^+\cup\Delta_\pi^-$, then
\begin{equation}
\label{tau1} \tau^0(\la,\pi)=-\log|{{\mathcal A}}^{-1}(c,\pi)\la|,
\end{equation}
where $c=a$ when $\la\in\Delta_\pi^+$, and $c=b$ when
$\la\in\Delta_\pi^-$.

Next denote
$$\Y_1(\R):=\{x=(\la,\pi,\delta)\in{\cal
V}(\R):|\la|=1,\ Area(\la,\pi,\delta)=1\},
$$
\begin{equation*}
\label{tau3}\tau(x):=\tau^0(\la,\pi)\text{ for
}x=(\la,\pi,\delta)\in\Y_1(\R),
\end{equation*}
\begin{equation}
\label{phase} {\cal V}_{1,\tau}(\R):=\bigcup_{x\in\Y_1(\R),\ 0\leq
t\leq \tau(x)}P^tx.
\end{equation}
 Let
$$
{\cal V}^{1,\pm}_{\ne}(\R):=\{(\la,\pi,\delta)\in {\cal
V}^{1,\pm}(\R):a_m(\delta)\ne 0\},
$$
$$
{\cal V}_\infty(\R):=\bigcap_{n\in\mathbb Z}\U^n{\cal
V}^{1,\pm}_{\ne}(\R).
$$
Clearly
$\U^n$ is well-defined on ${\cal
V}_\infty(\R)$ for all $n\in\mathbb Z$.

We now set
$$
{\cal Y}^{\prime}(\R):={\cal Y}_1(\R)\cap\mathcal V_\infty(\R),\ \ {\tilde{\cal V}}(\R):=
{\cal V}_{1,\tau}(\R)\cap{\cal V}_\infty(\R).
$$
The above identification enables us to define on $\tilde{\cal
V}(\R)$ a natural flow, for which we retain the notation $\{P^t\}$.
(Although the bounded positive function $\tau$ is not separated from
zero, the flow $\{P^t\}$ is well defined.)

Note that for any $s\in {\mathbb R}$ we have a natural ``tautological" map
$$
{\mathfrak t}_s: M({\mathscr X})\to M(P^s{\mathscr X})
$$
which on each rectangle $\Pi_i$ is simply expansion by $e^s$ in the horizontal direction and contraction
by $e^s$ in the vertical direction. By definition, the map ${\mathfrak t}_s$ sends the vertical and the
horizontal foliations of ${\mathscr X}$ to those of $P^s{\mathscr X}$.

Introduce the space
$$
{\mathfrak X}{\tilde{\cal V}}(\R)=\{({\mathscr X}, x): {\mathscr X}\in {\tilde{\cal V}}(\R),  x\in M({\mathscr X}) \}
$$
and endow the space ${\mathfrak X}{\tilde{\cal V}}(\R)$ with the flow $P^{s, {\mathfrak X}}$ given by the formula
$$
P^{s, {\mathfrak X}}({\mathscr X}, x)=(P^s{\mathscr X}, {\mathfrak t}_sx).
$$

The flow $P^s$ induces on the transversal ${\cal Y}(\R)$ the first-return map ${\overline {\mathscr T}}$ given by the
formula
\begin{equation}
{\overline {\mathscr T}}(\la,\pi, \delta)={\cal U}P^{\tau^0(\la,\pi)}(\la,\pi, \delta).
\end{equation}

Observe that, by definition, if ${\overline {\mathscr T}}(\la,\pi,\delta)=(\la^{\prime}, \pi^{\prime}, \delta^{\prime})$,
then $(\la^{\prime}, \pi^{\prime})={\mathscr T}(\la,\pi)$.

For $(\la,\pi, \delta)\in {\tilde{\cal V}}(\R)$, $s\in {\mathbb R}$, let ${\tilde n}(\la,\pi,\delta, s)$ be defined by the formula
$$
{\cal U}^{{\tilde n}(\la,\pi,\delta, s)}(e^s\la,\pi,e^{-s}\delta)\in {\cal V}_{1,\tau}(\R).
$$

Endow the space ${\tilde{\cal V}}(\R)$  with a matrix cocycle ${\overline {\mathscr {{\mathcal A}}}}^t$ over the flow $P^s$ given by the formula
$$
{\overline {\mathcal {A}}}^t(s, (\la,\pi,\delta))={{\mathcal A}}^t({\tilde n}(\la,\pi,\delta, s), (\la,\pi)).
$$

\subsubsection{The correspondence between cocycles.}

To a connected component $\HH$ of the space $\modk$ there corresponds a unique Rauzy
class $\R$ in such a way that the following is true \cite{veech, kz}.

\begin{theorem}[\rm Veech]
\label{zipmodule} There exists a finite-to-one measurable map
$\pi_{\R}:\tilde{\cal V}(\R)\to\HH$ such that $\pi_{\R}\circ P^t=g_t\circ \pi_{\R}$.
The image of $\pi_{\R}$ contains all abelian differentials whose vertical and
horizontal foliations are both minimal.
\end{theorem}
Following Veech \cite{veechamj},
we now describe the correspondence between the cocycle ${\cal A}^t$ and the Kontsevich-Zorich cocycle
${\bf A}_{KZ}$.

As before, let ${\mathbb H}^1(\HH)$ be the fibre bundle over $\HH$ whose fibre at a point $(M, \omega)$ is the cohomology group
$H^1(M, {\mathbb R})$. The Kontsevich-Zorich cocycle ${\bf A}_{KZ}$ induces
a skew-product flow $g_s^{{\bf A}_{KZ}}$ on ${\mathbb H}^1(\HH)$ given by the formula
$$
g_s^{{\bf A}_{KZ}}({\bf X}, v)=({\bf g}_s{\bf X}, {\bf A}_{KZ}v), \ {\bf X}\in\HH, v\in H^1(M, {\mathbb R}).
$$

Following Veech \cite{veech}, we now explain the connection between the Kontsevich-Zorich cocycle
${\bf A}_{KZ}$ and the cocycle  ${\overline {\mathcal A}}^t$.

For any irreducible permutation $\pi$ Veech \cite{veechamj} defines an alternating matrix $L^{\pi}$ by
setting $L_{ij}^{\pi}=0$ if $i=j$ or if $i<j, \pi i<\pi j$, $L_{ij}^{\pi}=1$ if $i<j, \pi i>\pi j$,
$L^{\pi}_{ij}=-1$ if $i>j, \pi i<\pi j$ and denotes by $N(\pi)$ the kernel of $L^{\pi}$ and by $H(\pi)=L^{\pi}({\mathbb R}^m)$ the image of $L^{\pi}$. The dimensions of $N(\pi)$ and $H(\pi)$ do not change as $\pi$ varies in $\R$, and, furthermore,
Veech \cite{veechamj} establishes the following properties of the spaces $N(\pi)$, $H(\pi)$.
\begin{proposition}
\label{veechtokz}
Let $c=a$ or $b$. Then
\begin{enumerate}
\item $H(c\pi)={\mathcal A}^t(c,\pi)H(\pi)$,
$N(c\pi)={\mathcal A}^{-1}(c,\pi)N(\pi)$;
\item the diagram
$$
\begin{CD}
{\mathbb R}^m/N(\pi)@ >L^{\pi}>> H(\pi) \\
@ VV {\mathcal A}^{-1}(\pi, c)V          @VV {\mathcal A}^t(\pi, c) V   \\
{\mathbb R}^m/N(c\pi)@ >L^{c\pi}>> H(c\pi) \\
\end{CD}
$$
is commutative and each arrow is an isomorphism.

\item For each $\pi$ there exists a basis ${\bf v}_{\pi}$ in $N(\pi)$ such that
the map ${\mathcal A}^{-1}(\pi, c)$ sends every element of ${\bf v}_{\pi}$ to an element of ${\bf v}_{c\pi}$.
\end{enumerate}
\end{proposition}

Each space $H^{\pi}$ is thus endowed with a natural anti-symmetric bilinear form ${\mathcal L}_{\pi}$
defined, for $v_1, v_2\in H(\pi)$, by the formula
\begin{equation}
\label{bilpi}
{\mathcal L}_{\pi}(v_1, v_2)=\langle v_1, (L^{\pi})^{-1}v_2\rangle.
\end{equation}

(The vector $(L^{\pi})^{-1}v_2$ lies in ${\mathbb R}^m/N(\pi)$; since for all $v_1\in H(\pi)$, $v_2\in N(\pi)$
by definition we have $\langle v_1, v_2\rangle=0$, the right-hand side is well-defined.)

Consider the ${\mathscr T}^{{\cal A}^t}$-invariant subbundle
${\mathscr H}(\Delta(\R))\subset \Delta(\R)\times {\mathbb R}^m$ given by the formula
$$
{\mathscr H}(\Delta(\R))=\{((\la,\pi), v), (\la,\pi)\in\Delta(\R), v\in H(\pi)\}.
$$
as well as a quotient bundle
$$
{\mathscr N}(\Delta(\R))=\{((\la,\pi), v), (\la,\pi)\in\Delta(\R), v\in {\mathbb R}^m/N(\pi)\}.
$$

The bundle map ${\mathscr L}_{\R}: {\mathscr H}(\Delta(\R))\to {\mathscr N}(\Delta(\R))$ given by
${\mathscr L}_{\R}((\la,\pi), v)=((\la,\pi), L^{\pi}v)$ induces a bundle isomorphism between
${\mathscr H}(\Delta(\R))$ and ${\mathscr N}(\Delta(\R))$.

Both bundles can be naturally lifted to bundles ${\mathscr H}({\tilde {\cal V}}(\R))$, ${\mathscr N}({\tilde {\cal  V}}(\R))$
over the space ${\tilde {\cal V}}(\R)$ of zippered rectangles; they are naturally invariant
under the corresponding skew-product flows  $P^{s, \overline {\mathcal A}^t}$, $P^{s, \overline {\mathcal A}^{-1}}$,
and the map ${\mathscr L}_{\R}$ lifts to a bundle isomorphism between
${\mathscr H}({\tilde {\cal V}}(\R))$ and ${\mathscr N}({\tilde {\cal V}}(\R))$.

Take ${\mathscr X}\in {\tilde {\mathcal V}}(\R)$ and write  $\pi_{\R}({\mathscr X})=(M({\mathscr X}), \omega({\mathscr X}))$.
Veech \cite{veechamj2} has shown that the map $\pi_{\R}$ lifts to a bundle epimorphism
 ${\tilde \pi}_{\R}$
from ${\mathscr H}({\tilde {\mathcal V}}(\R))$ onto ${\mathbb H}^1(\HH)$ that intertwines
the cocycle ${\overline {\mathcal A}}^t$ and the Kontsevich-Zorich cocycle ${\bf A}_{KZ}$:

\begin{proposition}[Veech]
For almost every ${\mathscr X}\in {\tilde {\cal V}}(\R)$, ${\mathscr X}=(\la,\pi,\delta)$,
there exists an isomorphism ${\mathcal I}_{\mathscr X}: H(\pi)\to H^1(M({\mathscr X}), {\mathbb R})$
such that
\begin{enumerate}
\item
the map
${\tilde \pi}_{\R}: {\mathscr H}(\Delta(\R)) \to {\mathbb H}^1(\HH)$ given by
$$
{\tilde \pi}_{\R}({\mathscr X}, v)=(\pi_{\R}({\mathscr X}), {\mathcal I}_Xv)
$$
induces a measurable bundle epimorphism from ${\mathscr H}(\Delta(\R))$ onto ${\mathbb H}^1(\HH)$;
\item the diagram
$$
\begin{CD}
{\mathscr H}({\tilde {\mathcal V}}(\R))@ > {\tilde \pi}_{\R}  >> {\mathbb H}^1(\HH) \\
@ VVP^{s, \overline {\mathcal A}^t} V          @VV g_s^{{\bf A}_{KZ}} V   \\
 {\mathscr H}({\tilde {\mathcal V}}(\R))@ > {\tilde \pi}_{\R}  >> {\mathbb H}^1(\HH)\\
\end{CD}
$$
is commutative;
\item for ${\mathscr X}=(\la,\pi,\delta)$, the isomorphism ${\mathcal I}_X$ takes  the bilinear form ${\mathcal L}_{\pi}$ on $H(\pi)$,
defined by (\ref{bilpi}), to the cup-product on  $H^1(M({\mathscr X}), {\mathbb R})$.
\end{enumerate}
\end{proposition}

Proof: Recall that to each rectangle $\Pi_i$ Veech \cite{veechamj} assigns a cycle $\gamma_i(\la,\pi, \delta)$ in the homology
group $H_1(M, {\mathbb Z})$: if $P_i$ is the left bottom corner of $\Pi_i$ and $Q_i$ the left top corner, then the cycle is
the union of the vertical interval $P_iQ_i$ and the horizontal subinterval of $I^{(0)}(\la,\pi, \delta)$ joining $Q_i$ to $P_i$. It is
clear that the cycles $\gamma_i(\la,\pi, \delta)$ span $H_1(M, {\mathbb Z})$; furthermore, Veech shows that the cycle
$t_1\gamma_1+\dots +t_m\gamma_m$ is homologous to $0$ if and only if $(t_1, \dots, t_m)\in N(\pi)$.
We thus obtain an identification of ${\mathbb R}^m/N(\pi)$ and $H_1(M, {\mathbb R})$.
Similarly, the subspace of ${\mathbb R}^m$ spanned  by the vectors $(f(\gamma_1), \dots, f(\gamma_m))$, $f\in H^1(M, {\mathbb R})$,
is precisely $H(\pi)$.
The  identification of the bilinear form ${\mathcal L}_{\pi}$ with the cup-product is established in Proposition 4.19 in \cite{V3}.

The third statement of Proposition \ref{veechtokz} has the following important
\begin{corollary}
\label{hypzip}
Let ${\Prob}_{\cal V}$ be an ergodic $P^s$-invariant probability measure for the flow $P^s$ on ${\cal V}(\R)$
and let $\Prob_{\HH}=(\pi_{\R})_*{\Prob}_{\cal V}$ be the corresponding ${\bf g}_s$-invariant measure on $\HH$.
If the Kontsevich-Zorich cocycle acts isometrically on its neutral subspace with respect to $\Prob_{\HH}$, then the
cocycle ${\overline {\cal A}}^t$  also acts isometrically on its neutral subspace with respect to $\Prob_{\cal V}$.
\end{corollary}

Note that the hypothesis of Corollary \ref{hypzip} is satisifed, in particular, for the Masur-Veech
smooth measure on the moduli space of abelian differentials.

\subsection{Zippered rectangles and Markov Compacta.}
\subsubsection{The main lemmas}

We now present the precise connection between Veech's zippered rectangles and Markov compacta.

Given a finite set ${\mathfrak G}_{0}\subset {\mathfrak G}$, denote
$$
\Omega_{{\mathfrak G}_0}=\{\omega\in\Omega: \omega_n\in {\mathfrak G}_{0}, n\in {\mathbb Z}\};
$$
$$
{\overline \Omega}_{{\mathfrak G}_0}=\{\oomega=(\omega, r): \omega\in\Omega_{{\mathfrak G}_{0}}\}.
$$

Our first lemma gives a map from the Veech space of zippered rectangles to the space $\Oomega$ of measured Markov compacta.
\begin{lemma}
\label{zipmarkovlemmaone}
Let $\R$ be a Rauzy class of irreducible permutations.
There exists a finite set ${\mathfrak G}_{\R}\subset {\mathfrak G}$
and a Vershik's ordering $\oo$ as well as a reverse Vershik's ordering ${\tilde \oo}$ on each $\Gamma\in{\mathfrak G}_{\R}$
such that the following is true.
There exists a map
$$
{\overline {\mathfrak Z}}_{\R}: {\tilde {\cal V}}_{ue}(\R)\to {\overline \Omega}_{{\mathfrak G}_{\R}},
$$
such that the diagram
$$
\begin{CD}
{\tilde {\cal V}}_{ue}(\R)@ > {\overline {\mathfrak Z}}_{\R}  >> {\overline \Omega} \\
@ VVP^s V          @VV g_s V   \\
{\tilde {\cal V}}_{ue}(\R)@ > {\overline {\mathfrak Z}}_{\R}  >> {\overline \Omega} \\
\end{CD}
$$
is commutative and if $\Prob_{{\cal V}}$ is an ergodic probability $P^s$-invariant measure, then we have
\begin{equation}
\label{zrpinu}
({\overline {\mathfrak Z}}_{\R})_*\Prob_{{\cal V}}\in {\mathscr P}^+.
\end{equation}
\end{lemma}

Our second lemma establishes a correspondence between a zippered rectangle ${\mathscr X}$ and the corresponding Markov compactum ${\overline {\mathfrak Z}}_{\R}({\mathscr X})$.

\begin{lemma}
\label{zipmarkovlemmatwo}
For any ${\mathscr X}\in {\tilde {\cal V}}_{ue}(\R)$, an
${\bf m}_{{\mathscr X}}$-almost surely defined map
$$
{\overline {\mathscr J}}_{\mathscr X}: M({\mathscr X})\to X(\oomega_{{\mathscr X}})
$$
such that the  map
$$
{\overline {\mathfrak Z}}^{{\mathfrak X}}_{\R}: {\mathfrak X}{\tilde {\cal V}}_{ue}(\R)\to {\mathfrak X}{\overline \Omega}
$$
is given by the formula
$$
{\overline {\mathfrak Z}}^{{\mathfrak X}}_{\R}({\mathscr X}, x)=({\overline {\mathfrak Z}}_{\R}{\mathscr X}, {\overline {\mathscr J}}_{\mathscr X}x),
$$
then the diagram
$$
\begin{CD}
 {\mathfrak X}{\tilde {\cal V}}_{ue}(\R)  @ > {\overline {\mathfrak Z}}_{\R}^{{\mathfrak X}}  >> {\mathfrak X}{\overline \Omega} \\
@ VVP^{s, {\mathfrak X}} V          @VV g_s^{{\mathfrak X}} V   \\
{\mathfrak X}{\tilde {\cal V}}_{ue}(\R)  @ > {\overline {\mathfrak Z}}_{\R}^{{\mathfrak X}}  >> {\mathfrak X}{\overline \Omega} \\
\end{CD}
$$
is commutative. The map ${\overline {\mathscr J}}_{\mathscr X}$ sends the vertical flow $h_t^+$ on ${\mathscr X}$ to the flow
$h_t^{+,\oo}$ on $X({\overline {\mathfrak Z}}_{\R}({\mathscr X}))$;   the horizontal flow $h_t^-$ on ${\mathscr X}$ to the flow
$h_t^{-,{\tilde \oo}}$ on $X({\overline {\mathfrak Z}}_{\R}({\mathscr X}))$ and induces isomorphisms between the space $\B^+_{{\mathscr X}}$ and the space
$\B^+_{({\overline {\mathfrak Z}}_{\R}({\mathscr X}))}$; the space $\B^-_{{\mathscr X}}$ and the space
$\B^-_{({\overline {\mathfrak Z}}_{\R}({\mathscr X}))}$.
\end{lemma}

Our third lemma shows that the map ${\overline {\mathfrak Z}}_{\R}$ intertwines the Kontsevich-Zorich cocycle and the  renormalization cocycle.
\begin{lemma}
\label{zipmarkovlemmathree}
If the map
$$
{\overline {\mathfrak Z}}^{\prime}_{\R}: {\tilde {\cal V}}_{ue}(\R)\times {\mathbb R}^m\to {\overline \Omega}\times {\mathbb R}^m
$$
is given by the formula
$$
{\overline {\mathfrak Z}}^{\prime}_{\R}({\mathscr X}, v)=({\overline {\mathfrak Z}}_{\R}{\mathscr X}, v),
$$
then  the diagram
$$
\begin{CD}
{\tilde {\cal V}}_{ue}(\R)\times {\mathbb R}^m  @ > {\overline {\mathfrak Z}}_{\R}^{{\prime}}  >> {\overline \Omega}\times {\mathbb R}^m \\
@ VVP^{s, {\mathcal A}^t} V          @VV g_s^{{\mathbb A}} V   \\
{\tilde {\cal V}}_{ue}(\R)\times {\mathbb R}^m  @ > {\overline {\mathfrak Z}}_{\R}^{{\prime}}  >> {\overline \Omega}\times {\mathbb R}^m \\
\end{CD}
$$
is commutative.
\end{lemma}

Informally, the map ${\overline {\mathfrak Z}}_{\R}$ is constructed as follows.
First, to a zippered rectangle one assigns its {\it Rauzy-Veech expansion},
the bi-infinite sequence of pairs $(\pi, c)$, $\pi\in\R$, $c=a$ or $b$.
To each pair $(\pi, c)$  we have assigned a unimodular matrix ${\cal A}(\pi, c)$.
Now to each such matrix we assign a graph $\Gamma$ in ${\mathfrak G}$.
The resulting sequence of graphs yields the desired Markov compactum.

{\bf Remark.} In \cite{veechmmj}  Veech has shown that for a minimal interval exchange $(\la,\pi)$
the sequence of its Rauzy-Veech renormalization matrices ${\cal A}({\mathscr T}^n(\la,\pi))$ uniquely determines
the permutation $\pi$. In particular, if $(\la,\pi)$ is uniquely ergodic, then the sequence of Rauzy-Veech
renormalization
matrices uniquely determines the interval exchange transformation.
The result of Veech implies that the map ${\overline {\mathfrak Z}}_{\R}$ is in fact {\it injective}.

\subsubsection {Rauzy-Veech expansions of zippered rectangles.}
Given a Rauzy class $\mathcal R$ of irreducible permutations, introduce an alphabet
$$
{{\mathfrak A}}_{\mathcal R}=\{ (\pi,c), c= a \text{ or } b,\pi\in\mathcal R\}.
$$

To each letter ${\bf p}_1\in {\mathfrak A}_{\mathcal R}$ assign a set
$\Delta_{{\bf p}_1}\subset \Delta(\R)$ given by
$$
\Delta_{{\bf p}_1}=\Delta_{\pi}^+ \ {\rm if}\ {\bf p}_1=(\pi,a); \ \Delta_{{\bf p}_1}=\Delta_{\pi}^- \ {\rm if}\ {\bf p}_1=(\pi,b).
$$

Take a zippered rectangle ${\mathscr X}\in {\cal Y}^{\prime}(\R)$, ${\mathscr X}=(\la, \pi, \delta)$,
$|\la|=1$.
For $n\in {\mathbb Z}$ write  ${\overline {\mathscr T}}{\mathscr X}=(\la^{(n)}, \pi^{(n)}, \delta^{(n)})$
and assign to ${\mathscr X}$ a sequence ${\bf p}_n({\mathscr X})_{n\in {\mathbb Z}}$ given by
\begin{equation}
(\la^{(n)}, \pi^{(n)})\in \Delta_{{\bf p}_n}.
\end{equation}

The sequence ${\bf p}_n({\mathscr X})_{n\in {\mathbb Z}}$ is the {\it  Rauzy-Veech expansion} of the zippered rectangle ${\mathscr X}$.

\subsubsection{A Markov compactum corresponding to a zippered rectangle.}

To each letter ${\bf p}_1\in {\mathcal A}_{\mathcal R}$, we assign
an oriented graph $\Gamma({\bf p}_1)$ on $m$ vertices in the following way.

\begin{paragraph}{Case 1.} Assume ${\bf p}_1=(\pi,a)$. Then the graph $\Gamma({\bf p}_1)$ has $m+1$ edges
$$
e_{ii},\  1\leq i\leq \pi^{-1}m; \
 e_{\pi^{-1}m+1,m};\
 e_{i,i-1},\  \pi^{-1}m+1\leq i\leq m.
 $$
We set $I(e_{ij})=i, F(e_{ij})=j$.
\end{paragraph}
\begin{paragraph}{Case 2.} $\bf p=(\pi,b)$. The graph $\Gamma(\pi, b)$ has $m+1$ edges
$$e_{ii}, i=1,\ldots, m;  \ e_{\pi^{-1}m,m}.$$
\end{paragraph}
Again we set $I(e_{ij})=i, F(e_{ij})=j$.

The incidence matrix of a graph $\Gamma({\bf p}_1)$, ${\bf p}_1=(\pi,c)$, is the transpose of the Rauzy
matrix assigned to $(\pi,c)$.

A canonical Vershik's ordering on the graphs $\Gamma({\bf p}_1)$ is given by the rule: $e_{ij}<e_{ik}$ if and only if $j<k$.

A word $\mathbf p$ in the alphabet  $\mathfrak A_{\mathcal R}$,
$$
\mathbf p=(\mathbf p_1,\ldots,\mathbf p_l),\  \mathbf p_i=(\pi_i,c_i),
$$
will be called {\it admissible} if $\pi_{i+1}=c_i\pi_i$. The set of all admissible words will
be denoted ${\mathscr W}_{\mathcal R}$. Similarly, an infinite sequence will be called admissible if its every finite
subsequence is admissible. The set of all admissible bi-infinite sequences will be denoted ${\Sigma}_{\R}$.
We have a natural map $Gr_{\R}: \Sigma_{\R}\to\Omega$ given by
$$
Gr_{\R}: ({\bf p}_n)_{n\in {\mathbb Z}}\to \Gamma(({\bf p}_n))_{n\in {\mathbb Z}}.
$$

There is a natural map $Code^+_{\R}:\Delta(\R) \to{\Sigma}^+_{\R}$ which sends $(\la,\pi)$
to a sequence ${\bf p}_n$, $n\in {\mathbb N}$
given by
$$
{\mathscr T}^n(\la,\pi)\in \Delta_{{\bf p}_n}.
$$

This map is extended to a map $Code_{\R}:{\cal Y}^{\prime}(\R) \to{\Sigma}_{\R}$ which sends $(\la,\pi, \delta)$
to a sequence ${\bf p}_n$, $n\in {\mathbb Z}$
given by
$$
{\overline {\mathscr T}}^n(\la,\pi, \delta)=(\la^{(n)},\pi^{(n)}, \delta^{(n)}), (\la^{(n)}, \pi^{(n)})\in \Delta_{{\bf p}_n}.
$$

We thus obtain the desired composition map:
$$
{\mathfrak Z}_{\R}=Gr_{\R}\circ Code_{\R}:{\cal Y}^{\prime}(\R)\to \Omega.
$$

\subsubsection{Properties of the symbolic coding.}

We have thus constructed a measurable coding mapping
$
{\mathfrak Z}_{\R}: {\cal Y}^{\prime}(\R)\to \Omega.
$
The diagram
$$
\begin{CD}
{\mathcal Y^{\prime}}(\R)@ > {\mathfrak Z}_{\R}  >> \Omega \\
@ VV{\overline {\mathscr T}} V          @VV \sigma V   \\
{\mathcal Y^{\prime}}(\R)@ > {\mathfrak Z}_{\R}  >> \Omega \\
\end{CD}
$$
is commutative by construction.

For a zippered rectangle ${\mathscr X}\in {\cal Y}^{\prime}(\R)$, consider the corresponding
abelian differential ${\bf X}=\pi_{\R}({\mathscr X})$ with underlying surface $M({\mathscr X}$
and let ${\bf m}_{{\mathscr X}}$ be the Lebesgue measure on $M({\mathscr X})$.
Write $\omega_{{\mathscr X}}={\mathfrak Z}_{\R}({\mathscr X})$.
We then have a ``tautological'' coding  mapping from  the Markov compactum  $X(\omega_{{\mathscr X}})$ to $M({\mathscr X})$. The foliaitons $\F^+_{X(\omega_{{\mathscr X}})}$ and $\F^-_{X(\omega_{{\mathscr X}})}$ are taken, respectively, to the vertical and the horizontal foliations on  $M({\mathscr X})$; unique ergodicity of the Markov compactum $X(\omega_{{\mathscr X}})$ is equivalent to the unique ergodicity of both the horizontal and the vertical flows on $M({\mathscr X})$.

Now assume that the Markov compactum $X(\omega_{{\mathscr X}})$ is indeed uniquely ergodic.  Then
the coding mapping is $\nu_{\omega_{{\mathscr X}}}$-almost surely invertible, and we obtain a  ${\bf m}_{{\mathscr X}}$-almost surely defined map
$$
{\mathscr J}_{\mathscr X}: M({\mathscr X})\to X(\omega_{{\mathscr X}}).
$$

Recall that ${\bf X}=\pi_{\R}({\mathscr X})$.
By definition, the mapping ${\mathscr J}_{\mathscr X}$ induces a linear isomorphism
between the space $\BB^+_{{\bf X}}$ and the space $\BB^+_{X(\omega_{{\mathscr X}})}$;
and similarly between the space $\BB^-_{{\bf X}}$ and the space $\BB^-_{X(\omega_{{\mathscr X}})}$.
We have $\left({\mathscr J}_{\mathscr X}\right)_*{\bf m}_{{\mathscr X}}=\nu_{\omega_{{\mathscr X}}}$.
The mapping ${\mathscr J}_{\mathscr X}$ takes the space of weakly Lipschitz functions on $M({\mathscr X})$
to the space of weakly Lipschitz functions on $X(\omega_{{\mathscr X}})$.

The map ${\mathfrak Z}_{\R}$ lifts to a natural map
$$
{\overline {\mathfrak Z}}_{\R}: {\tilde {\cal V}}_{ue}(\R)\to {\overline \Omega},
$$
and, again, the diagram
$$
\begin{CD}
{\tilde {\cal V}}_{ue}(\R)@ > {\overline {\mathfrak Z}}_{\R}  >> {\overline \Omega} \\
@ VVP^s V          @VV g_s V   \\
{\tilde {\cal V}}_{ue}(\R)@ > {\overline {\mathfrak Z}}_{\R}  >> {\overline \Omega} \\
\end{CD}
$$
is commutative.

If $\Prob_{{\cal V}}$ is an ergodic probability $P^s$-invariant measure, then we have
\begin{equation}
\label{zrpinpplus}
({\overline {\mathfrak Z}}_{\R})_*\Prob_{{\cal V}}\in {\mathscr P}^+.
\end{equation}
Indeed, (\ref{zrpinu}) is a reformulation of a Lemma due to Veech \cite{veech} which states that
every finite $P^s$-invariant measure assigns positive probability to a Rauzy matrix
with positive entries.

For any ${\mathscr X}\in {\tilde {\cal V}}_{ue}(\R)$ we again obtain
${\bf m}_{{\mathscr X}}$-almost surely defined map
$$
{\overline {\mathscr J}}_{\mathscr X}: M({\mathscr X})\to X(\oomega_{{\mathscr X}}).
$$

Introduce a map
$$
{\overline {\mathfrak Z}}^{{\mathfrak X}}_{\R}: {\mathfrak X}{\tilde {\cal V}}_{ue}(\R)\to {\mathfrak X}{\overline \Omega}
$$
by the formula
$$
{\overline {\mathfrak Z}}^{{\mathfrak X}}_{\R}({\mathscr X}, x)=({\overline {\mathfrak Z}}_{\R}{\mathscr X}, {\overline {\mathscr J}}_{\mathscr X}x).
$$
The diagram
$$
\begin{CD}
 {\mathfrak X}{\tilde {\cal V}}_{ue}(\R)  @ > {\overline {\mathfrak Z}}_{\R}^{{\mathfrak X}}  >> {\mathfrak X}{\overline \Omega} \\
@ VVP^{s, {\mathfrak X}} V          @VV g_s^{{\mathfrak X}} V   \\
{\mathfrak X}{\tilde {\cal V}}_{ue}(\R)  @ > {\overline {\mathfrak Z}}_{\R}^{{\mathfrak X}}  >> {\mathfrak X}{\overline \Omega} \\
\end{CD}
$$
is commutative.

The map ${\overline {\mathfrak Z}}_{\R}$  intertwines the cocycles ${\cal A}^t$ and ${\mathbb A}$ in the following sense.
Take ${\mathscr X}\in {\tilde {\cal V}}_{ue}(\R)$, $v\in {\mathbb R}^m$ and write
$$
{\overline {\mathfrak Z}}^{\prime}_{\R}({\mathscr X}, v)=({\overline {\mathfrak Z}}_{\R}{\mathscr X}, v).
$$
The resulting map
$$
{\overline {\mathfrak Z}}^{\prime}_{\R}: {\tilde {\cal V}}_{ue}(\R)\times {\mathbb R}^m\to {\overline \Omega}\times {\mathbb R}^m
$$
intertwines the cocycles ${\cal A}^t$ and ${\mathbb A}$: indeed, by definition, the diagram
$$
\begin{CD}
{\tilde {\cal V}}_{ue}(\R)\times {\mathbb R}^m  @ > {\overline {\mathfrak Z}}_{\R}^{{\prime}}  >> {\overline \Omega}\times {\mathbb R}^m \\
@ VVP^{s, {\mathcal A}^t} V          @VV g_s^{{\mathbb A}} V   \\
{\tilde {\cal V}}_{ue}(\R)\times {\mathbb R}^m  @ > {\overline {\mathfrak Z}}_{\R}^{{\prime}}  >> {\overline \Omega}\times {\mathbb R}^m \\
\end{CD}
$$
is commutative.

Denote ${\mathfrak G}_{\R}=\{\Gamma({\mathbf p}_1), {\bf p}_1\in {\mathfrak A}_{\R}\}$ and set
$$
\Omega_{{\mathfrak G}_{\R}}=\{\omega\in\Omega: \omega_n\in {\mathfrak G}_{\R}, n\in {\mathbb Z}\};
$$
$$
{\overline \Omega}_{{\mathfrak G}_{\R}}=\{\oomega=(\omega, r): \omega\in\Omega_{{\mathfrak G}_{\R}}\}.
$$

By construction, ${\mathfrak Z}_{\R}({\cal Y}^{\prime}_{\R})\subset \Omega_{{\mathfrak G}_{\R}}$.
Every graph $\Gamma\in {\mathfrak G}_{\R}$ is endowed with a Vershik's ordering constructed in the previous subsection,
and we obtain a $\sigma$-equivariant Vershik's ordering ${\mathfrak o}_{\R}$ on ${\mathfrak Z}_{\R}({\cal Y}^{\prime}_{\R})$.

By definition, the mapping ${\overline {\mathscr J}}_{\mathscr X}$ sends the vertical flow $h_t^+$ on $M({\mathscr X})$
to the flow $h_t^{+, \omega_{{\mathscr X}}}$ and the horizontal flow $h_t^-$ on $M({\mathscr X})$ to the flow
$h_t^{-, \omega_{{\mathscr X}}}$.

Lemmas \ref{zipmarkovlemmaone}, \ref{zipmarkovlemmatwo}, \ref{zipmarkovlemmathree} are proved.
Theorems \ref{multiplicmoduli}, \ref{limthmmoduli} follow now
from their symbolic counterparts, Corollary \ref{symbmultiplic} and Theorem \ref{limthmmarkcomp}.

Theorems \ref{multiplicmoduli}, \ref{limthmmoduli} are proved completely.

\section{Appendix A: On the Oseledets Multiplicative Ergodic Theorem.}

\subsection{The Oseledets-Pesin Reduction Theorem.}
For the reader's convenience, we recall here several refinements of the Oseledets Multiplicative Ergodic Theorem.
The first statement we need is an immediate corollary of the Oseledets-Pesin Reduction Theorem \cite{pesinbarreira}.
We restrict ourselves to discrete time and  to the invertible case; the case of flows is completely similar.

Let $(Y,\mathcal{B},\mu)$ be a probability space, let $T:\:Y\rightarrow Y$ be an invertible $\mu$-preserving transformation, let $m\in\mathbb{N}$ and let
$$\mathbb{A}:\:Y\longrightarrow GL(m,\mathbb{R})$$
be a measurable map.

Denote
$$\mathbb{A}(n,y)=\begin{cases}\mathbb{A}\left(T^{n-1}y\right)\cdot\ldots\cdot\mathbb{A}(y), & n>0; \\ Id, & n=0; \\ \mathbb{A}^{-1}\!\left(T^{-n}y\right)\cdot\ldots\cdot\mathbb{A}^{-1}\!\left(T^{-1} y\right), & n<0.\end{cases}$$

The family of maps $\mathbb{A}(n,y),\:n\in\mathbb{Z}$, is called a {\it measurable cocycle} over $T$. Together with the cocycle $\mathbb{A}$ over the automorphism $T$, we consider the transpose cocycle $\mathbb{A}^t$ over the transformation $T^{-1}$ which is defined by the formula:
$$\mathbb{A}^t(n,y)=\begin{cases}\mathbb{A}^t\!\left(T^{1-n}y\right)\cdot\ldots\cdot\mathbb{A}^t(y), & n>0;\\ Id, & n=0;\\ (\mathbb{A}^t)^{-1}\!\left(T^{-n}y\right)\cdot\ldots\cdot(\mathbb{A}^t)^{-1} \left(T^{-1}y\right) ,& n<0.\end{cases}$$

We also let $\|A\|$ stand for the usual Euclidean norm of the matrix $A$, and, as above, for $v\in\mathbb{R}^m,\,v=(v_1,\ldots,v_m)$, we set
$$|v|=\sum\limits^m_{i=1}|v_i|\,.$$

\begin{theorem}
Let $(Y,\mathcal{B},\mu)$ be a probability space, let $T:\:Y\rightarrow Y$ be an invertible ergodic $\mu$-preserving transformation, let $m\in\mathbb{N}$ and let $\mathbb{A}:\:Y\longrightarrow GL(m,\mathbb{R})$\; be a measurable map such that both functions $\log\left(1+\|A(y)\|\right)$,\: $\log\left(1+\|A^{-1}(y)\|\right)$ belong to the class $L_1(Y,\mu)$. Then there exist numbers $\theta_1>\theta_2>\ldots>\theta_r$\, and, for $\mu$-almost every $y\in Y$, direct-sum decompositions

\begin{equation} \label{osdec}
\mathbb{R}^m=E^1_y\oplus\ldots\oplus E^r_y
\end{equation}

\begin{equation} \label{osdectran}
\mathbb{R}^m=\widetilde{E}^1_y\oplus\ldots\oplus \widetilde{E}^r_y
\end{equation}

that depend measurably on $y\in Y$ and satisfy the following.

\quad1) for $\mu$-almost any $y\in Y$, $n\in {\mathbb Z}$ and any $i=1,\ldots,r$, we have
$$\mathbb{A}(n,y)E^i_y=E^i_{T^ny};$$
$$\mathbb{A}^t(n,y){\tilde E}^i_y={\tilde E}^i_{T^{-n}y}.$$

\quad2) for any $v\in E^i_y$, $v\neq0$, we have
$$\lim\limits_{|n|\rightarrow\infty}\frac{\log|\mathbb{A}(n,y)v|}{n}=\theta_i\,,$$
and the convergence is uniform on the sphere $\left\{v\in E^i_y,\,|v|=1\right\}$.

\quad3) for any $v\in\widetilde{E}^i_y$, $v\neq0$, we have
$$\lim\limits_{|n|\rightarrow\infty}\frac{\log|\mathbb{A}^t(n,y)v|}{n}=\theta_i\,,$$
and the convergence is uniform on the sphere $\left\{v\in \widetilde{E}^i_y,\,|v|=1\right\}$.

\quad4) For any $\varepsilon>0$ there exist positive measurable functions
$$C^{(i)}_{\varepsilon, 1},\,C^{(i)}_{\varepsilon, 2},\,\widetilde{C}^{(i)}_{\varepsilon, 1},\, \widetilde{C}^{(i)}_{\varepsilon, 2}\,:\:Y\longrightarrow\mathbb{R}_{>0}$$
such that for $\mu$-almost every $y\in Y$ the inequalities

$$
C^{(i)}_{\varepsilon, 1}(y) e^{(\theta_i-\varepsilon)(n-|k|)}\leqslant\left\|\left.\mathbb{A}(n,T^ky) \right|_{E^i_{T^ky}}\right\|\leqslant C^{(i)}_{\varepsilon, 2}(y) e^{(\theta_i+\varepsilon)(n+|k|)}$$
$$\widetilde{C}^{(i)}_{\varepsilon, 1}(y)e^{(\theta_i-\varepsilon)(n-|k|)}\leqslant \left\|\left.\mathbb{A}^t(n,T^ky) \right|_{\widetilde{E}^i_{T^ky}}\right\|\leqslant \widetilde{C}^{(i)}_{\varepsilon, 2}(y) e^{(\theta_i+\varepsilon)(n+|k|)}$$
hold for all $k\in\mathbb{Z}$, $n\in\mathbb{N}$.

\quad5) $$\mathrm{dim} E^i_y = \mathrm{dim} \widetilde{E}^i_y\,,\;i=1,\ldots,r\,,$$

and if $i\neq j$, $v\in E^i_y$, $\widetilde{v}\in\widetilde{E}^j_y$, then
$$\langle v,\widetilde{v}\rangle=0.$$
\end{theorem}

This is immediate from the Oseledets-Pesin Reduction Theorem, see Theorem 3.5.5 on p.77 in \cite{pesinbarreira}.

\subsection{ Viana's Lemma  on the Simplicity of the Top Lyapunov Exponent.}

\begin{lemma}[Viana]
\label{vianalemma}
Let $(Y,\mathcal{B},\mu)$ be a probability space, let $T:\:Y\rightarrow Y$ be an invertible ergodic $\mu$-preserving transformation, let $m\in\mathbb{N}$ and let $\mathbb{A}:\:Y\longrightarrow GL(m,\mathbb{R})$\; be a measurable map such that both functions $\log\left(1+\|A(y)\|\right)$,\: $\log\left(1+\|A^{-1}(y)\|\right)$ belong to the class $L_1(Y,\mu)$.

Assume that for $\mu$-almost all $y\in Y$ all entries of the matrices ${\mathbb A}$
are positive, and, moreover, that
there exists a constant $C>0$ such that for $\mu$-almost all $y\in Y$ and all
$i,j,k\in\{1,\dots,m\}$ we have
$$
C^{-1}<\frac{{\mathbb A}_{ij}(y)}{{\mathbb A}_{{kj}}(y)}<C.
$$
Then the top Lyapunov exponent of the cocycle ${\mathbb A}$ is positive and simple.
\end{lemma}

This is a reformulation of Lemma 5.7 in M. Viana \cite{Viana-IMPA}.
Note that our cocycle ${\mathbb A}$ is the inverse of the cocycle considered by M. Viana.
Lemma \ref{vianalemma} immediately implies
\begin{corollary}
\begin{enumerate}
\item Let $\mu$ be a $\sigma$-invariant ergodic probability measure on the space $\Omega$ of Markov compacta.
If $\mu$ satisfies Assumption \ref{asos}, then the top Lyapunov exponent of the renormalization cocycle ${\mathbb A}$ with respect to $\mu$ is positive and simple.
\item Let $\Prob$ be an ergodic probability $g_s$-invariant measure on the space
${\overline \Omega}$ of measured Markov compacta. If $\Prob\in {\mathscr P}^+$, then  the top Lyapunov exponent of
the renormalization cocycle ${\overline {\mathbb A}}$ with respect to $\Prob$ is positive and simple.
\end{enumerate}
\end{corollary}

\section{Appendix B: Metrics on the Space of Probability  Measures.}
\subsection{The Weak Topology.}

In this Appendix, we collect some standard facts about the weak
topology on the space of probability measures. For a detailed
treatment, see, e.g., \cite{bogachev}.

Let $(\emph{X}, {d})$ be a complete separable metric
space, and let  $\mathfrak{M}(\emph{X})$ be the space of
Borel probability measures on \emph{X}. The {\it weak topology}
on $\mathfrak{M}(\emph{X})$ is defined as follows. Let
$\varepsilon>0$, $\nu_0\in \mathfrak{M}(\emph{X})$, and let
$f_1,..,f_k: \emph{X} \rightarrow \mathbb{R} $ be bounded
continuous functions. Introduce the set
$$U(\nu_0, \varepsilon,f_1,..,f_k)=\{\nu\in\mathfrak{M}(\emph{X}):
 |  \int\limits_X f_i d\nu-\int\limits_X f_i d\nu_0|<\varepsilon,
 i=1,..,k\}.$$

The basis of neighbourhoods for the weak topology is given
precisely by sets of the form $U(\nu_0, \varepsilon,f_1,..,f_k)$, for
all $\varepsilon>0,$ $\nu_0\in \mathfrak{M}(\emph{X}),$
$f_1,..,f_k$ continuous and bounded.

The weak topology is metrizable and there are several natural
metrics on $\mathfrak{M}(\emph{X})$ inducing the weak topology.

\subsection{The Kantorovich-Rubinstein metric}

Let $$Lip_1^1=\{f:\emph{X}\rightarrow \mathbb{R}\ \ :\
\sup_X|f|\leqslant 1,\  |f(x_1)-f(x_2)|\leqslant d(x_1,x_2)\  \mathrm{for}\
\mathrm{all} \ x_1,x_2\in\emph{X}\}.$$

The Kantorovich-Rubinstein metric is defined, for
$\nu_1,\nu_2\in\mathfrak{M}(\emph{X}),$ by the formula
$$d_{KR}(\nu_1,\nu_2)=\sup_{f\in Lip_1^1(\emph{X})} |
\int\limits_{\emph{X}} f d\nu_1-\int\limits_{\emph{X}} f
d\nu_2|.$$

The Kantorovich-Rubinstein metric induces the weak topology on
$\mathfrak{M}(\emph{X}).$ By the Kantorovich-Rubinstein
Theorem, for bounded metric spaces, the Kantorovich-Rubinstein metric admits the following
equivalent dual description. Given
$\nu_1,\nu_2\in\mathfrak{M}(\emph{X}),$ let $\rm
Join(\nu_1,\nu_2)\in \mathfrak{M}(\emph{X}\times\emph{X})$ be the
set of probability measures $\eta$ on $\emph{X}\times\emph{X}$
such that projection of $\eta$ on the first coordinate is equal to
$\nu_1,$ the projection of $\eta$ on the second coordinate is
equal to $\nu_2.$ The Kantorovich-Rubinstein Theorem claims that
$$d_{KR}(\nu_1,\nu_2)=\inf_{\eta\in\rm
Join(\nu_1,\nu_2)}\int\limits_{\emph{X}\times\emph{X}}d(x_1,x_2)d\eta.$$

\subsection {The L{\'e}vy-Prohorov metric.}

Let $\mathcal{B}_\emph{X}$ be the $\sigma$-algebra of Borel
subsets of $\emph{X}.$ For $B\in\mathcal{B}_\emph{X},
\varepsilon>0,$ set $$B^\varepsilon=\{x\in\emph{X}:\ \inf_{y\in
B}d(x,y)\leqslant\varepsilon\}.$$

Given $\nu_1,\nu_2\in \mathfrak{M}(\emph{X})$, introduce the
L{\'e}vy-Prohorov distance between them by the formula
$$d_{LP}(\nu_1,\nu_2)=\inf\{\varepsilon>0: \nu_1(B)\leqslant\nu_2(B^\varepsilon)+\varepsilon,
\nu_2(B)\leqslant\nu_1(B^\varepsilon)+\varepsilon \ \mathrm{for}\ \mathrm{any}\ B\in\mathcal{B}\}.$$

The L{\'e}vy-Prohorov metric also induces the weak topology on
$\mathfrak {M}(\emph{X}).$

\subsection{ An estimate on the distance between images of
measures.}

Let $(\Omega,\mathfrak{B}_\Omega,\mathbb{P})$ be a probability
space, and let $\xi_1,\xi_2:\Omega\rightarrow\emph{X}$ be two
measurable maps.

In the proof of the limit theorems, we use the following
simple estimate on the L{\'e}vy-Prohorov and the Kantorovich-Rubinstein
distance between the push-forwards
$(\xi_1)_\ast\mathbb{P},(\xi_2)_\ast\mathbb{P}$ of the measure
$\mathbb{P}$ under the mappings $\xi_1,\xi_2.$

\begin{lemma}\label{dist-images} Let $\varepsilon>0$ and assume
that for $\mathbb{P}-almost$ all $\omega\in\Omega$ we have
$d(\xi_1(\omega),\xi_2(\omega))\leqslant\varepsilon.$

Then we have
$$d_{KR}((\xi_1)_\ast\mathbb{P},(\xi_2)_\ast\mathbb{P})\leqslant\varepsilon,$$

$$d_{LP}((\xi_1)_\ast\mathbb{P},(\xi_2)_\ast\mathbb{P})\leqslant\varepsilon.$$
\end{lemma}

The proof of the lemma is immediate from the definitions of the
Kantorovich-Rubinstein and the L{\'e}vy-Prohorov metric.

\section{Appendix C: Correspondence Between the Symbolic and the Geometric Language}
\begin{center}
\begin{tabular}{|p{5.2cm}|p{5.2cm}|}
\hline
\begin{center}\textbf{Symbolic language}\end{center}&\begin{center}\textbf{Geometric language}\end{center}\\
\hline
Markov Compactum {X}&Abelian Differential {\bf X}\\
\hline
Asymptotic foliation $\mathcal{F}^+$\qquad

corresponding to the future
&
Vertical foliation $\mathcal{F}^+$\\
\hline
Asymptotic foliation $\mathcal{F}^-$\qquad

corresponding to the past
&
Horizontal foliation $\mathcal{F}^-$\\
\hline
The spaces $\mathfrak{B}^+$ and $\mathfrak{B}^-$ of finitely-additive H\"{o}lder measures&The spaces $\mathfrak{B}^+$ and $\mathfrak{B}^-$ of H\"{o}lder cocycles\\
\hline
Finitely-additive measures $m_{\Phi^-}$, $\Phi^-\in\mathfrak{B}^-$&Forni's Invariant Distributions of the Vertical Flow\\
\hline
Vershik's Automorphisms&Interval Exchange Transformations\\
\hline
Symbolic Flows $h_t^+$ &Translation Flows $h_t^+$\\
\hline
Adjacency Matrices $A(\Gamma_n)$ &Rauzy-Veech Matrices $A(\pi,c)$\\
\hline

The space  $\Oomega$ of measured Markov compacta&The moduli space $\HH$ of abelian differentials\\
\hline
The renormalization flow $g_s$ &The Teichm\"{u}ller flow ${\bf g}_s$\\
\hline
The renormalization cocycle ${\overline {\mathbb A}}$ &The Kontsevich-Zorich cocycle ${\mathbf {A}}$\\
\hline
Products of measures and duality between ${\mathfrak B}^+$, ${\mathfrak B}^-$& The Kolmogorov-Alexander product and the Poincar{\'e} duality in cohomology \\
\hline
\end{tabular}
\end{center}

\medskip

\end{document}